\documentclass[11pt, a4paper]{article} 
\title{Global existence in the Lipschitz class for the N-Peskin problem} 
\author{
	Francisco Gancedo
	\\{\footnotesize Departamento de An\'alisis Matemático \& IMUS}
	\\{\footnotesize Universidad de Sevilla}
	\\{\footnotesize Sevilla, Espa\~na}
	\\{\footnotesize email: {\it fgancedo@us.es}}
	\and 
	Rafael Granero-Belinch\'on
	\\{\footnotesize Departamento de Matem\'aticas, Estad\'istica y Computaci\'on}
	\\{\footnotesize Universidad de Cantabria}
	\\{\footnotesize Santander, Espa\~na}
	\\{\footnotesize email: {\it rafael.granero@unican.es}}
	\and 
	Stefano Scrobogna
	\\{\footnotesize Departamento de An\'alisis Matemático \& IMUS}
	\\{\footnotesize Universidad de Sevilla}
	\\{\footnotesize Sevilla, Espa\~na}
	\\{\footnotesize email: {\it scrobogna@us.es}}
}

\usepackage[a4paper]{geometry}
\usepackage{empheq}
\usepackage{times}
\usepackage{stmaryrd}
\usepackage{fourier}
\usepackage[T1]{fontenc}
\usepackage{amssymb, amsmath,  bm, mathrsfs}
\usepackage{amsfonts}
\usepackage[english]{babel}
\usepackage{amssymb,amsthm}
\usepackage{mathtools}
\usepackage{marginnote}

\DeclareMathAlphabet{\mathcal}{OMS}{cmsy}{m}{n}
\usepackage{hyperref}
\usepackage{cite}
\allowdisplaybreaks[1]
\usepackage{xfrac}
\usepackage[utf8]{inputenc}
\usepackage[T1]{fontenc}
\usepackage{enumerate}
\usepackage{accents}

\usepackage{tikz}
\usetikzlibrary{arrows, automata, backgrounds, shadows, patterns, calc, hobby, plotmarks, shapes}
\usetikzlibrary{shapes.misc}

\tikzset{cross/.style={cross out, draw=black, minimum size=2*(#1-\pgflinewidth), inner sep=0pt, outer sep=0pt},
cross/.default={1pt}}

\newcommand{\ds}{\textnormal{d}{s}}
\newcommand{\dd}{\textnormal{d}}

\newcommand{\pare}[1]{\left( #1 \right)}

\newcommand{\Pare}[1]{\Big( #1 \Big)}
\newcommand{\norm}[1]{\left\| #1 \right\|}
\newcommand{\av}[1]{\left| #1 \right|}
\newcommand{\bra}[1]{\left[ #1 \right]}
\newcommand{\set}[1]{\left\{ #1 \right\}}

\newcommand{\ddt}{\frac{\textnormal{d}}{\textnormal{d}t}}
\newcommand{\pv}{\textnormal{p.v.}}

\newcommand{\RN}[1]{%
  \textup{\uppercase\expandafter{\romannumeral#1}}%
}

\newcommand{\cA}{\mathcal{A}}
\newcommand{\cC}{\mathcal{C}}
\newcommand{\cI}{\mathcal{I}}

\newcommand{\cB}{\mathcal{B}}

\newcommand{\cJ}{\mathcal{J}}

\newcommand{\cS}{\mathcal{S}}

\newcommand{\bR}{\mathbb{R}}
\newcommand{\bN}{\mathbb{N}}

\newcommand{\bS}{\mathbb{S}}

\newcommand{\cZ}{\mathcal{Z}}
\newcommand{\cH}{\mathcal{H}}
\newcommand{\cO}{\mathcal{O}}

\newcommand{\cY}{\mathcal{Y}}
\newcommand{\cP}{\mathcal{P}}
\newcommand{\cN}{\mathcal{N}}

\newcommand{\bZ}{\mathbb{Z}}

\newcommand{\bL}{\mathbb{L}}

\newcommand{\bt}{\bar{\theta}}
\newcommand{\br}{\bar{r}}
\newcommand{\be}{\bar{\eta}}

\newcommand{\btp}{\overline{\theta'}}
\newcommand{\brp}{\overline{r'}}
\newcommand{\bep}{\overline{\eta'}}

\newcommand{\ut}{\underline{\theta}}
\newcommand{\ur}{\underline{r}}
\newcommand{\ue}{\underline{\eta}}

\newcommand{\utp}{\underline{\theta'}}
\newcommand{\urp}{\underline{r'}}
\newcommand{\uep}{\underline{\eta'}}

\newcommand{\hra}{\hookrightarrow}

\newcommand{\jump}[1]{\left\llbracket #1 \right\rrbracket}

\theoremstyle{theorem}
\newtheorem{theorem}{Theorem}[section]

\newtheorem*{theorem*}{Theorem}
\newtheorem{prop}[theorem]{Proposition}
\newtheorem{lemma}[theorem]{Lemma}

\theoremstyle{definition}
\newtheorem{definition}[theorem]{Definition}

\numberwithin{equation}{section}

\begin{document}
\maketitle

\begin{abstract}
In this paper we study a  toy model of the Peskin problem  that captures the motion of the full Peskin problem in the normal direction and discards the tangential elastic stretching contributions. This model takes the form of a fully nonlinear scalar contour equation. The Peskin problem is a fluid-structure interaction problem that describes the motion of an elastic rod immersed in an incompressible Stokes fluid. We prove global in time existence of solution for initial data in the critical Lipschitz space. Using a new decomposition together with cancellation properties, pointwise methods allow us to obtain the desired estimates in the Lipschitz class. Moreover, we perform energy estimates in order to obtain that the solution lies in the space $L^2 \pare{ \bra{0,T};H^{3/2} }$ to satisfy the contour equation pointwise.
\end{abstract}

{\small
\tableofcontents}



\allowdisplaybreaks

\section{Introduction}
The two-dimensional Peskin problem  \cite{peskin1,peskin2} is a fluid-solid interaction problem that describes the flow of a viscous incompressible fluid in a region containing immersed boundaries. These immersed boundaries move with the fluid and exert forces on the fluid itself. An example of such a boundary is the flexible leaflet of a human heart valve. The immersed boundary method was initially formulated by Peskin to study flow patterns around heart valves \cite{peskin1}. This method was later developed to solve other fluid-structure interaction problems appearing in many different applications in physics, biology and medical sciences\cite{peskin2}. The distinguishing feature of this method was that the entire simulation was carried out on a Cartesian grid, and a novel procedure was formulated for imposing the effect of the immersed boundary on the flow.

More concretely, we consider the scenario where there is a elastic rod immersed in Stokes flow. Consequently, the filament, described by the simple, closed curve
\begin{equation*}
\Gamma(t)=\set{X(s,t)=\pare{ X_1(s,t),X_2(s,t) },\ s\in\mathbb{S}^1 },
\end{equation*}
drives the fluid and  generates the flow, while the flow pushes the rod and changes its shape. This curve separates the plane into two different regions, the outer region $\Omega^-(t)$ and the inner region $\Omega^+(t)$. \\

Mathematically, when the tension is $T(\alpha)$ and the elastic force density takes the form
\begin{equation*}
F(X)=\partial_s\left(T(|\partial_s X|)\frac{\partial_s X}{|\partial_s X|}\right),
\end{equation*}
the Peskin problem reads (see \cite{rodenberg20182d} for more details)
\begin{subequations}\label{peskin}
	\begin{align}
	-\Delta u^\pm&=-\nabla p^\pm && \text{ in }\Omega^\pm(t)\\
	\nabla \cdot u^\pm&=0 && \text{ in }\Omega^\pm(t)\\
	\jump{u}&=0 && \text{ on }\Gamma(t)\\
	\jump{\left(\nabla u+\nabla u^T-p\text{Id}\right)n}&=\frac{F(X)}{|\partial_s X|} && \text{ on }\Gamma(t)\\
	\partial_t X&=u && \text{ on }\Gamma(t),
	\end{align}
\end{subequations}
where $n$ denotes the outward pointing unit normal to the free boundary $\Gamma(t)$ and 
$$
\jump{U}=U^+-U^-.
$$

In the particular case where each infinitesimal segment of the rod behaves like a Hookean spring with elasticity coefficient equal to 1, we have that $T(\alpha)=\alpha$ and \eqref{peskin} provides
\begin{subequations}\label{peskin2}
	\begin{align}
	-\Delta u^\pm&=-\nabla p^\pm && \text{ in }\Omega^\pm(t)\\
	\nabla \cdot u^\pm&=0 && \text{ in }\Omega^\pm(t)\\
	\jump{u}&=0 && \text{ on }\Gamma(t)\\
	\jump{\left(\nabla u+\nabla u^T-p\text{Id}\right)n}&=\frac{\partial_s^2 X}{|\partial_s X|} && \text{ on }\Gamma(t)\\
	\partial_t X&=u && \text{ on }\Gamma(t).
	\end{align}
\end{subequations}

There is a large literature in the numerical analysis and applied mathematics communities for this problem. However, the works developing the theory for the PDEs \eqref{peskin2} are still scarce. On one hand Lin \& Tong \cite{LT19} proved a local existence result for arbitrary $H^{5/2}$ initial data. Furthermore, they also proved the global existence and exponential decay towards equilibrium for $H^{5/2}$ initial data near certain particular configurations. 

Mori, Rodenberg \& Spirn proved in \cite{MRS19} a local well-posedness result for \eqref{peskin2} for initial data of arbitrary size in the \emph{little H\"older} space $h^{1,\gamma}, \  \gamma > 0$. In addition, these authors also proved that the solution becomes $\cC^n$ for arbitrary $n$ in arbitrarily short amount of time, and that the above unique local solutions are global and decay exponentially toward a uniformly distributed circle of positive radius when the initial data is small in the $ h^{1, \gamma} $ topology.

The authors of \cite{MRS19} proved the above results taking advantage of the contour dynamics formulation of the problem \eqref{peskin2}. Indeed, if we drop the $t$ from the notation the system \eqref{peskin2} can be equivalently written as the following nonlinear and nonlocal system of 2 equations for $X$ \cite{LT19,MRS19}:
\begin{align}
\label{peskin2contoura}
\partial_t X(s)=\pv\int_{\mathbb{S}^1} G(X(s)-X(\sigma))\partial_\sigma^2 X(\sigma) \dd \sigma.
\end{align}
where the kernel $G$ is the so-called Stokeslet
$$
G\pare{z}= \frac{1}{4\pi}\left( -\log \av{z} \ I + \frac{z\otimes z}{\av{z}^2} \right). 
$$ 
Very recently, Garcia-Juarez, Mori \& Strain \cite{garcia2020peskin} proved a global well-posedness result for the Peskin problem when two fluids with different viscosities are considered. Their result applies for medium size initial interfaces in critical spaces akin to the Wiener algebra and shows instant analytic smoothing.

As noted before \cite{LT19,MRS19}, the Peskin problem has certain similarities with the Muskat problem (see \cite{alazard2020convexity,alazard2020endpoint,cameron2018global,castro2012rayleigh,cordoba2011interface,cheng2016well,gancedo2017survey,gancedo2020global,granero2020growth,
	matioc2018muskat,matioc2018viscous,nguyen2020paradifferential, GGS19, scrobogna2020well} and the references therein)
\begin{subequations}\label{Muskat}
	\begin{align}
	u^\pm&=-\nabla p^\pm -\rho^\pm(0,1) && \text{ in }\Omega^\pm(t)\\
	\nabla \cdot u^\pm&=0 && \text{ in }\Omega^\pm(t)\\
	\jump{p}&=0 && \text{ on }\Gamma(t)\\
	\partial_t x &=u\cdot n && \text{ on }\Gamma(t).
	\end{align}
\end{subequations} 
First of all, both free boundary problems can be written as contour equations akin to \eqref{peskin2contoura}. Indeed, the Muskat problem when the fluids are separated by the graph of a the function $x(s,t)\in\mathbb{R}$ can be written as
\begin{align}\label{peskin2contourab}
\partial_t x(s)=\pv\int_{\mathbb{S}^1} K\pare{ x(s)-x(\sigma) }\pare{ x'(s)-x'(\sigma) } \dd \sigma,
\end{align}
where the kernel $K$ is a nonlinear version of the Hilber transform \cite{cordoba2011interface}. Also, both systems have a natural energy balance, in the case of the Muskat problem, the energy law reads
$$
-\jump{\rho}\|x(T)\|_{L^2(\mathbb{S}^1)}^2+2\int_0^T\|u(t)\|_{L^2(\mathbb{S}^1\times\mathbb{R})}^2\dd t= -\jump{\rho}\|x_0\|_{L^2(\mathbb{S}^1)}^2 , 
$$
while for the Peskin problem, the energy balance is
$$
\|X'(T)\|_{L^2(\mathbb{S}^1)}^2+2\int_0^T\|\nabla u(t)\|_{L^2(\mathbb{R}^2)}^2\dd t= \|X'_0\|_{L^2(\mathbb{S}^1)}^2.
$$
In both cases the energy balance is too weak to, just by itself, provide us with global existence of weak solutions. Some other similarities appear at the linear level but before stating that we need to introduce some notation. Let us recall the definition of the periodic Hilbert transform
\begin{equation*}
\cH f \pare{s} = \frac{1}{2\pi} \pv\int_{\mathbb{S}^1} \cot\pare{\alpha / 2} f \pare{s-\alpha} \dd \alpha.
\end{equation*}
Then we define the Lambda operator $ \Lambda f =  \cH \partial_s f $. With this notation we observe that the linearized Peskin problem (around the unitary circle) is \cite[Lemma 6.2]{LT19}
\begin{align} \label{eq:linear_Peskin}
\partial_t \cY = - \frac{1}{4} \ \Lambda \cY  +  \frac{1}{4}\pare{\begin{array}{cc}
	0 & -\cH \\ \cH & 0
	\end{array}} \cY,  
\end{align}
while the linear Muskat problem equals
\begin{equation}
\label{eq:linear_Muskat}
\partial_t y =  \frac{\jump{\rho}}{2} \Lambda y.
\end{equation}
Then we notice that, despites its numerous similarities, the Peskin problem and the Muskat problem are rather different, being this difference already clear at the linear level. The first stark difference can be immediately deduced comparing \eqref{eq:linear_Peskin} with \eqref{eq:linear_Muskat}; while \eqref{eq:linear_Muskat} has a difussive operator that behaves well for $ \dot{W}^{k, \infty}, k\in \bN $ functions, the linearized Peskin problem \eqref{eq:linear_Peskin} is necessarily more challenging in the same functional setting due to the unboundedness of the Hilbert transform $ \cH $ in $ L^\infty $ and the coupling of both unknowns in the system. This problem is also present in the toy model that we study in the present manuscript and will be handled noticing that, denoting with $ J $ the symplectic matrix, the term $ -J\cH \cY $ appearing on the right hand side of \eqref{eq:linear_Peskin} codify an inertial displacement. 
 
Furthermore if we decouple the linear Peskin problem we find additional differences. Indeed, if we take a time derivative, we find that
$$
\partial_t^2 \cY^1 =-\frac{1}{4} \Lambda  \partial_t\cY^1-\frac{1}{4} \cH \partial_t\cY^2=-\frac{1}{4} \Lambda  \partial_t\cY^1-\frac{1}{4} \cH \left(-\frac{1}{4} \Lambda \cY^2+\frac{1}{4}\cH \cY^1\right).
$$
Taking a space derivative of the equation for $\cY^1$ we obtain that
$$
\partial_t \partial_s\cY^1 +\frac{1}{4} \Lambda \partial_s \cY^1=-\frac{1}{4}\Lambda \cY^2.
$$
Substituting the latter expression we conclude that
$$
\partial_t^2 \cY^1 =-\frac{1}{4} \Lambda  \partial_t\cY^1-\frac{1}{4} \cH \left(\partial_t \partial_s\cY^1 +\frac{1}{4} \Lambda \partial_s \cY^1+\frac{1}{4}\cH \cY^1\right).
$$
Using the properties of the Hilbert transform for zero-mean functions we find the linear Klein-Gordon-like equation 
$$
\partial_t^2 \cY +\frac{1}{2} \Lambda  \partial_t\cY=\frac{1}{16} \partial_s^2\cY+\frac{1}{16} \cY
$$
for each component. This equation is very different to the parabolic equation for the Muskat problem.

On the other hand, in the Peskin problem it is not possible to reparameterize the contour equation at convenience to obtain the same nonlinear solution. While the reparameterization freedom in free boundary problems for incompressible fluid have been extensively used as a help to deal with the nonlinear structure of nonlocal equations, the Peskin problem is sensible to reparameterizations in the sense that concentration of particles in the rod affects the elastic dynamics. Thus, different reparameterizations  give rise to different dynamics and can converge to different steady state as time goes to infinity \cite{LT19,MRS19,garcia2020peskin}. Nevertheless,  the right hand side of the nonlocal system  \eqref{peskin2contoura} is invariant with respect to translations, so that the appropriate time-dependent translation $M(t)$ allows us to control  some linear contributions which arise in the dynamics of the linearized version of \eqref{peskin2contoura}.

Additionally, the Peskin problem lacks a divergence form structure. This is another rather big difference with the Muskat problem and makes passing to the limit in the weak formulation a rather delicate issue. To overcome this challenge we will use energy estimates in $\dot{H}^1$. This energy estimate will give the parabolic effect which is necessary to pass to the limit in the weak formulation.

In order to better understand the mathematical subtleties and challenges of the Peskin problem, in this work we consider a scalar model of the Peskin problem (see equation \eqref{eq:model} below). This  toy model, which we denote from now on with the name of \emph{N-Peskin} problem, takes the form of a fully nonlinear contour equation and shares most of the difficulties mentioned above  but discards the  contributions of the motion due to tangential elastic stretching of the rod. The study of scalar toy models in fluid dynamics  is a classical research area that goes back to the work of Constantin-Lax-Majda \cite{CLM1985}.

The plan of the paper is as follows: In Section \ref{sec2} we present our main result and the methodology. Furthermore, we also introduce there our new formulation of the Peskin problem. In Section \ref{sec3}, we state several pointwise bounds for singular integral operators. In Section \ref{sec:W1inftydecay} we prove the \emph{a priori} estimates showing the decay in the Lipschitz norm. Later, in Section \ref{sec:H1inftydecay} we prove the \emph{a priori} estimates in Sobolev spaces. These estimates are lower order but allow us to use the parabolic gain of regularity. In Section \ref{sec:path} we prove the estimates for the time derivative of the solution. These estimates are required to ensure the compactness required to pass to the limit in the weak formulation. Finally, in Section \ref{sec7} we prove the main result of this paper.

\subsection{Notation}
We denote with $ C $ any positive constant whose value is independent of the physical parameters of the problem, the explicit value of $ C $ may vary from line to line. We write $ A\lesssim B $ if $ A\leq C B $ and $ A\sim B $ if $ A\lesssim B $ and $ B\lesssim A $. 

We denote by $ \cP \in \cC^\infty\pare{\left[0, 1 \right) ; \bR_+} $ any universal function such that $ \cP\pare{0} \geq 0 $ and such that for any $ y\in \bra{0, 1/2} $ there exists a $ N $ for which the bound $
\cP\pare{y}\leq C\pare{1+y}^N $
holds true. The explicit value of $ \cP $ may vary from line to line. 

The one dimensional torus, i.e. the interval $ \bra{-\pi, \pi} $ endowed with periodic boundary conditions, is denoted by $ \bS^1 $. Given any $ f\in \cC^1\pare{\bS^1} $ we denote with $ f' $ the covariant derivative of $ f $ onto $ \bS^1 $ endowed with the euclidean metric, and $ f^{(k)} $ denotes the operator $ \cdot ' $ iterated $ k $ times. We define $ \Lambda f = \cH f' $ and we recall that such operator can be expressed as the Fourier multiplier $ \widehat{\Lambda f}\pare{n} = \av{n} \hat{f}\pare{n} $. We can thus define the Sobolev spaces of fractional order (here $ \cS  $ denotes the periodic Schwartz class and $ \cS_0 $ the periodic Schwartz class with zero average)
\begin{align*}
H^s\pare{\bS^1} = \set{f\in \cS' \ \left| \ \pare{1+\Lambda}^s f \in L^2 \right. }, && \dot{H}^s\pare{\bS^1} = \set{f\in \cS'_0 \ \left| \ \Lambda^s f \in L^2 \right. }, 
\end{align*}
for any $ s \in \bR $. For any $ \pare{p, k}\in\bra{1, \infty}\times \bN $ we denote with 
\begin{align*}
W^{k, p}\pare{\bS^1} = \set{f\in\cS' \ \left| \ f, f^{(k)}\in L^p\pare{\bS^1} \right. }, && \dot{W}^{k, p}\pare{\bS^1} = \set{f\in\cS'_0 \ \left| \ f^{(k)}\in L^p\pare{\bS^1} \right. }. 
\end{align*}
We use the notation 
\begin{align*}
L^p = L^p\pare{\bS^1}, && H^s = H^s\pare{\bS^1}, && W^{k, p} = W^{k, p}\pare{\bS^1}, 
\end{align*}
for functional spaces defined on the one-dimensional torus.
Additionally, we use the simplified notation
\begin{equation*}
\int \bullet \ \dd s=\text{p.v.}\int_{\mathbb{S}^1}  \bullet \ \dd s =\text{p.v.}\int_{-\pi}^\pi  \bullet \ \dd s  , 
\end{equation*}
in order to indicate Cauchy principal value integrals on the one-dimensional torus.

\section{Main result and methodology}\label{sec2}
\subsection{Derivation of the N-Peskin problem} \label{sec:h-M}
As we have seen before, the Peskin problem can be written as the following contour equations 
\begin{equation}\label{eq:Peskin} \tag{P}
\begin{aligned}
\partial_t X\pare{s, t} & = \int G\pare{X\pare{s, t} - X\pare{\sigma , t}} X''\pare{\sigma, t} \dd \sigma, \\
G\pare{z} & = G_1(z)+G_2(z),\\
G_{1} \pare{z} & = -\frac{1 }{4\pi}  \log\av{z}\ I, \\
G_{2} \pare{z} & = \frac{1}{4\pi} \frac{z\otimes z}{\av{z}^{2}}=\frac{1}{4\pi|z|^2}\left(\begin{array}{cc}z_1^2 & z_1z_2\\ z_1z_2 & z_2^2\end{array}\right).
\end{aligned}
\end{equation}
In this section we present the model of the Peskin problem that we  consider in this work.

To simplify the notation we write $ \gamma = \gamma \pare{s} = \pare{ \cos  s, \sin  s} $ and $Y\pare{s, t} =  \ 
\pare{1 + h\pare{s, t}} \gamma\pare{s}$. Then let us suppose that
\begin{equation}
\label{eq:ansatz}
\begin{aligned}
X\pare{s, t} = & \ M\pare{t} + Y\pare{s, t}, 
\end{aligned}
\end{equation}
where we define the point $M(t)$ as the solution of the following ODE in terms of $h(s,t)$
\begin{equation}\label{eq:M}
\ddt{M}\pare{t} =  \frac{1}{4} \frac{1}{2\pi}\int h\pare{s, t} \pare{ \cos (s), \sin  (s)} \ds . 
\end{equation} Roughly speaking, we use this $M(t)$ to control the inertial effects of the system. Mathematically, this unknown is required in order to absorb a low order nonlocal linear contribution akin to the first Fourier mode.

With the ansatz \eqref{eq:ansatz} the evolution equation \eqref{eq:Peskin} becomes
\begin{equation*}
\partial_t Y\pare{s, t} + \ddt{M}\pare{t} = \int G\pare{Y \pare{s, t} - Y\pare{\sigma , t}} Y''\pare{\sigma, t} \dd \sigma. 
\end{equation*}

We can further compute
\begin{equation*}
\gamma\pare{s} \partial_t h\pare{s} +  \ddt{M} \pare{t} 
= \int G \pare{\pare{1+h\pare{s}}\gamma \pare{s} - \pare{1+h\pare{\sigma}}\gamma \pare{\sigma}} \bra{ \gamma\pare{\sigma}\pare{ h''\pare{\sigma} -1-h\pare{\sigma}  } + 2 \gamma'\pare{\sigma} h'\pare{\sigma} } \dd \sigma .
\end{equation*}
We write
\begin{align*}
I_1 \pare{s} & =  \int G_{1}\pare{\pare{1+h\pare{s}}\gamma \pare{s} - \pare{1+h\pare{\sigma}}\gamma \pare{\sigma}} \bra{ \gamma\pare{\sigma}\pare{ h''\pare{\sigma} - 1 - h\pare{\sigma}  } + 2 \gamma'\pare{\sigma} h'\pare{\sigma} } \dd \sigma, \\
I_2 \pare{s} & =  \int G_{2} \pare{\pare{1+h\pare{s}}\gamma \pare{s} - \pare{1+h\pare{\sigma}}\gamma \pare{\sigma}} \bra{ \gamma\pare{\sigma}\pare{h''\pare{\sigma} - 1 - h\pare{\sigma}  } + 2 \gamma'\pare{\sigma} h'\pare{\sigma} } \dd \sigma
,
\end{align*}

so that
\begin{equation}\label{eq:eveqh1}
\begin{aligned}
\gamma\pare{s} \partial_t h\pare{s} & =  I_1 \pare{s} + I_2 \pare{s} -  \ddt{M} \pare{t}.
\end{aligned}
\end{equation}

So far, \eqref{eq:M} and \eqref{eq:eveqh1} are a new formulation of the full Peskin problem. This new $h-M$ formulation is the starting point for the derivation of our scalar N-Peskin.

Taking the scalar product of \eqref{eq:eveqh1} with $ \gamma\pare{s} $, we derive the \emph{scalar} evolution equation
\begin{equation}\label{eq:model}
\partial_t h\pare{s, t}  = \gamma\pare{s} \cdot I_1 \pare{s, t} + \gamma\pare{s} \cdot I_2 \pare{s, t} -   \gamma\pare{s}\cdot   \ddt{M} \pare{t}. 
\end{equation}
We propose equation \eqref{eq:model} as a scalar model of the full Peskin problem. However, as the tangential velocity is neglected in our approach we name this equation the N-Peskin problem.

Let us simplify \eqref{eq:model}. Using classical trigonometric identities, the first term can be explicitly written as:
\begin{multline}\label{eq:computation_I1}
\gamma\pare{s} \cdot I_1 \pare{s}  \\ =    \int -\frac{1 }{4\pi}  \log\pare{\av{\pare{1+h\pare{s}}\gamma \pare{s} - \pare{1+h\pare{\sigma}}\gamma\pare{\sigma}}} \bra{ \cos\pare{s-\sigma}\pare{ h''\pare{\sigma} - 1 - h\pare{\sigma}  } + 2 \sin\pare{s-\sigma} h'\pare{\sigma} } \dd \sigma\\ =    \int -\frac{1 }{8\pi}  \log\pare{\av{\pare{1+h\pare{s}}\gamma \pare{s} - \pare{1+h\pare{\sigma}}\gamma\pare{\sigma}}^2} \partial_\sigma^2\bra{ \cos\pare{s-\sigma}\pare{ 1+ h\pare{\sigma}  }}\dd \sigma  .
\end{multline}
Let us now simplify the expression of the second kernel $G_2$. We have that
\begin{multline*}
\gamma \pare{s} \cdot G_{2} \pare{\pare{1+h\pare{s}}\gamma \pare{s} - \pare{1+h\pare{\sigma}}\gamma \pare{\sigma}} \cdot \gamma \pare{\sigma} \\
\begin{aligned}
= & \ \frac{1}{4\pi}  \frac{\gamma_i \pare{s} \pare{\pare{1+h\pare{s}}\gamma_i \pare{s} - \pare{1+h\pare{\sigma}}\gamma_i \pare{\sigma}} \pare{\pare{1+h\pare{s}}\gamma_j \pare{s} - \pare{1+h\pare{\sigma}}\gamma_j \pare{\sigma}} \gamma_j \pare{\sigma}}{\av{\pare{1+h\pare{s}}\gamma \pare{s} - \pare{1+h\pare{\sigma}}\gamma \pare{\sigma}}^2} ,  \\
= & \ \frac{1}{4\pi}  \frac{ \bra{1+h\pare{s}- \pare{1+h\pare{\sigma}} \cos \pare{s-\sigma}} \bra{\pare{1+h\pare{s}} \cos \pare{s-\sigma} - \pare{1+h\pare{\sigma}}}}{ 
	\pare{1 + h\pare{s}}^2 + \pare{1 + h\pare{\sigma}}^2  - 2 \pare{1 + h\pare{s}}\pare{1 + h\pare{\sigma}} \cos\pare{s-\sigma}  
} , 
\end{aligned}
\end{multline*}
and 
\begin{multline*}
\gamma \pare{s} \cdot G_{2} \pare{\pare{1+h\pare{s}}\gamma \pare{s} - \pare{1+h\pare{\sigma}}\gamma \pare{\sigma}} \cdot \gamma' \pare{\sigma} \\
\begin{aligned}
= & \ \frac{1}{4\pi} \frac{\gamma_i \pare{s} \pare{\pare{1+h\pare{s}}\gamma_i \pare{s} - \pare{1+h\pare{\sigma}}\gamma_i \pare{\sigma}} \pare{\pare{1+h\pare{s}}\gamma_j \pare{s} - \pare{1+h\pare{\sigma}}\gamma_j \pare{\sigma}} \gamma_j' \pare{\sigma}}{ \av{\pare{1+h\pare{s}}\gamma \pare{s} - \pare{1+h\pare{\sigma}}\gamma \pare{\sigma}}^2 },  \\
= & \ \frac{1}{4\pi}  \frac{ \bra{1+h\pare{s}- \pare{1+h\pare{\sigma}} \cos \pare{s-\sigma}} \pare{1+h\pare{s}} \sin \pare{s-\sigma} }{
	\pare{1 + h\pare{s}}^2 + \pare{1 + h\pare{\sigma}}^2  - 2 \pare{1 + h\pare{s}}\pare{1 + h\pare{\sigma}} \cos\pare{s-\sigma}  
} , 
\end{aligned}
\end{multline*}

After the change of variables $ \sigma = s-\alpha $ we find that
\begin{multline*}
\gamma\pare{s} \cdot I_2 \pare{s}  \\
= \frac{1}{4\pi}   \int    \frac{ \bra{1+h\pare{s}- \pare{1+h\pare{s-\alpha}} \cos \alpha} \bra{\pare{1+h\pare{s}} \cos \alpha - \pare{1+h\pare{s-\alpha}}}   }{ 
	\pare{1 + h\pare{s}}^2 + \pare{1 + h\pare{s-\alpha}}^2  - 2 \pare{1 + h\pare{s}}\pare{1 + h\pare{s-\alpha}} \cos\alpha } \ \pare{ h''\pare{s-\alpha}  - 1 - h\pare{s-\alpha}  } \dd \alpha \\
+ \frac{1}{2\pi}  \pare{1+h\pare{s}} \int    \frac{ \bra{1+h\pare{s}- \pare{1+h\pare{s-\alpha}} \cos \alpha}  \sin \alpha    }{
	\pare{1 + h\pare{s}}^2 + \pare{1 + h\pare{s-\alpha}}^2  - 2 \pare{1 + h\pare{s}}\pare{1 + h\pare{s-\alpha}} \cos\alpha }  h'\pare{s-\alpha}  \dd \alpha.
\end{multline*}

Collecting the previous expressions and changing variables, we conclude the following scalar equation for $h$:
\begin{equation}
\label{eq:eveqh2}
\begin{aligned}
& \ \partial_t h\pare{s} + \gamma\pare{s}\cdot   \ddt{M}
\\
= &   - \frac{1}{8\pi}\int \Bigg\{  \log  \pare{ 
	\pare{1 + h\pare{s}}^2 + \pare{1 + h\pare{s-\alpha}}^2  - 2 \pare{1 + h\pare{s}}\pare{1 + h\pare{s-\alpha}} \cos\alpha} \ \partial_\alpha^2 \bra{ \big. \cos\alpha \pare{1 + h\pare{s-\alpha}} } \Bigg\} \dd \alpha \\
& +\frac{1}{4\pi}  \int    \frac{ \bra{1+h\pare{s}- \pare{1+h\pare{s-\alpha}} \cos \alpha} \bra{\pare{1+h\pare{s}} \cos \alpha - \pare{1+h\pare{s-\alpha}}}   }{ 
	\pare{1 + h\pare{s}}^2 + \pare{1 + h\pare{s-\alpha}}^2  - 2 \pare{1 + h\pare{s}}\pare{1 + h\pare{s-\alpha}} \cos\alpha } \ \pare{ h''\pare{s-\alpha}  - 1 - h\pare{s-\alpha}  } \dd \alpha \\
& + \frac{1}{2 \pi} \pare{1+h\pare{s}} \int    \frac{ \bra{1+h\pare{s}- \pare{1+h\pare{s-\alpha}} \cos \alpha}  \sin \alpha    }{
	\pare{1 + h\pare{s}}^2 + \pare{1 + h\pare{s-\alpha}}^2  - 2 \pare{1 + h\pare{s}}\pare{1 + h\pare{s-\alpha}} \cos\alpha }  h'\pare{s-\alpha}  \dd \alpha.
\end{aligned}
\end{equation}

Then, we define the following notion of weak solution

\begin{definition}\label{definition1}
	We say that 
	$h$ is a weak solution of the N-Peskin problem \eqref{eq:eveqh2} if the following equality holds
	$$
	-\int \varphi(s,0)h_{0}(s)\dd s +\int_0^T\int -\pare{ \partial_t \varphi (s,t) h (s,t) +  \frac{1}{4} \Lambda\varphi (s,t) h(s,t)  - \cN\pare{h(s,t)}\varphi(s,t) } \dd s\dd t=0,$$ 
	for all $\varphi\in C^\infty_{c}([ 0,T)\times \mathbb{S}^1)$,
	where $ \cN $ is the nonlinearity
	\begin{equation*}
	\cN\pare{h\pare{s, t}} - \frac{1}{4}\Lambda h\pare{s, t} + \gamma\pare{s}\cdot   \ddt{M}\pare{t} = \text{r.h.s. of \eqref{eq:eveqh2}}.  
	\end{equation*}
\end{definition}

\subsection{The linear $h-M$ formulation of the Peskin problem} \label{sec:linearized_h}

To better understand the role of $M(t)$ and the reason behind its definition through the aforementioned ODE, we are going to compute the he linearized Peskin problem in the $h-M$ formulation. The linear Peskin problem for arbitrary curves can be expressed as \eqref{eq:linear_Peskin}. In the radial configuration we have that
\begin{align} \label{eq:cYtor}
\cY_1\pare{s} = r\pare{s,t} \ \cos s , && \cY_2\pare{s} = r\pare{s,t} \ \sin s,
\end{align}
where
$$
r(s,t)=1+h(s,t).
$$
Thus, multiplying \eqref{eq:linear_Peskin} by $ \gamma $, we obtain that
\begin{align*}
\partial_t r(s,t)&= - \frac{1}{4}\bra{ \cos (s) \ \Lambda \pare{r(s,t) \cos(s) } + \sin (s) \  \Lambda \pare{r(s,t)\sin(s) } \big.  } \\&\quad+ 
\frac{1}{4}\bra{ -\cos (s) \ \cH \pare{r(s,t) \sin(s) } + \sin (s) \  \cH \pare{r(s,t)\cos(s) }  \big.  }\\& = L_1 + L_2. 
\end{align*}

Dropping the $t$ from the notation we compute that
\begin{align*}
- \frac{1}{4}  \cos (s) \ \Lambda \pare{r \cos } \pare{s} = & \ -\frac{1}{8\pi} \cos (s) \  \int  \cot\pare{\alpha/2} \pare{r\pare{s-\alpha} \cos\pare{s-\alpha}}' \dd \alpha , \\
= & \ \frac{1}{8\pi} \cos (s) \   \int  \cot\pare{\alpha/2} \partial_\alpha \pare{r\pare{s-\alpha} \cos\pare{s-\alpha} - r\pare{s} \cos s} \dd \alpha , \\
= & \ \frac{1}{8\pi} \cos^2 (s) \ \int  \frac{1}{2\sin^2\pare{\alpha / 2}}    \pare{r\pare{s-\alpha} \cos\alpha - r\pare{s}} \dd \alpha , \\
& \ + \frac{1}{8\pi} \sin (s) \ \cos (s) \  \int  \frac{1}{2\sin^2\pare{\alpha / 2}}  r\pare{s-\alpha}\sin(\alpha) \  \dd \alpha , 
\end{align*}
and
\begin{align*}
- \frac{1}{4}  \sin (s) \ \Lambda \pare{r \sin } \pare{s} = & \ -\frac{1}{8\pi} \sin (s)  \  \int  \cot\pare{\alpha/2} \pare{r\pare{s-\alpha} \sin\pare{s-\alpha}}' \dd \alpha , \\
= & \ \frac{1}{8\pi} \sin (s) \ \int  \cot\pare{\alpha/2} \partial_\alpha \pare{r\pare{s-\alpha} \sin\pare{s-\alpha} -r\pare{s} \sin s} \dd \alpha , \\
= & \ \frac{1}{8\pi} \sin^2 (s) \ \int  \frac{1}{2\sin^2\pare{\alpha / 2}} \pare{r\pare{s- \alpha}\cos\alpha - r\pare{s}} \dd \alpha , \\
& \  -\frac{1}{8\pi}  \sin (s) \cos (s) \ \int  \frac{1}{2\sin^2\pare{\alpha / 2}}  r\pare{s-\alpha}\sin(\alpha) \  \dd \alpha.
\end{align*}
As a consequence we obtain
\begin{equation*}
L_1 =  \frac{1}{4}\frac{1}{2\pi}\int  \frac{1}{2\sin^2\pare{\alpha / 2}} \pare{r\pare{s- \alpha}\cos\alpha - r\pare{s}} \dd \alpha. 
\end{equation*}
In a similar fashion we compute 
\begin{equation*}
L_2 =   \frac{1}{2}\frac{1}{2\pi} \int \cos^2\pare{\alpha/2} \ r\pare{s-\alpha} \dd \alpha. 
\end{equation*}
Summing up these two expressions and substituting $ r = 1+ h $, we conclude that
\begin{equation}\label{eq:evolution_linearization_h}
\partial_t h \pare{s} = - \frac{1}{4} \frac{1}{2\pi} \int \frac{h\pare{s} - h\pare{s-\alpha}}{2\sin^2\pare{\alpha/2}} \dd \alpha + \frac{1}{4} \frac{1}{2\pi} \int h\pare{s-\alpha} \cos \alpha \ \dd \alpha. 
\end{equation}
We see now that in the $h-M$ formulation the Peskin problem is parabolic at the linear level with a nonlocal zeroth-order forcing term. Moreover, we can compute
$$
\ddt{M}\pare{t}\cdot \gamma(s) =  \frac{1}{4}\frac{1}{2\pi} \int h\pare{\alpha, t} \pare{ \cos (\alpha), \sin  (\alpha)} \cdot \pare{\cos (s), \sin  (s)} \dd \alpha =  \frac{1}{4}\frac{1}{2\pi} \int h\pare{\alpha, t} \cos (s-\alpha) \dd \alpha. 
$$
As a consequence, we also realize that the ODE for $M(t)$ is designed to absorb some of the linear contributions.

\subsection{Main result}
It is known \cite[Section 5]{MRS19} that the set of stationary solutions of \eqref{eq:Peskin} are circles. As a consequence, the equilibrium configurations are determined by the center and radius of the stationary circle. Without loss of generality we  assume in what follows that the radius of the equilibrium circle equals one (different values can be handled similarly). The purpose of this paper is to establish the global existence and decay to equilibrium for \eqref{eq:eveqh2} in the case of Lipschitz initial data $X_0(s)\in W^{1,\infty}(\mathbb{S}^1)$ sufficiently close to an equilibrium configuration. 

In particular the following theorem is the main result of the present manuscript:

\begin{theorem}\label{teo1}Let $h_0(s)\in W^{1,\infty}(\mathbb{S}^1)$ be the initial data for \eqref{eq:eveqh2}. There exists a universal constant $c_0\ll1$ such that if $h_0$ satisfies
	\begin{align}
	\label{eq:average_relation}
	& \av{h_0}_{W^{1,\infty}(\mathbb{S}^1)}  \leq c_0, \nonumber
	\end{align}
	then there exists a global in time weak solution of \eqref{eq:eveqh2}  in the sense of Definition \ref{definition1}
	which belong to the energy space
	\begin{align*}
	h\in L^\infty \pare{ [0,T);W^{1,\infty}\pare{\bS^1}  }\cap C\pare{ [0,T);H^1\pare{\bS^1} }\cap L^2 \pare{ [0,T);H^{3/2}\pare{\bS^1}  }, && \forall \  T\in\pare{0, \infty}
	\end{align*}
Furthermore for any $ 0<1\ll t < T $ we have that
	\begin{equation}\label{eq:convergence_to_zero}	
		\av{h(t)}_{W^{1,\infty}(\mathbb{S}^1)} \leq \av{h_0}_{W^{1,\infty}(\mathbb{S}^1)} ,
	\end{equation}
	and 
\begin{equation*}
\av{h'(t)} _{L^\infty}\leq \av{h'_0}_{L^\infty} e^{-\delta t},
\end{equation*}
for a small enough $\delta(h_0)>0$.
\end{theorem}


We remark that  \eqref{eq:Peskin} is invariant with respect to the transformation
\begin{align*}
X_\lambda(s,t)=\frac{1}{\lambda}X(\lambda s,\lambda t),
&& \lambda\in\bZ^+ , 
\end{align*}
thus,
$$
 L^\infty\pare{ 0,T; \dot{W}^{1,\infty}\pare{ \mathbb{S}^1 } },
$$
is critical with respect to the previous scaling.

\subsection{Methodology}
Let us explain the main ideas behind Theorem \ref{teo1} using a simpler equation. We consider the following equation:
$$
\partial_t f+f'\Lambda f+\Lambda f=0.
$$
Using pointwise methods as in \cite{cordoba2004maximum,cordoba2009maximum}, we can obtain the following bounds
$$
\ddt{ \av{f}_{W^{1,\infty}}}\leq0,
$$
for initial data such that
$$
\av{f_0}_{W^{1,\infty}}<1.
$$
Then, we conclude the \emph{a priori} estimates in the Lipschitz class. For equations in divergence form, such estimates would lead to the global existence of weak solution via a vanishing viscosity type argument \cite{CCGS12,granero2014global}. However, for equations in non-divergence form, it is not obvious how to translate the previous bound into the global existence of weak solutions. When comparing the Peskin and the Muskat problem we see that this is an additional challenge that is inherent to the Peskin problem. To overcome this difficulty, we perform an additional $H^1$ energy estimate. A careful study of the nonlinearity will give us the appropriate bounds. Indeed, we have that
\begin{align*}
-\int f''f'\Lambda f \dd s &=\frac{1}{2}\int (f')^2\Lambda f' \dd s=\frac{1}{2}\iint(f'(s))^2\frac{\left(f'(s)-f'(\sigma)\right)}{|s-\sigma|^2} \dd \sigma \dd s=\frac{1}{2}\iint(f'(\sigma))^2\frac{\left(f'(\sigma)-f'(s)\right)}{|s-\sigma|^2} \dd \sigma \dd s\\
&=\frac{1}{4}\iint((f'(s))^2-(f'(\sigma))^2)\frac{\left(f'(s)-f'(\sigma)\right)}{|s-\sigma|^2} \dd \sigma \dd s=\frac{1}{4}\iint(f'(s)+f'(\sigma))\frac{\left(f'(s)-f'(\sigma)\right)^2}{|s-\sigma|^2} \dd \sigma \dd s\\
&\leq \frac{\av{f'}_{L^\infty}}{2} \ \av{f'}_{\dot{H}^{1/2}}^2.
\end{align*}
Thus, both estimates combined lead to a bound in
$$
f\in L^2\pare{ \bra{0,T};H^{3/2} }.
$$
With this parabolic effect we have the strong convergence of $f', \Lambda f$ and this allow us to pass to the limit in the weak formulation. 

Then, the proof of Theorem \ref{teo1} of this work can be summarized as follows:
\begin{enumerate}
	\item Using pointwise methods we conclude that, for initial data close enough to the equilibrium, the solution decays in the Lipschitz norm. The purpose of Proposition \ref{prop:W1infty_enest} is to prove the previous claim. The decay in $W^{1,\infty}$ is a crucial point in the argument as it will allow to obtain the parabolic gain of regularity. Furthermore, the exponential decay of this norm ensures that the point $M(t)$ remains uniformly in a ball for every $0<t<\infty$.
	\item The decay in the Lipschitz norm is then used to find a global estimate in 
	$$
	L^2 \pare{ \bra{0,T} ;H^{3/2} }.
	$$
	\item We invoke the parabolic gain of regularity obtained before to conclude the strong convergence of the derivative.
\end{enumerate}

\section{Pointwise estimates for the $\Lambda$ operator} \label{sec3}
In this section we collect some pointwise estimates for the fractional Laplacian that will be used in the sequel and that may be of independent interest. We start with a lemma that compares the Lambda and the Hilbert transform:

\begin{lemma}\label{lem:monotonicity_lemma}
Let $ f$ be a smooth function and define $ \bar{s},\underline{s} \in \bS^1 $ such that 
\begin{align*}
f'\pare{\bar{s}} = \max_{s\in\mathbb{S}^1} f'(s), &&
f'\pare{\underline{s}} = \min_{s\in\mathbb{S}^1}f'(s).
\end{align*} 
Then
\begin{align*}
\Lambda f'\pare{\bar{s}} - \Lambda f\pare{\bar{s}} \geq 0, && 
\Lambda f'\pare{\underline{s}} - \Lambda f\pare{\underline{s}}\leq 0. 
\end{align*}
\end{lemma}

\begin{proof}
We know that
\begin{align*}
\Lambda f' \pare{s} = & \ \frac{1}{2\pi}  \int \frac{f'\pare{s} - f'\pare{s-\alpha}}{2\sin^2\pare{\alpha/2}}, \\
\Lambda f \pare{s} = & \  -\frac{1}{2\pi} \int \cot \pare{\alpha/2} \pare{f'\pare{s} - f'\pare{s-\alpha}}, 
\end{align*}
so that
\begin{equation*}
\Lambda f' \pare{s} - \Lambda f \pare{s} = \frac{1}{2\pi}\int \frac{1+\sin(\alpha)}{2\sin^2\pare{\alpha/ 2}} \ \pare{f'\pare{s} - f'\pare{s-\alpha}}, 
\end{equation*}
and the claim follows since the integration kernel is nonnegative and
\begin{align*}
f'\pare{\bar{s}} - f'\pare{\bar{s}-\alpha} \geq 0, &&
f'\pare{\underline{s}} - f'\pare{\underline{s}-\alpha} \leq 0 . 
\end{align*}
\end{proof}

Furthermore, we observe that for zero-mean functions we have the following Poincar\'e-type pointwise inequalities (see \cite{ascasibar2013approximate} for instance)
\begin{align}\label{eq:Lambda_Linfty}
f\pare{\bar{s}}\leq C \ \Lambda f\pare{\bar{s} }, && -f\pare{\underline{s}}\leq -C \ \Lambda f\pare{\underline{s} }. 
\end{align}
Similarly as in the proof of Lemma \ref{lem:monotonicity_lemma} we can prove the following result:

\begin{lemma}
\label{lem:monotonicity_lemma2}
Let $ f$ be a smooth function and define $ \bar{s},\underline{s} \in \bS^1 $ such that 
\begin{align*}
f'\pare{\bar{s}} = \max_{s\in\mathbb{S}^1} f'(s), &&
f'\pare{\underline{s}} = \min_{s\in\mathbb{S}^1}f'(s).
\end{align*} 
Let $ b=b\pare{\alpha} \geq 0 $ for every $ \alpha\in \bS^1 $ and let us define the operators
\begin{align*}
\Lambda_b f\pare{s} & = \frac{1}{2\pi}  \int \frac{b\pare{\alpha}}{2 \sin^2\pare{\alpha / 2}} \pare{f' \pare{s} - f' \pare{s-\alpha}} \dd\alpha, \\ 
\cH_b f\pare{s} & = -\frac{1}{2\pi}  \int \frac{b\pare{\alpha}}{\tan\pare{\alpha / 2}} \pare{f'\pare{s} - f'\pare{s-\alpha}}, 
\end{align*}
then we have
\begin{align*}
\Lambda_b f\pare{\bar{s}} - \cH_b f\pare{\bar{s}} \geq 0, &&
\Lambda_b f\pare{\underline{s}} - \cH_b f\pare{\underline{s}} \leq 0
\end{align*}
\end{lemma}

Finally, let us provide with an alternative expression for $ \Lambda f' $ when $ f $ is $ \cC^2\pare{\bS^1} $. This expression will be very useful when performing the pointwise estimates. We know that
\begin{equation*}
\Lambda f'\pare{s} = \frac{1}{2\pi}  \int \frac{f'\pare{s} - f'\pare{s-\alpha}}{2\sin^2 \pare{\alpha / 2}} \dd \alpha = \frac{1}{2\pi}  \int \frac{\partial_\alpha \bra{f'\pare{s} \alpha - \pare{ f\pare{s} - f\pare{s-\alpha} }}}{2\sin^2 \pare{\alpha / 2}} \dd \alpha. 
\end{equation*}
Integrating by parts and exploiting the regularity $ f \in \cC^2\pare{\bS^1} $, we obtain that
\begin{equation}
\label{eq:Lambda_alternative}
\Lambda f'\pare{s} = \frac{1}{2\pi} \int \frac{f'\pare{s}\alpha -\pare{f \pare{s} - f \pare{s-\alpha}}}{2\sin^3\pare{\alpha/2}} \cos\pare{\alpha/2} \ \dd\alpha. 
\end{equation}

 \section{\emph{A priori} estimates in $ W^{1, \infty} $}\label{sec:W1inftydecay}
Let us introduce some notation that will simplify the exposition. For a smooth function $h$ and any $ n\in \bN $ let us denote the application s $ t \mapsto \overline{s^n_t} $ and $ t \mapsto \underline{s^n_t} $ such that
\begin{align*}
 h^{\pare{ n }} \pare{\overline{s^n_t} , t} = \max _s \set{ h^{\pare{ n }} \pare{s , t}}=\max_s \set{ \partial_s^n h} \pare{s , t}, &&
 h^{\pare{ n }} \pare{\underline{s^n_t} , t} = \min_s  \set{h^{\pare{ n }} \pare{s , t}}=\min _s \set{ \partial_s^n h} \pare{s , t}. 
\end{align*}
Let us define the auxiliary functions
\begin{align}\label{eq:auxiliary_variables}
r\pare{s} = 1+h\pare{s}, && \theta\pare{s, s-\alpha} = h\pare{s}-h\pare{s-\alpha}, &&  \eta \pare{s, s-\alpha, \alpha} = h\pare{s} - h\pare{s-\alpha}\cos\alpha, 
\end{align}
thus we have that
\begin{align*}
1+h\pare{s-\alpha} = r-\theta , && h'\pare{s-\alpha} = -\partial_\alpha h\pare{s-\alpha} = \partial_\alpha \theta. 
\end{align*}
Finally, we denote
\begin{equation}
\label{eq:notation_bar_0}
\begin{aligned}
\bt & = \theta \pare{\overline{s^0_t}, \overline{s^0_t}-\alpha}, & \be & = \eta \pare{\overline{s^0_t}, \overline{s^0_t}-\alpha}, & \br & = r\pare{\overline{s^0_t}}\\
\ut & = \theta \pare{\underline{s^0_t}, \underline{s^0_t}-\alpha}, & \ue& = \eta \pare{\underline{s^0_t}, \underline{s^0_t}-\alpha}, & \ur & = r\pare{\underline{s^0_t}}. 
\end{aligned}
\end{equation} 
 
In the present section we prove that the unitary circumference $\gamma$ is globally stable under small $ W^{1, \infty} $ perturbations which are graphs on the unitary circle. The detailed statement is formulated in the following proposition:
 
 \begin{prop}\label{prop:W1infty_enest}
Let  $T^\star\in(0,\infty]$ and $h = h(s,t)$ be a $\cC \pare{ \bra{ 0,T^\star}; \cC^2 } $ solution of \eqref{eq:eveqh2} such that $ \left. h\right|_{t=0} = h_0 $. Then there exists a $ c_0 > 0 $ such that if $ \av{h_0}_{W^{1, \infty}}< c_0 $ the following inequality holds
 \begin{align*}
 \av{h\pare{t}}_{W^{1, \infty}}\leq \av{h_0}_{W^{1, \infty}}&& \forall \  0\leq t<T^\star.
 \end{align*} 
Furthermore, we have that there exists a $ \delta > 0 $ s.t.
\begin{equation}
\av{h'(t)} _{L^\infty}\leq \av{h'_0}_{L^\infty} e^{-\delta t}.\label{eq:exp_decay_derivative}
\end{equation}
 \end{prop}
The proof of this proposition is based on pointwise methods as in \cite{cordoba2009maximum,cordoba2014confined}. In particular, we will obtain the inequality 
$$
\ddt \av{h(t)}_{W^{1,\infty}}< 0\;\text{ a.e. in $t$},
$$
for small enough initial data. Integrating in time will lead to the maximum principle for this norm. Once the maximum principle is obtained, we can even obtain the exponential decay of the norm.

\subsection{Estimates in $ L^\infty $} \label{sec:estLinfty}
We consider now \eqref{eq:eveqh2}. With the notation introduced in \eqref{eq:auxiliary_variables} combined with the elementary identity $ 1-\cos\alpha = 2\sin^2\pare{\alpha/2} $, we can deduce that
$$
\partial_t h\pare{s} + \gamma\pare{s}\cdot   \ddt{M}= \ J_1 \pare{s} + J_2 \pare{s}  + J_3 \pare{s},
$$
where
\begin{equation}
\label{eq:J's}
\begin{aligned}
J_1 = & \  -\frac{ 1}{8\pi} \int \log \pare{4r\pare{r-\theta} \sin^2\pare{\alpha/2} + \theta^2} \partial_\alpha^2\bra{\pare{r-\theta} \cos\alpha } \dd \alpha, \\
J_2 = & \ \frac{1}{4\pi} \int \frac{\bra{2r\sin^2\pare{\alpha / 2} + \theta \cos\alpha} \bra{2r\sin^2\pare{\alpha/2}-\theta}}{4r\pare{r-\theta} \sin^2\pare{\alpha/2} +\theta^2} \pare{\partial_\alpha^2 \theta +\pare{r-\theta}} \dd \alpha , \\
J_3 = & \ \frac{r}{2\pi} \int \frac{\bra{2r\sin^2\pare{\alpha/2} + \theta\cos\alpha}\sin(\alpha)}{4r\pare{r-\theta}\sin^2\pare{\alpha / 2} + \theta^2} \ \partial_\alpha \theta \dd \alpha. 
\end{aligned}
\end{equation}

\subsubsection*{Bound of $ J_1 $}
Let us decompose
 \begin{equation*}
 J_1 = J_{1, \gamma} + J_{1, 1} + J_{1, 2} + J_{1, 3}, 
 \end{equation*}
 where  
 \begin{align*}
 J_{1, \gamma} & = \frac{ 1}{8\pi} \int \log \pare{4r\pare{r-\theta} \sin^2\pare{\alpha/2} + \theta^2} \cos\alpha \  \dd \alpha
 , \\
 J_{1, 1} & =   \frac{1}{8 \pi} \int \frac{ 4r \partial_\alpha\theta \ \sin^2\pare{\alpha/2} }{4r\pare{r-\theta}\sin^2\pare{\alpha/2}+ \theta^2} \  \partial_\alpha\eta \ \dd \alpha , \\
  J_{1, 2} & = -\frac{1}{8\pi} \int \frac{ 2 r\pare{r-\theta}\sin(\alpha)   }{4r\pare{r-\theta}\sin^2\pare{\alpha/2}+ \theta^2} \  \partial_\alpha\eta \  \dd \alpha ,  \\
  J_{1, 3} & =  -\frac{1}{8 \pi} \int \frac{2\theta\partial_\alpha\theta}{4r\pare{r-\theta}\sin^2\pare{\alpha/2}+ \theta^2} \  \partial_\alpha \eta \  \dd \alpha . 
 \end{align*}
 
We analyze now the term $ J_{1, \gamma} $. We use the Taylor expansion of the logarithm to find that
\begin{align*}
\log\pare{4r\pare{r-\theta}\sin^2\pare{\alpha/2} + \theta^2} &= \log\pare{1-\frac{\theta}{r}+\frac{\theta^2}{4r^2\sin^2\pare{\alpha/2}}} + \log\pare{4r^2\sin^2\pare{\alpha/2}}, \\
& \leq \log\pare{4r^2\sin^2\pare{\alpha/2}} - \frac{\theta}{r} + \cP \pare{\av{h'}_{ L^\infty }^2}.
\end{align*}
As a consequence of the previous estimate we find the following bound for  $ J_{1, \gamma} \pare{\overline{s^0_t}} $: 
\begin{equation}
\label{eq:J1gamma}
J_{1, \gamma}\pare{\overline{s^0_t}} \leq  \frac{1}{4}\frac{1}{2\pi} \int \log\pare{4 \br^2  \sin^2\pare{\alpha/2}} \cos\alpha \ \dd \alpha   -\frac{1}{4}\frac{1}{2\pi}  \int \bt\cos\alpha \dd \alpha    + \cP\pare{\av{h}_{W^{1, \infty}}}  \av{h}_{W^{1, \infty}} \av{h'}_{L^{\infty}}, 
\end{equation}
so that in \eqref{eq:J1gamma} we isolate static, linear end nonlinar contributions to the evolution of $|h|_{L^\infty}$.

Since the term $ J_{1, 1}\pare{\overline{s^0_t}} $ is not a singular integral it can be bounded straightforwardly and we conclude that
\begin{equation}\label{eq:J11}
J_{1, 1} \pare{\overline{s^0_t}} \leq \cP\pare{\av{h}_{W^{1, \infty}}}\av{h}_{W^{1, \infty}} \av{h'}_{L^{\infty}}. 
\end{equation}

 We start studying the term $ J_{1, 2} $. We remark that
 \begin{align*}
 J_{1, 2} = & \  - \frac{1}{4} \frac{1}{2 \pi}\int \pare{ 1 - \frac{ \theta^2 }{4r\pare{r-\theta}\sin^2\pare{\alpha/2}+ \theta^2} } \ \cot \pare{\alpha / 2}  \partial_\alpha\eta \  \dd \alpha  , \\
  = & \  J_{1,2, \RN{1}} +  J_{1,2, \RN{2}}  . 
 \end{align*}
The smallness of $h$ and the positivity of $\bt $ allow us to deduce that
 \begin{equation*}
 J_{1, 2, \RN{2}} \pare{\overline{s^0_t}}=\frac{1}{8 \pi}\int \frac{ \bar{\theta} }{4\bar{r}\pare{\bar{r}-\bar{\theta}}\sin^2\pare{\alpha/2}+ \theta^2} \ \frac{\bar{\theta}}{\tan \pare{\alpha / 2}}  \partial_\alpha\bar{\eta} \  \dd \alpha   \leq \av{h}_{W^{1, \infty}} \cP\pare{\av{h}_{W^{1, \infty}}}\int \frac{\bt}{2\sin^2\pare{\alpha/2}}\dd \alpha .
 \end{equation*} 
 
 Thus, considering that 
 $$ 
 \eta = \theta +2h\pare{s-\alpha}\sin^2\pare{\alpha / 2}, 
 $$ we obtain  
\begin{equation} \label{eq:J12}
J_{1, 2} \pare{\overline{s^0_t}} \leq
  - \frac{1}{4}  \frac{1}{2\pi } \int \frac{\bt}{2\sin^2\pare{\alpha/2}} \dd \alpha  - \frac{1}{4}  \frac{1}{2\pi } \int h\pare{\overline{s^0_t}-\alpha}\dd\alpha
 + \av{h}_{W^{1, \infty}} \cP\pare{\av{h}_{W^{1, \infty}}}  \int \frac{\bt}{2\sin^2\pare{\alpha/2}} \dd \alpha.
\end{equation}

We rearrange the term $ J_{1, 3}\pare{\overline{s^0_t}} $
\begin{align*}
  J_{1, 3}  &=  - \frac{1}{4} \frac{1}{2 \pi} \int \frac{\theta}{2\sin^2(\alpha/2)} \frac{4\sin^2(\alpha/2)\partial_\alpha\theta}{4r\pare{r-\theta}\sin^2\pare{\alpha/2}+ \theta^2} \  \partial_\alpha \eta \  \dd \alpha \\
\end{align*}

The last term $ J_{1, 3}\pare{\overline{s^0_t}} $ can easily be controlled as follows:
\begin{equation}
\label{eq:J13}
\av{J_{1, 3}\pare{\overline{s^0_t}}} \leq \av{h}_{W^{1, \infty}} \cP\pare{\av{h}_{W^{1, \infty}}}  \int \frac{\bt}{2\sin^2\pare{\alpha /2}}. 
\end{equation}

We combine the estimates \eqref{eq:J1gamma}, \eqref{eq:J11}, \eqref{eq:J12} and \eqref{eq:J13} and obtain the bound 
\begin{multline}\label{eq:J1}
J_1\pare{\overline{s^0_t}} \leq  - \frac{1}{4}  \frac{1}{2\pi }  \int \frac{\bt }{2\sin^2\pare{\alpha/2}} \dd \alpha  - \frac{1}{4}  \frac{1}{2\pi }  \int h\pare{\overline{s^0_t}-\alpha}\dd \alpha -\frac{1}{4}\frac{1}{2\pi}  \int \bt\cos\alpha \dd \alpha  
   \\
 + \av{h}_{W^{1, \infty}} \cP\pare{\av{h}_{W^{1, \infty}}}  \int \frac{\bt}{2\sin^2\pare{\alpha/2}} \dd \alpha +   \cP\pare{\av{h}_{W^{1, \infty}}} \av{h}_{W^{1, \infty}}\av{h'}_{L^{\infty}}  \\
  + \frac{1}{4} \frac{1}{2 \pi} \int \log\pare{4 \br^2  \sin^2\pare{\alpha/2}} \cos\alpha \ \dd \alpha.
\end{multline}

\subsubsection*{Bound of $ J_2 $} \label{sec:J2}

We study now the term $ J_2 $. In order to do that we decompose it as
\begin{equation*}
J_2 = J_{2, \gamma} + J_{2, 1} + J_{2, 2} , 
\end{equation*}
where 
\begin{align*}
J_{2, \gamma} & =  \ \frac{1}{4\pi}  \int \frac{\bra{2r\sin^2\pare{\alpha / 2} + \theta \cos\alpha} \bra{2r\sin^2\pare{\alpha/2}-\theta}}{4r\pare{r-\theta} \sin^2\pare{\alpha/2} +\theta^2} \ r \  \dd \alpha , \\
J_{2, 1} & = 
\ -\frac{1}{4\pi}  \int \frac{\bra{2r\sin^2\pare{\alpha / 2} + \theta \cos\alpha} \bra{2r\sin^2\pare{\alpha/2}-\theta}}{4r\pare{r-\theta} \sin^2\pare{\alpha/2} +\theta^2} \ \theta \  \dd \alpha ,
\\
J_{2, 2} & = 
\ - \frac{1}{4\pi}  \int \partial_\alpha \pare{ \frac{\bra{2r\sin^2\pare{\alpha / 2} + \theta \cos\alpha} \bra{2r\sin^2\pare{\alpha/2}-\theta}}{4r\pare{r-\theta} \sin^2\pare{\alpha/2} +\theta^2} } \ \partial_\alpha \theta \  \dd \alpha .
\end{align*}

We start analyzing the term $ J_{2, \gamma} $. We write this term as
\begin{equation*}
J_{2, \gamma} =  \frac{1}{4\pi} \int \frac{4r^2 \sin^2\pare{\alpha / 2}}{4r\pare{r-\theta} \sin^2\pare{\alpha/2} +\theta^2} \bra{ \pare{r-\theta}\sin^2\pare{\alpha / 2} - \frac{\theta^2\cos\pare{\alpha}}{4r\sin^2\pare{\alpha/2}}}  \dd \alpha .
\end{equation*}
We observe that
\begin{align*}
\frac{ 4r^2 \sin^2\pare{\alpha / 2}}{4r^2\sin^2\pare{\alpha/2} -4r\theta \sin^2\pare{\alpha/2} +\theta^2}  & = \frac{ 1}{1-\frac{\theta }{r } +\frac{\theta^2}{4r^2\sin^2\pare{\alpha/2} }}  , \\
&= 1 +\frac{ \frac{\theta }{r } -\frac{\theta^2}{4r^2\sin^2\pare{\alpha/2} }}{1-\frac{\theta }{r } +\frac{\theta^2}{4r^2\sin^2\pare{\alpha/2} }}\\
&=1 +\sum_{\ell=1}^\infty\left(\frac{\theta }{r } -\frac{\theta^2}{4r^2\sin^2\pare{\alpha/2} }\right)^\ell\\
\frac{\theta^2\cos\pare{\alpha}}{4r \sin^2\pare{\alpha/2}} & \leq C\av{h'}_{L^\infty}^2,
\end{align*}
where the convergence of the series is ensured due to the smallness of $h$ in $W^{1,\infty}$ and $C$ is a universal constant that may change from line to line. Then, noticing that
$$
\left(1+\frac{\theta }{r }\right)(r-\theta)=\frac{r^2-\theta^2}{r},
$$
we deduce that
\begin{equation}
\label{eq:J2gamma}
J_{2, \gamma}\pare{\overline{s^0_t}} \leq  \frac{\br }{2} \frac{1}{2\pi} \int \sin^2\pare{\alpha / 2} \dd \alpha 
 +\cP\pare{\av{h}_{W^{1, \infty}}}\av{h}_{W^{1, \infty}} \av{h'}_{L^\infty} .  
\end{equation}

Next we study the term $ J_{2, 1} $. Using similar computations as the ones performed for the term $ J_{2, \gamma} $ we can reformulate it as
\begin{align*}
J_{2, 1}= & - \frac{1}{4\pi}  \int \frac{4r^2 \sin^2\pare{\alpha / 2}}{4r\pare{r-\theta} \sin^2\pare{\alpha/2} +\theta^2} \bra{ \pare{1-\frac{\theta}{r}}\theta \sin^2\pare{\alpha / 2} - \frac{\theta^3\cos\pare{\alpha/2}}{4r^2\sin^2\pare{\alpha/2}}}  \dd \alpha. 
\end{align*}
From the previous expression we can deduce the estimate
\begin{equation}\label{eq:J21}
\begin{aligned}
J_{2, 1} \pare{\overline{s^0_t}} \leq & \   -\frac{1}{4\pi} \int \bt \sin^2\pare{\alpha / 2} \dd \alpha  
 + \cP\pare{\av{h}_{W^{1, \infty}}} \av{h}_{W^{1, \infty}} \av{h'}_{L^\infty} , \\
\end{aligned}
\end{equation}

At last we study the term $ J_{2, 2} $. We decompose it as
\begin{equation*}
J_{2, 2} = J_{2, 2, \RN{1}} + J_{2, 2, \RN{2}} + J_{2, 2, \RN{3}}, 
\end{equation*}
where
\begin{align*}
J_{2, 2, \RN{1}} = & \   \frac{1}{4\pi}  \int \frac{-4r \partial_\alpha \theta \sin^2\pare{\alpha/2} + 2r\pare{r-\theta}\sin(\alpha) + 2\theta\partial_\alpha \theta}{\pare{4r\pare{r-\theta}\sin^2\pare{\alpha / 2} + \theta^2 }^2} \pare{2r\sin^2\pare{\alpha / 2} + \theta \cos\alpha}\pare{2r\sin^2\pare{\alpha / 2} - \theta } \partial_\alpha \theta \ \dd \alpha , \\
J_{2, 2, \RN{2}} = & \  - \frac{1}{4\pi} \int  \frac{ 
\pare{r\sin(\alpha) + \partial_\alpha \theta \cos\alpha - \theta \sin \alpha}\pare{2r\sin^2\pare{\alpha / 2} - \theta }
}{4r\pare{r-\theta}\sin^2\pare{\alpha / 2} + \theta^2 } \ \partial_\alpha \theta \ \dd \alpha , \\
J_{2, 2, \RN{3}} = & \  - \frac{1}{4\pi}  \int \frac{ \pare{2r\sin^2\pare{\alpha / 2} + \theta \cos\alpha}\pare{ r\sin(\alpha) - \partial_\alpha \theta } }{4r\pare{r-\theta}\sin^2\pare{\alpha / 2} + \theta^2 } \ \partial_\alpha \theta \ \dd \alpha . 
\end{align*}

Let us consider at first $ J_{2, 2, \RN{1}}\pare{\overline{s^0_t}} $. 
Recalling
$$
\frac{\bra{2r\sin^2\pare{\alpha / 2} + \theta \cos\alpha} \bra{2r\sin^2\pare{\alpha/2}-\theta}}{4r\pare{r-\theta} \sin^2\pare{\alpha/2} +\theta^2}=\frac{4r^2 \sin^2\pare{\alpha / 2}}{4r\pare{r-\theta} \sin^2\pare{\alpha/2} +\theta^2} \bra{ \pare{1-\frac{\theta}{r}}\sin^2\pare{\alpha / 2} - \frac{\theta^2\cos\pare{\alpha}}{4r^2\sin^2\pare{\alpha/2}}},
$$
using similar computations as before, we can compute 
\begin{align*}
\frac{ \pare{2r\sin^2\pare{\alpha / 2} + \theta \cos\alpha}\pare{2r\sin^2\pare{\alpha / 2} - \theta } 	}{(4r\pare{r-\theta}\sin^2\pare{\alpha / 2} + \theta^2)^2 }
 &= 
\left[1 +\sum_{\ell=1}^\infty\left(\frac{\theta }{r } -\frac{\theta^2}{4r^2\sin^2\pare{\alpha/2} }\right)^\ell\right]\bra{ \pare{1-\frac{\theta}{r}} - \frac{\theta^2\cos\pare{\alpha}}{4r^2\sin^4\pare{\alpha/2}}}\\
&\quad\times\frac{1}{4r\pare{r-\theta} + \frac{\theta^2}{\sin^2\pare{\alpha / 2}}}.
\end{align*}

Then, using the positivity of $\bar{\theta}$, we have that
\begin{equation*}
 J_{2, 2, \RN{1}}\pare{\overline{s^0_t}} 
 \leq  -  \frac{1}{2}\frac{1}{4\pi}  \int \bt \ \cos\alpha \  \dd\alpha  +  \av{h}_{W^{1, \infty}}\cP\pare{\av{h}_{W^{1, \infty}} }  \int \frac{\bt}{2\sin^2\pare{\alpha / 2}} \dd\alpha  + \cP\pare{\av{h}_{W^{1, \infty}} }  \av{h}_{W^{1, \infty}} \av{h'}_{L^\infty} . 
\end{equation*}
Very similar bounds hold for  $ J_{2, 2, \RN{2}}\pare{\overline{s^0_t}} $ and  $ J_{2, 2, \RN{3}}\pare{\overline{s^0_t}} $, namely
\begin{align*}
J_{2, 2, \RN{2}}\pare{\overline{s^0_t}} \leq& \   \frac{1}{2}\frac{1}{4 \pi} \int  \bt \ \cos \alpha \ \dd \alpha  +  \av{h}_{W^{1,\infty}} \cP \pare{ \av{h}_{W^{1, \infty}}} \int \frac{\bt}{2\sin^2\pare{\alpha / 2}} \dd \alpha  + \cP\pare{\av{h}_{W^{1, \infty}}} \av{h}_{W^{1, \infty}} \av{h'}_{L^\infty} , \\
J_{2, 2, \RN{3}}\pare{\overline{s^0_t}} \leq & \  \  \frac{1}{2}\frac{1}{4 \pi} \int  \bt \ \cos \alpha \ \dd \alpha  +  \av{h}_{W^{1,\infty}} \cP \pare{ \av{h}_{W^{1, \infty}}} \int \frac{\bt}{2\sin^2\pare{\alpha / 2}} \dd \alpha  + \cP\pare{\av{h}_{W^{1, \infty}}} \av{h}_{W^{1, \infty}} \av{h'}_{L^\infty}  , 
\end{align*}
and so we obtain that
\begin{equation}\label{eq:J22}
J_{2, 2}\pare{\overline{s^0_t}} \leq  \   \frac{1}{2}\frac{1}{4 \pi}  \int  \bt \ \cos \alpha \ \dd \alpha  +  \av{h}_{W^{1,\infty}} \cP \pare{ \av{h}_{W^{1, \infty}}} \int \frac{\bt}{2\sin^2\pare{\alpha / 2}} \dd \alpha   + \cP\pare{\av{h}_{W^{1, \infty}}} \av{h}_{W^{1, \infty}} \av{h'}_{L^\infty}  . 
\end{equation}

We combine the estimates \eqref{eq:J2gamma}, \eqref{eq:J21} and \eqref{eq:J22} and deduce that
\begin{multline}\label{eq:J2}
J_2\pare{\overline{s^0_t}} \leq \frac{\br }{2} \frac{1}{2\pi} \int \sin^2\pare{\alpha / 2} \dd \alpha -\frac{1}{4\pi} \int \bt \sin^2\pare{\alpha / 2} \dd \alpha  + \frac{1}{4}\frac{1}{2 \pi}  \int  \bt \ \cos \alpha \ \dd \alpha  \\
+ \av{h}_{W^{1,\infty}} \cP \pare{ \av{h}_{W^{1, \infty}}} \int \frac{\bt}{2\sin^2\pare{\alpha / 2}} \dd \alpha +  \cP\pare{\av{h}_{W^{1, \infty}}} \av{h}_{W^{1, \infty}} \av{h'}_{L^\infty}  . 
\end{multline}

\subsubsection*{Bound of $ J_3 $}
We decompose $ J_3 $ as
\begin{equation*}
J_3 = J_{3 , 1} + J_{3, 2}, 
\end{equation*}
with
\begin{align*}
J_{3 , 1} & = \frac{r}{2\pi}  \int \frac{2r \ \sin^2 \pare{\alpha / 2} \sin(\alpha)}{4r \pare{r-\theta}\sin^2\pare{\alpha / 2} + \theta^2} \ \partial_\alpha \theta \ \dd \alpha ,  \\
J_{3 , 2} & = \frac{r}{2\pi}  \int \frac{ \theta \cos(\alpha)\sin(\alpha)}{4r \pare{r-\theta}\sin^2\pare{\alpha / 2} + \theta^2} \ \partial_\alpha \theta \ \dd \alpha. 
\end{align*}

Using a Taylor expansion we know that
\begin{equation}\label{eq:J31}
J_{3 , 1}\pare{\overline{s^0_t}} = \frac{1}{2\pi} \int \left[1 +\sum_{\ell=1}^\infty\left(\frac{\theta }{r } -\frac{\theta^2}{4r^2\sin^2\pare{\alpha/2} }\right)^\ell\right]\frac{\sin(\alpha)\partial_\alpha \theta}{2} \ \dd\alpha ,
\end{equation}
from which we easily deduce the bound
\begin{equation*}
J_{3 , 1}  \pare{\overline{s^0_t}} \leq  - \frac{1}{2} \frac{1}{2\pi} \int \bt \cos(\alpha) \ \dd\alpha  + \av{h}_{W^{1, \infty}} \cP\pare{\av{h}_{W^{1, \infty}}} \int \frac{\bt}{2\sin^2\pare{\alpha / 2}} \dd \alpha. 
\end{equation*}

The term $ J_{3, 2} $ can be handled similarly and, in fact,
\begin{equation*}
J_{3 , 2 }  \pare{\overline{s^0_t}} \leq \av{h}_{W^{1, \infty}} \cP\pare{\av{h}_{W^{1, \infty}}} \int \frac{\bt}{2\sin^2\pare{\alpha / 2}} \dd \alpha.
\end{equation*}
Collecting the previous estimates we deduce the desired bound for $ J_3 $
\begin{equation}\label{eq:J3}
J_{3}  \pare{\overline{s^0_t}} \leq    - \frac{1}{2} \frac{1}{2\pi} \int \bt \cos(\alpha) \ \dd\alpha  + \av{h}_{W^{1, \infty}} \cP\pare{\av{h}_{W^{1, \infty}}} \int \frac{\bt}{2\sin^2\pare{\alpha / 2}} \dd \alpha. 
\end{equation}

\subsubsection*{The equation for the evolution of $|h|_{L^\infty}$}
We sum the inequalities \eqref{eq:J1}, \eqref{eq:J2} and \eqref{eq:J3} and obtain the bound 
\begin{multline}\label{eq:PD1}
J_{1}  \pare{\overline{s^0_t}} + J_{2}  \pare{\overline{s^0_t}} + J_{3}  \pare{\overline{s^0_t}} \leq \frac{1}{4 \pi} \int \log\pare{2 \br  \sin \pare{\alpha/2}} \cos(\alpha) \ \dd \alpha + \frac{\br }{4 \pi}  \int \sin^2\pare{\alpha / 2} \dd \alpha  \\
-\frac{1}{4}\frac{1}{2\pi} \bra{1-\av{h}_{W^{1, \infty}}\cP\pare{\av{h}_{W^{1, \infty}}}} \int \frac{\bt}{2\sin^2\pare{\alpha / 2}} \dd\alpha
-\frac{1}{4} \frac{1}{2\pi} \int \bt\pare{1+\cos(\alpha)}\dd\alpha \\
-\frac{1}{4} \frac{1}{2\pi} \int h\pare{ \overline{s^0_t} -\alpha} \dd \alpha  + \cP\pare{\av{h}_{W^{1, \infty}}} \av{h}_{W^{1, \infty}} \av{h'}_{L^\infty} . 
\end{multline}
Using pointwise methods as in \cite{cordoba2009maximum,cordoba2014confined} we have that
$$
\ddt\max_{s\in\mathbb{S}^1}\set{h(s,t) }=\partial_t h\pare{ \overline{s^0_t} } \ \text{ a.e. }
$$

Recalling equation \eqref{eq:eveqh2}, we obtain that
\begin{multline*}
\ddt \max_{s\in\mathbb{S}^1}\{h(s,t)\} + \gamma\pare{\overline{s^0_t}} \cdot \ddt M \leq \frac{1}{4 \pi} \int \log\pare{2 \br  \sin \pare{\alpha/2}} \cos(\alpha) \ \dd \alpha + \frac{\br }{4 \pi}  \int \sin^2\pare{\alpha / 2} \dd \alpha  \\
-\frac{1}{4}\frac{1}{2\pi} \bra{1-\av{h}_{W^{1, \infty}}\cP\pare{\av{h}_{W^{1, \infty}}}}  \int \frac{\bt}{2\sin^2\pare{\alpha / 2}} \dd\alpha
-\frac{1}{4} \frac{1}{2\pi} \int \bt\pare{1+\cos(\alpha)}\dd\alpha \\
-\frac{1}{4} \frac{1}{2\pi} \int h\pare{ \overline{s^0_t} -\alpha} \dd \alpha  + \cP\pare{\av{h}_{W^{1, \infty}}} \av{h}_{W^{1, \infty}} \av{h'}_{L^\infty}. 
\end{multline*}
Since
\begin{align*}
\gamma\pare{ s } \cdot \ddt M  & = \frac{1}{4} \frac{1}{2\pi} \int h\pare{ s  - \alpha}\cos(\alpha) \ \dd\alpha , & \forall  \ s\in \bS^1, \\
 \frac{1}{4 \pi} \int \log\pare{2 \br  \sin \pare{\alpha/2}} \cos(\alpha) \ \dd \alpha + \frac{\br }{4 \pi}  \int \sin^2\pare{\alpha / 2} \dd \alpha& = \frac{h\pare{\overline{s^0_t}}}{4},  \\
-\frac{1}{4}\frac{1}{2\pi} \int \bt \ \dd\alpha -\frac{1}{4} \frac{1}{2\pi} \int h\pare{ \overline{s^0_t} -\alpha} \dd \alpha  & = -\frac{h\pare{\overline{s^0_t}}}{4}  ,  \\
-\frac{1}{4}\frac{1}{2\pi} \int \bt \cos(\alpha) \ \dd\alpha & = \frac{1}{4}\frac{1}{2\pi} \int h\pare{\overline{s^0_t} -\alpha}\cos(\alpha) \ \dd\alpha, 
\end{align*}
we obtain that the evolution equation for $\displaystyle\max_{s\in\mathbb{S}^1}\{h(s,t)\}$ can be estimated as
\begin{equation*}
\ddt h\pare{\overline{s^0_t}} \leq -\frac{1}{4}\frac{1}{2\pi} \bra{1-\av{h}_{W^{1, \infty}}\cP\pare{\av{h}_{W^{1, \infty}}}}  \int \frac{\bt}{2\sin^2\pare{\alpha / 2}} \dd\alpha 
+ \cP\pare{\av{h}_{W^{1, \infty}}} \av{h}_{W^{1, \infty}} \av{h'}_{L^\infty} .
\end{equation*}
We can perform the same computations for the quantity
$$
0<-\min_{s\in\mathbb{S}^1}\{h(s,t)\}=-h\pare{\underline{s^0_t}},
$$
and obtain the bound
\begin{equation*}
- \ddt \min_{s\in\mathbb{S}^1}\{h(s,t)\} \leq -\frac{1}{4}\frac{1}{2\pi} \bra{1-\av{h}_{W^{1, \infty}}\cP\pare{\av{h}_{W^{1, \infty}}}}  \int \frac{ - \ut}{2\sin^2\pare{\alpha / 2}} \dd\alpha 
+ \cP\pare{\av{h}_{W^{1, \infty}}} \av{h}_{W^{1, \infty}} \av{h'}_{L^\infty} .
\end{equation*}

As a consequence we have that the time-evolution of $ \av{h\pare{t}}_{L^\infty} = \max \set{ h\pare{\overline{s^0_t}, t} , \ - h\pare{\underline{s^0_t}, t} } $ is given by
\begin{multline}
\ddt \av{h}_{L^\infty} \leq -\frac{1}{4}\frac{1}{2\pi} \bra{1-\av{h}_{W^{1, \infty}}\cP\pare{\av{h}_{W^{1, \infty}}}} \max\set{\int \frac{\bt}{2\sin^2\pare{\alpha/2}}  \dd \alpha ,  \int \frac{ -\ut }{2\sin^2\pare{\alpha/2}}  \dd \alpha} \\
+ \cP\pare{\av{h}_{W^{1, \infty}}} \av{h}_{W^{1, \infty}} \av{h'}_{L^\infty} .\label{eq:control_h_in_Linfty}
\end{multline}

\subsection{Estimates in $ W^{1,\infty}$}  We have to find the evolution equation for $h'$. In order to do that we differentiate now $ J_1 $, integrate by parts in $\alpha$ and use
\begin{equation*}
\partial_\alpha^2 \bra{\cos \alpha \ \pare{1+h\pare{s-\alpha}}} = -\partial_\alpha \bra{\sin \alpha \ \pare{1+h\pare{s-\alpha}}} - \partial_\alpha \bra{\cos(\alpha) \ h'\pare{s-\alpha}}, 
\end{equation*}
 to obtain
\begin{align}
J_1'\pare{s} = & \  \frac{1}{2}\frac{1}{2\pi} \int \Bigg\{ \frac{\pare{1+h\pare{s}}h'\pare{s} + \pare{1+h\pare{s -\alpha}}h'\pare{s -\alpha} - \pare{h'\pare{s}\pare{1+h\pare{s-\alpha}} + \pare{1+h\pare{s}} h'\pare{s-\alpha}}\cos \alpha }{  
\pare{1 + h\pare{s}}^2 + \pare{1 + h\pare{s-\alpha}}^2  - 2 \pare{1 + h\pare{s}}\pare{1 + h\pare{s-\alpha}} \cos(\alpha)} \nonumber\\
& \qquad \qquad \times  \ \partial_\alpha \bra{ \big. \sin(\alpha) \pare{1 + h\pare{s-\alpha}} } \Bigg\} \dd \alpha \nonumber\\
& \ -  \frac{1}{2}\frac{1}{2\pi} \int   \Bigg\{ \frac{
-\pare{1+h\pare{s}}  h'\pare{s} +  h'\pare{s} \pare{1+h \pare{s-\alpha}} \cos \alpha -\pare{1+h\pare{s}}\pare{1+h\pare{s-\alpha}} \sin \alpha
 }{ 
\pare{1 + h\pare{s}}^2 + \pare{1 + h\pare{s-\alpha}}^2  - 2 \pare{1 + h\pare{s}}\pare{1 + h\pare{s-\alpha}} \cos(\alpha)} \nonumber\\
& \qquad \qquad \times  \ \partial_\alpha \bra{ \big. \cos \alpha \ h'\pare{s-\alpha} } \Bigg\} \dd \alpha .\label{eq:J1'derivation}
\end{align}
Using the trigonometric identity
$$
1-\cos(\alpha)=\sin^2(\alpha/2),
$$
taking a derivative of the term $ J_2 $ and using 
$$
\partial_s\pare{ h''\pare{s-\alpha}  - 1 - h\pare{s-\alpha}  } =-\partial_\alpha\pare{ h''\pare{s-\alpha}  - 1 - h\pare{s-\alpha}  } 
$$
we find that
\begin{equation*}
\begin{aligned}
J_2'\pare{s} = & \ \frac{1}{4\pi}  \int  \Bigg\{  \frac{ \bra{h' \pare{s}- h'\pare{s-\alpha} \cos \alpha} \bra{\pare{1+h\pare{s}} \cos \alpha - \pare{1+h\pare{s-\alpha}}}     }{ 
\pare{1 + h\pare{s}}^2 + \pare{1 + h\pare{s-\alpha}}^2  - 2 \pare{1 + h\pare{s}}\pare{1 + h\pare{s-\alpha}} \cos(\alpha) } \\
& \qquad +\frac{\bra{1+h\pare{s}- \pare{1+h\pare{s-\alpha}} \cos \alpha} \bra{ h'\pare{s} \cos \alpha - h'\pare{s-\alpha}}  }{ 
\pare{1 + h\pare{s}}^2 + \pare{1 + h\pare{s-\alpha}}^2  - 2 \pare{1 + h\pare{s}}\pare{1 + h\pare{s-\alpha}} \cos(\alpha) }\Bigg\}\ \pare{ h''\pare{s-\alpha}  - 1 - h\pare{s-\alpha}  }  \ \dd \alpha \\
& \ - \frac{1}{2\pi}  \int \Bigg\{ \frac{ \bra{1+h\pare{s}- \pare{1+h\pare{s-\alpha}} \cos \alpha} \bra{\pare{1+h\pare{s}} \cos \alpha - \pare{1+h\pare{s-\alpha}}}   }{\pare{ 
\pare{1 + h\pare{s}}^2 + \pare{1 + h\pare{s-\alpha}}^2  - 2 \pare{1 + h\pare{s}}\pare{1 + h\pare{s-\alpha}} \cos(\alpha) }^{2}} \\
& \qquad \times \bra{\pare{1+h\pare{s}}h'\pare{s} + \pare{1+h\pare{s-\alpha}}h'\pare{s-\alpha} - \pare{ \pare{1+h\pare{s}}h'\pare{s-\alpha} + \pare{1+h\pare{s-\alpha}}h'\pare{s} } \cos \alpha  } \Bigg\} \\
& \qquad \times  \pare{ h''\pare{s-\alpha}  - 1 - h\pare{s-\alpha}  }   \ \dd \alpha \\
& \ - \frac{1}{4\pi}  \int    \frac{ \bra{1+h\pare{s}- \pare{1+h\pare{s-\alpha}} \cos \alpha} \bra{\pare{1+h\pare{s}} \cos \alpha - \pare{1+h\pare{s-\alpha}}}   }{ 
\pare{1 + h\pare{s}}^2 + \pare{1 + h\pare{s-\alpha}}^2  - 2 \pare{1 + h\pare{s}}\pare{1 + h\pare{s-\alpha}} \cos(\alpha) } \ \partial_\alpha\pare{ h''\pare{s-\alpha}  - 1 - h\pare{s-\alpha}  } \dd \alpha . 
\end{aligned}
\end{equation*}
We can now integrate by parts in $\alpha$ and obtain that
\begin{align}
J_{2}'  \pare{s} = & \  \frac{1}{4\pi}  \int  \Bigg\{  \frac{ \bra{  h'\pare{s}  + \pare{1+h\pare{s-\alpha}} \sin(\alpha) } \bra{\pare{1+h\pare{s}} \cos \alpha - \pare{1+h\pare{s-\alpha}}}     }{  
\pare{1 + h\pare{s}}^2 + \pare{1 + h\pare{s-\alpha}}^2  - 2 \pare{1 + h\pare{s}}\pare{1 + h\pare{s-\alpha}} \cos(\alpha) } \nonumber\\
& \qquad +\frac{\bra{1+h\pare{s}- \pare{1+h\pare{s-\alpha}} \cos \alpha} \bra{ h'\pare{s}\cos \alpha  - \pare{ 1 + h\pare{s} }\sin(\alpha) }  }{ 
\pare{1 + h\pare{s}}^2 + \pare{1 + h\pare{s-\alpha}}^2  - 2 \pare{1 + h\pare{s}}\pare{1 + h\pare{s-\alpha}} \cos(\alpha)}\ \pare{ h''\pare{s-\alpha}  - 1 - h\pare{s-\alpha}  } \Bigg\} \ \dd \alpha \nonumber\\
& \ - \frac{1}{2 \pi}  \int \Bigg\{ \frac{ \bra{1+h\pare{s}- \pare{1+h\pare{s-\alpha}} \cos \alpha} \bra{\pare{1+h\pare{s}} \cos \alpha - \pare{1+h\pare{s-\alpha}}}   }{\pare{ 
\pare{1 + h\pare{s}}^2 + \pare{1 + h\pare{s-\alpha}}^2  - 2 \pare{1 + h\pare{s}}\pare{1 + h\pare{s-\alpha}} \cos(\alpha) }^{2}} \nonumber\\
& \qquad \times \bra{ \pare{1+h\pare{s}}h'\pare{s} -\pare{1-h\pare{s-\alpha}}h'\pare{s}\cos \alpha + \pare{1+h\pare{s}}\pare{1+h\pare{s-\alpha}} \sin(\alpha) } \nonumber\\
& \qquad \times  \pare{ h''\pare{s-\alpha}  - 1 - h\pare{s-\alpha}  }  \Bigg\} \ \dd \alpha  .\label{eq:J2'derivation}
\end{align}
A similar procedure can be used in order to compute $ J_3' $. By doing this we obtain
\begin{equation}\label{eq:J3'derivation}
\begin{aligned}
J'_3 \pare{s} = & \  \frac{1}{2 \pi}  \int    \frac{ h' \pare{s} \bra{1+h\pare{s}- \pare{1+h\pare{s-\alpha}} \cos \alpha}  \sin \alpha     + \pare{1+h\pare{s}} h' \pare{s}   \sin \alpha  }{
\pare{1 + h\pare{s}}^2 + \pare{1 + h\pare{s-\alpha}}^2  - 2 \pare{1 + h\pare{s}}\pare{1 + h\pare{s-\alpha}} \cos(\alpha) }  h'\pare{s-\alpha}  \dd \alpha \\
 & \ + \frac{1}{2 \pi}  \int    \frac{ \pare{1+h\pare{s}}^2 \cos(\alpha)  -\pare{1+h\pare{s}}\pare{1+h\pare{s-\alpha}}\pare{\cos^2\alpha-\sin^2\alpha} }{
\pare{1 + h\pare{s}}^2 + \pare{1 + h\pare{s-\alpha}}^2  - 2 \pare{1 + h\pare{s}}\pare{1 + h\pare{s-\alpha}} \cos(\alpha) }  h'\pare{s-\alpha}  \dd \alpha \\
& \ -  \frac{1}{2 \pi} \int    \frac{  \pare{1+h\pare{s}} \bra{1+h\pare{s}- \pare{1+h\pare{s-\alpha}} \cos \alpha}  \sin \alpha    }{\pare{ 
\pare{1 + h\pare{s}}^2 + \pare{1 + h\pare{s-\alpha}}^2  - 2 \pare{1 + h\pare{s}}\pare{1 + h\pare{s-\alpha}} \cos(\alpha) }^2 } \\
& \qquad\quad \times   \pare{2\pare{1+h\pare{s}} h'\pare{s} + 2 \pare{1+h\pare{s-\alpha}}\pare{ \pare{1+h\pare{s}}\sin(\alpha) -h'\pare{s} \cos(\alpha) }}h'\pare{s-\alpha}  \dd \alpha . 
\end{aligned}
\end{equation}
We combine equations \eqref{eq:J1'derivation}, \eqref{eq:J2'derivation} and \eqref{eq:J3'derivation} with \eqref{eq:eveqh2} in order to obtain the evolution equation for $ h' $:
\begin{equation*}
\partial_t h'\pare{s} +   \gamma'\pare{s} \cdot \ddt M = \sum_{j=1}^7 \cJ_j\pare{s} , 
\end{equation*}
where
\begin{equation}
\label{eq:cJ's}
\begin{aligned}
 \cJ_1 = & \ \frac{1}{2}\frac{1}{2\pi} \int  \frac{r r' + \pare{r-\theta} \pare{r-\theta}' - \pare{r' \pare{r-\theta} + r \pare{r-\theta}'}\cos \alpha }{ 4r\pare{r-\theta}\sin^2\pare{\alpha / 2} + \theta^2} \ \partial_\alpha \bra{  \sin(\alpha) \pare{r-\theta} }  \dd \alpha \\
 \cJ_2 = & \  -  \frac{1}{2}\frac{1}{2\pi} \int    \frac{
-r  r' +  r' \pare{r-\theta} \cos \alpha -r \pare{r-\theta} \sin \alpha
 }{ 4r\pare{r-\theta}\sin^2\pare{\alpha / 2} + \theta^2}  \partial_\alpha \bra{ \big. \cos \alpha \ \pare{r-\theta}' }  \dd \alpha
 \\
 \cJ_3 = & \  \frac{1}{4\pi}  \int  \set{  \frac{ \bra{  r'  + \pare{r-\theta}\sin(\alpha) } \bra{r  \cos \alpha - \pare{r-\theta}}     }{ 4r\pare{r-\theta}\sin^2\pare{\alpha / 2} + \theta^2 } +\frac{\bra{r - \pare{r-\theta} \cos \alpha} \bra{ r' \cos \alpha  - r \sin(\alpha) }  }{ 4r\pare{r-\theta}\sin^2\pare{\alpha / 2} + \theta^2}}\\
 & \qquad\quad\times\ \bra{ \pare{r-\theta}''  - \pare{r-\theta}  }  \ \dd \alpha \\
\end{aligned}
\end{equation}
and
\begin{equation}
\label{eq:cJ's2}
\begin{aligned}
 \cJ_4 = & \  - \frac{1}{2\pi}  \int  \set{ \frac{ \bra{r- \pare{r-\theta} \cos \alpha} \bra{ r \cos \alpha - \pare{r-\theta}}   }{ \pare{4r\pare{r-\theta}\sin^2\pare{\alpha / 2} + \theta^2}^{2}} 
 \bra{ r r' -\pare{r-\theta}r'\cos \alpha + r \pare{r-\theta} \sin(\alpha) } }\\
 &\qquad\quad\times \bra{  \pare{r-\theta}''  - \pare{r-\theta}  }\ \dd \alpha  \\
 \cJ_5 = & \  \frac{1}{2 \pi}  \int    \frac{ r' \bra{r - \pare{r-\theta } \cos \alpha}  \sin \alpha     +  r r'   \sin \alpha   }{4r\pare{r-\theta}\sin^2\pare{\alpha / 2} + \theta^2 }  \pare{r-\theta}'   \dd \alpha \\
 \cJ_6 = &  \frac{1}{2 \pi}  \int    \frac{ r^2 \cos(\alpha)  -r\pare{r-\theta}\cos\pare{2\alpha} }{ 4 r \pare{r-\theta}\sin^2\pare{\alpha/2} + \theta^2 }  \pare{r-\theta}'  \dd \alpha , \\
 \cJ_7 = & \  -  \frac{1}{2 \pi} \int    \frac{  r  \bra{r - \pare{r-\theta} \cos \alpha}  \sin \alpha    }{\pare{4r\pare{r-\theta}\sin^2\pare{\alpha / 2} + \theta^2  }^2 }  \pare{2 r r' + 2 \pare{r-\theta }\pare{ r \sin(\alpha) -r' \cos(\alpha) }} \pare{r-\theta}'   \dd \alpha  .
\end{aligned}
\end{equation}
Let us remark that in \eqref{eq:cJ's} and \eqref{eq:cJ's2} the second-order terms are $ \pare{r-\theta}'' $ and $ \partial_\alpha\pare{r-\theta}' $. Using
\begin{equation*}
\pare{r-\theta}'' = h''\pare{s-\alpha} = -\partial_\alpha h'\pare{s-\alpha} = -\partial_\alpha \pare{r-\theta}', 
\end{equation*}
we will be able to integrate by parts. After this, only first derivatives of $h$ appear in the evolution equation for $ h' $. This will allow us to close the pointwise estimates in $ W^{1, \infty} $. 

Additionally to the notation introduced in \eqref{eq:notation_bar_0} we denote with
\begin{equation}
\label{eq:notation_bar_1}
\begin{aligned}
\btp = \btp \pare{\alpha} & = \theta' \pare{\overline{s^1_t}, \overline{s^1_t}-\alpha}, & \bep = \bep \pare{\alpha} & = \eta' \pare{\overline{s^1_t}, \overline{s^1_t}-\alpha, \alpha}, & \brp & = r' \pare{\overline{s^1_t}} , \\
\utp = \utp \pare{\alpha} & = \theta' \pare{\underline{s^1_t}, \underline{s^1_t}-\alpha}, & \uep = \uep \pare{\alpha} & = \eta' \pare{\underline{s^1_t}, \underline{s^1_t}-\alpha, \alpha}, & \urp & = r' \pare{\underline{s^1_t}}. 
\end{aligned}
\end{equation}

\subsubsection*{Bound of $ \cJ_1 $}
Let us remark at first that
\begin{multline}\label{eq:dec_cJ1_1}
\cJ_1   = \frac{1}{2}\frac{1}{2\pi} \int  \frac{ 2 \pare{ r r' + \pare{r-\theta} \pare{r-\theta}' }\sin^2\pare{\alpha / 2} }{ 4r\pare{r-\theta}\sin^2\pare{\alpha / 2} + \theta^2} \ \partial_\alpha \bra{  \sin(\alpha) \pare{r-\theta} }  \dd \alpha \\
-
\frac{1}{2}\frac{1}{2\pi} \int  \frac{\theta \theta' \cos \alpha }{ 4r\pare{r-\theta}\sin^2\pare{\alpha / 2} + \theta^2} \ \partial_\alpha \bra{  \sin(\alpha) \pare{r-\theta} }  \dd \alpha = \cJ_{1, 1} + \cJ_{1, 2}.
\end{multline}
The term $ \cJ_{1, 1} $ is not singular and, using that
$$
\int h'(s)\cos(\alpha)\dd \alpha=0
$$
we obtain
\begin{equation*}
\cJ_{1, 1} \pare{\overline{s^1_t}}  \leq  - \frac{1}{4} \frac{1}{2\pi } \int \theta' \ \cos(\alpha) \  \dd \alpha  +  \cP\pare{\av{h}_{W^{1, \infty}}}   \av{h}_{W^{1, \infty}} \av{h'}_{L^\infty}    . 
\end{equation*}
Using the Taylor expansion, we can write 
\begin{align*}
\cJ_{1, 2}&= -
\frac{1}{2\pi} \int  \frac{\theta }{4r^2\sin^2\pare{\alpha / 2} }\frac{4r^2\sin^2\pare{\alpha / 2} }{ 4r\pare{r-\theta}+ \frac{\theta^2}{\sin^2\pare{\alpha / 2} }} \ \theta' \cos (\alpha)\partial_\alpha \bra{  \sin(\alpha) \pare{r-\theta} }  \dd \alpha\\
&=-\frac{1}{2}\frac{1}{r^2}\frac{1}{2\pi} \int  \frac{\theta }{2\sin^2\pare{\alpha / 2} }\bigg{[}1 +\sum_{\ell=1}^\infty\left(\frac{\theta }{r } -\frac{\theta^2}{4r^2\sin^2\pare{\alpha/2} }\right)^\ell
\bigg{]} \ \theta' \cos (\alpha)\partial_\alpha \bra{  \sin(\alpha) \pare{r-\theta} }  \dd \alpha.
\end{align*}

Hence, evaluating this identity in $ \overline{s^1_t} $ we find the estimate
\begin{equation*}
\cJ_{1, 2} \pare{\overline{s^1_t}} \leq  \av{h}_{W^{1, \infty}} \cP\pare{\av{h}_{W^{1, \infty}}}  \int \frac{\btp}{2\sin^2\pare{\alpha / 2}} \dd \alpha.
\end{equation*}
Collecting both expressions we conclude that
\begin{equation}\label{eq:cJ1}
\cJ_{1} \pare{\overline{s^1_t}} \leq  - \frac{1}{4} \frac{1}{2\pi } \int \btp \ \cos(\alpha) \  \dd \alpha  + \av{h}_{W^{1, \infty}}  \cP\pare{\av{h}_{W^{1, \infty}}}  \int \frac{\btp}{2\sin^2\pare{\alpha / 2}} \dd \alpha   +  \cP\pare{\av{h}_{W^{1, \infty}}}   \av{h}_{W^{1, \infty}} \av{h'}_{L^\infty}   .
\end{equation}

\subsubsection*{Bound of $ \cJ_2 $  }
Let us integrate $ \cJ_2 $ by parts in the variable $ \alpha $, we obtain
\begin{multline}\label{eq:defcJ2}
\cJ_2 =  \   -  \frac{1}{2}\frac{1}{2\pi}  \int    \frac{
 -  r' \partial_\alpha \theta \cos(\alpha) -  \pare{r-\theta} r' \sin(\alpha) +r \partial_\alpha\theta \sin \alpha - r\pare{r-\theta} \cos(\alpha)
 }{ 4r\pare{r-\theta}\sin^2\pare{\alpha / 2} + \theta^2}  \  \eta '   \dd \alpha \\
 +  \frac{1}{2}\frac{1}{2\pi}  \int    \frac{
-r  r' +  r' \pare{r-\theta} \cos \alpha -r \pare{r-\theta} \sin \alpha
 }{ \pare{ 4r\pare{r-\theta}\sin^2\pare{\alpha / 2} + \theta^2 }^2 } \\
 \times  \pare{-4r\partial_\alpha \theta \sin^2\pare{\alpha  / 2} + 2r\pare{r-\theta}\sin(\alpha) + 2\theta\partial_\alpha \theta}  \   \eta '   \dd \alpha = \cJ_{2, 1} + \cJ_{2, 2}. 
\end{multline}

We now evaluate $ \cJ_{2, 1} $ at $ \overline{s^1_t} $ and isolate its linear part while bounding from above the nonlinear part to find that
\begin{multline}\label{eq:cJ21}
\cJ_{2, 1} \pare{\overline{s^1_t}} \leq   \frac{1}{2} \frac{1}{2\pi} \int \frac{\cos \alpha }{4\sin^2\pare{\alpha/2}} \ \bep \ \dd\alpha  \\
 +   \av{h}_{W^{1, \infty}} \cP\pare{\av{h}_{W^{1, \infty}}}  \int \frac{\btp}{2\sin^2\pare{\alpha/2}}  \ \dd\alpha + \cP\pare{\av{h}_{W^{1, \infty}}}\av{h}_{W^{1, \infty}} \av{h'}_{L^\infty} , 
\end{multline}

The same can be done for $ \cJ_{2, 2} $, and this gives
\begin{multline}\label{eq:cJ22}
\cJ_{2, 2} \pare{\overline{s^1_t}} \leq  - \frac{1}{2} \frac{1}{2\pi}  \int \frac{\cos^2\pare{\alpha/2} }{2\sin^2\pare{\alpha/2}} \ \bep\  \dd \alpha  \\
 +   \av{h}_{W^{1, \infty}} \cP\pare{\av{h}_{W^{1, \infty}}} \int \frac{\btp}{2\sin^2\pare{\alpha/2}}  \ \dd\alpha + \cP\pare{\av{h}_{W^{1, \infty}}}\av{h}_{W^{1, \infty}} \av{h'}_{L^\infty} , 
\end{multline}
so that \eqref{eq:cJ21} and \eqref{eq:cJ22} together with
$$
\frac{\cos(x)}{2}-\cos^2(x/2)=-\frac{1}{2},
$$
lead to
\begin{equation*}
\cJ_{2} \pare{\overline{s^1_t}} \leq  - \frac{1}{4} \frac{1}{2\pi}   \int \frac{\bep }{2\sin^2\pare{\alpha/2}}  \dd \alpha   \\
 + \av{h}_{W^{1, \infty}} \cP\pare{\av{h}_{W^{1, \infty}}} \int \frac{\btp}{2\sin^2\pare{\alpha/2}}  \ \dd\alpha + \cP\pare{\av{h}_{W^{1, \infty}}}\av{h}_{W^{1, \infty}} \av{h'}_{L^\infty} . 
\end{equation*}

We now use the following identity which holds for any $ s\in\bS^1 $
\begin{multline*}
\int \frac{\eta'}{2\sin^2\pare{\alpha / 2}} \dd\alpha = \int \frac{h'\pare{s} - h'\pare{s-\alpha}\cos(\alpha)}{2\sin^2\pare{\alpha / 2}} \dd\alpha \\
 = \int \frac{h'\pare{s} - h'\pare{s-\alpha}}{2\sin^2\pare{\alpha / 2}} \dd\alpha + \int h'\pare{s-\alpha}\dd \alpha = \int \frac{\theta ' }{2\sin^2\pare{\alpha / 2}} \dd\alpha. 
\end{multline*}
As a consequence,
\begin{equation}\label{eq:cJ2}
\cJ_{2} \pare{\overline{s^1_t}} \leq - \frac{1}{4} \frac{1}{2\pi} \bra{1- \av{h}_{W^{1, \infty}} \cP\pare{\av{h}_{W^{1, \infty}}}} \int \frac{\btp }{2\sin^2\pare{\alpha/2}}  \dd \alpha + \cP\pare{\av{h}_{W^{1, \infty}}}\av{h}_{W^{1, \infty}} \av{h'}_{L^\infty} . 
\end{equation}
\subsubsection*{Bound of $ \cJ_3 $}\label{sec:cJ3}
The terms $ \cJ_3 $ and $ \cJ_4 $ are more challenging to bound. Let us write
\begin{equation*}
\cJ_3 = \cJ_{3, 1} + \cJ_{3, 2}, 
\end{equation*}
where
\begin{align*}
 \cJ_{3, 1} = & \  \frac{1}{4\pi}  \int   \set{  \frac{ \bra{  r'  + \pare{r-\theta}\sin(\alpha)} \bra{r  \cos \alpha - \pare{r-\theta}}     }{ 4r\pare{r-\theta}\sin^2\pare{\alpha / 2} + \theta^2 } +\frac{\bra{r - \pare{r-\theta} \cos \alpha} \bra{ r' \cos \alpha  - r \sin(\alpha)  }  }{ 4r\pare{r-\theta}\sin^2\pare{\alpha / 2} + \theta^2}} \  \pare{r-\theta}''\ \dd \alpha , \\
 \cJ_{3, 2} = & \  - \frac{1}{4\pi}  \int   \set{  \frac{ \bra{  r'  + \pare{r-\theta}\sin(\alpha)} \bra{r  \cos \alpha - \pare{r-\theta}}     }{ 4r\pare{r-\theta}\sin^2\pare{\alpha / 2} + \theta^2 } +\frac{\bra{r - \pare{r-\theta} \cos \alpha} \bra{ r' \cos \alpha  - r \sin(\alpha)  }  }{ 4r\pare{r-\theta}\sin^2\pare{\alpha / 2} + \theta^2}}\  \pare{r-\theta} \ \dd \alpha, 
\end{align*}
and integrate by parts $ \cJ_{3, 1} $ in $ \alpha $. This gives
\begin{equation*}
\cJ_{3, 1}= \cJ_{3, 1, \RN{1}} + \cJ_{3, 1, \RN{2}}, 
\end{equation*}
where, using the identity $ \pare{r-\theta}'' = -\partial_\alpha\pare{r-\theta}' = \partial_\alpha \theta' $ we obtain that
\begin{align*}
\cJ_{3, 1, \RN{1}} = & -\frac{1}{4\pi} \int \partial_\alpha \set{\frac{ \bra{  r'  + \pare{r-\theta}\sin(\alpha)} \bra{\theta - 2r\sin^2\pare{\alpha / 2}}     }{ 4r\pare{r-\theta}\sin^2\pare{\alpha / 2} + \theta^2 }  } \theta' \ \dd \alpha, \\
\cJ_{3, 1, \RN{2}} = & -\frac{1}{4\pi}  \int \partial_\alpha \set{ \frac{\bra{2r\sin^2\pare{\alpha / 2} + \theta\cos(\alpha) } \bra{ r' \cos \alpha  - r \sin(\alpha)   }  }{ 4r\pare{r-\theta}\sin^2\pare{\alpha / 2} + \theta^2}  } \theta' \ \dd \alpha.
\end{align*}

As before, we use a Taylor expansion for the term $ \cJ_{3, 1, \RN{1}} $ and we obtain that

\begin{equation}\label{eq:cJ31I}
\cJ_{3, 1, \RN{1}} \pare{\overline{s^1_t}} \leq  \frac{1}{4}\frac{1}{2\pi}\int \btp \cos(\alpha)\  \dd\alpha   +\av{h}_{W^{1, \infty}} \cP\pare{\av{h}_{W^{1, \infty}} }  \int \frac{\btp}{2\sin^2\pare{\alpha/2}} \dd\alpha +  \cP\pare{\av{h}_{W^{1, \infty}} }   \av{h}_{W^{1, \infty}}\av{h'}_{L^\infty}  . 
\end{equation}

We can now start studying the term $ \cJ_{3, 1, \RN{2}} $. We expand the derivative and we obtain that 
\begin{multline*}
\cJ_{3, 1, \RN{2}} = \frac{1}{4\pi}  \int \frac{-4r\partial_\alpha\theta \sin^2\pare{\alpha / 2} + 2r\pare{r-\theta}\sin(\alpha) + 2\theta\partial_\alpha\theta}{\pare{4r\pare{r-\theta}\sin^2\pare{\alpha/2} + \theta^2}^2} \bra{2r\sin^2\pare{\alpha / 2} + \theta\cos(\alpha) } \bra{ r' \cos \alpha  - r \sin(\alpha)  }  \ \theta' \ \dd\alpha \\
-\frac{1}{4\pi}  \int \Bigg\{ \frac{\bra{r\sin(\alpha)+ \partial_\alpha\theta\cos(\alpha) - \theta\sin(\alpha) } \bra{ r' \cos \alpha  - r \sin(\alpha)    } }{ 4r\pare{r-\theta}\sin^2\pare{\alpha / 2} + \theta^2 }  \\
 + \frac{\bra{2r\sin^2\pare{\alpha / 2} + \theta\cos(\alpha) } \bra{ - r' \sin \alpha  - r \cos(\alpha)  }   }{4r\pare{r-\theta}\sin^2\pare{\alpha / 2} + \theta^2} \Bigg\}   \theta' \ \dd \alpha. 
\end{multline*}

Computations similar to the ones performed for the terms $ \cJ_1$ and $ \cJ_2 $ allow us to deduce the estimate
\begin{equation}\label{eq:cJ31II}
\cJ_{3, 1, \RN{2}} \pare{\overline{s^1_t}} \leq   \frac{1}{4}\frac{1}{2\pi}\int \btp \cos(\alpha)\  \dd\alpha   + \av{h}_{W^{1, \infty}} \cP\pare{\av{h}_{W^{1, \infty}} }   \int \frac{\btp}{2\sin^2\pare{\alpha/2}} \dd\alpha +  \cP\pare{\av{h}_{W^{1, \infty}} }   \av{h}_{W^{1, \infty}}\av{h'}_{L^\infty}  . 
\end{equation}

We combine now the estimates \eqref{eq:cJ31I} and \eqref{eq:cJ31II} and obtain that
\begin{equation}\label{eq:cJ31}
\cJ_{3, 1 } \pare{\overline{s^1_t}} \leq  \frac{1}{2}\frac{1}{2\pi}\int \btp \cos(\alpha)\  \dd\alpha   + \av{h}_{W^{1, \infty}} \cP\pare{\av{h}_{W^{1, \infty}} }  \int \frac{\btp}{2\sin^2\pare{\alpha/2}} \dd\alpha +  \cP\pare{\av{h}_{W^{1, \infty}} }   \av{h}_{W^{1, \infty}}\av{h'}_{L^\infty}   .
\end{equation}

We can now focus on the term $ \cJ_{3, 2} $. let us use the identity
\begin{multline*}
   \bra{  r'  + \pare{r-\theta}\sin(\alpha) } \bra{r  \cos \alpha - \pare{r-\theta}}     + \bra{r - \pare{r-\theta} \cos \alpha} \bra{ r' \cos \alpha  - r \sin(\alpha)}  \\
   = \bra{r'\pare{1+\cos^2\alpha}\theta} - \bra{\pare{4r \pare{r-\theta} \sin^2\pare{\alpha/2} + \theta^2}\sin(\alpha)} - \bra{4rr'\sin^4\pare{\alpha / 2}}, 
\end{multline*}
 in order to reformulate $ \cJ_{3, 2} $ as
\begin{equation*}
\cJ_{3, 2} = \cJ_{3, 2, \RN{1}} +  \cJ_{3, 2, \RN{2}} +  \cJ_{3, 2, \RN{3}}, 
\end{equation*}
where 
\begin{align*}
\cJ_{3, 2, \RN{1}} & = -\frac{1}{4\pi}  \int \frac{r'\pare{r-\theta} \pare{1+\cos^2\alpha}}{2r\pare{r-\theta} + \frac{\theta^2}{2\sin^2\pare{\alpha / 2}}}\frac{\theta}{2\sin^2\pare{\alpha/2}} \dd \alpha, \\
\cJ_{3, 2, \RN{2}} & = \frac{1}{4\pi} \int\pare{r-\theta}\sin(\alpha) \ \dd \alpha, \\
\cJ_{3, 2, \RN{3}} & = \frac{1}{4\pi} \int \frac{ 2\sin^2\pare{\alpha/2}}{2r\pare{r-\theta} + \frac{\theta^2 }{2 \sin^2\pare{\alpha/2}}} rr'\pare{r-\theta}\dd\alpha.
\end{align*}
We can further compute
\begin{align*}
\cJ_{3, 2, \RN{1}} & = -\frac{1}{4\pi} \int \frac{r'\pare{r-\theta} \pare{1+\cos^2\alpha}}{2r\pare{r-\theta} + \frac{\theta^2}{2\sin^2\pare{\alpha / 2}}}\frac{\theta}{2\sin^2\pare{\alpha/2}} \dd \alpha, \\
& = \frac{1}{4\pi} \int  \frac{r'\pare{r-\theta} \pare{1+\cos^2\alpha}}{2r\pare{r-\theta} + \frac{\theta^2}{2\sin^2\pare{\alpha / 2}}}\frac{h'\pare{s}\alpha - \theta}{2\sin^2\pare{\alpha/2}} \dd \alpha - \frac{1}{4\pi} \int  \frac{r'\pare{r-\theta} \pare{1+\cos^2\alpha}}{2r\pare{r-\theta} + \frac{\theta^2}{2\sin^2\pare{\alpha / 2}}}\frac{h'\pare{s}\alpha }{2\sin^2\pare{\alpha/2}} \dd \alpha\\
& = \mathbb{L}_1 + \mathbb{L}_2. 
\end{align*}
We use now the identity
\begin{equation*}
1= 2\sin^2\pare{\alpha/4} + \cos\pare{\alpha / 2}, 
\end{equation*}
in order to deduce that
\begin{multline*}
\mathbb{L}_1 = \frac{1}{4\pi} \int  \frac{r'\pare{r-\theta} \pare{1+\cos^2\alpha}}{2r\pare{r-\theta} + \frac{\theta^2}{2\sin^2\pare{\alpha / 2}}}\frac{\sin^2\pare{\alpha/4}}{\sin^2\pare{\alpha/2}}\pare{h'\pare{s}\alpha - \theta} \dd \alpha  \\
+ \frac{1}{4\pi} \int \pare{  \frac{r'\pare{r-\theta} \pare{1+\cos^2\alpha}}{2r\pare{r-\theta} + \frac{\theta^2}{2\sin^2\pare{\alpha / 2}}} \sin\pare{\alpha / 2} }\frac{h'\pare{s}\alpha - \theta}{2\sin^3\pare{\alpha/2}} \cos\pare{\alpha / 2} \dd \alpha = \bL_{1, 1} + \bL_{1, 2}. 
\end{multline*}
The term $ \bL_{1, 1} $ is not a singular integral so that we can easily obtain the bound
\begin{equation*}
\bL_{1, 1} \pare{\overline{s^1_t}} \leq \cP\pare{\av{h}_{W^{1, \infty}}}  \av{h}_{W^{1, \infty}}\av{h'}_{L^\infty} .
\end{equation*}
Since 
\begin{equation*}
\left. \frac{h'\pare{s}\alpha - \theta}{2\sin^3\pare{\alpha/2}} \cos\pare{\alpha / 2} \right|_{s=\overline{s^1_t}} \geq 0, 
\end{equation*}
we deduce the bound
\begin{equation}\label{eq:bL1}
\bL_{1} \pare{\overline{s^1_t}} \leq \av{h}_{W^{1, \infty}} \cP\pare{\av{h}_{W^{1, \infty}}}  \int   \frac{h'\pare{\overline{s^1_t}}\alpha - \overline{\theta}  }{2\sin^3\pare{\alpha/2}}\cos\pare{\alpha/2} \dd \alpha  + \cP\pare{\av{h}_{W^{1, \infty}}} \av{h}_{W^{1, \infty}}\av{h'}_{L^\infty}  . 
\end{equation}
We use now a Taylor expansion together with the symmetry of the integrand in order to write $ \bL_2 $ as
\begin{align*}
\bL_2&=- \frac{1}{4}\frac{1}{2\pi}\frac{1}{r^2} \int  \left[1 +\sum_{\ell=1}^\infty\left(\frac{\theta }{r } -\frac{\theta^2}{4r^2\sin^2\pare{\alpha/2} }\right)^\ell\right]\pare{1+\cos^2\alpha}r'\pare{r-\theta}\frac{h'\pare{s}\alpha }{2\sin^2\pare{\alpha/2}} \dd \alpha\\
&=- \frac{1}{4}\frac{1}{2\pi}\frac{(r')^2}{r} \int  \left[1 +\sum_{\ell=1}^\infty\left(\frac{\theta }{r } -\frac{\theta^2}{4r^2\sin^2\pare{\alpha/2} }\right)^\ell\right]\pare{1+\cos^2\alpha}\frac{\alpha }{2\sin^2\pare{\alpha/2}} \dd \alpha\\
&\quad+ \frac{1}{4}\frac{1}{2\pi}\frac{(r')^2}{r^2} \int  \left[1 +\sum_{\ell=1}^\infty\left(\frac{\theta }{r } -\frac{\theta^2}{4r^2\sin^2\pare{\alpha/2} }\right)^\ell\right]\pare{1+\cos^2\alpha}\frac{\alpha \theta }{2\sin^2\pare{\alpha/2}} \dd \alpha\\
&=- \frac{1}{4}\frac{1}{2\pi}\frac{(r')^2}{r} \int  \left[\sum_{\ell=1}^\infty\left(\frac{\theta }{r } -\frac{\theta^2}{4r^2\sin^2\pare{\alpha/2} }\right)^\ell\right]\pare{1+\cos^2\alpha}\frac{\alpha }{2\sin^2\pare{\alpha/2}} \dd \alpha\\
&\quad+ \frac{1}{4}\frac{1}{2\pi}\frac{(r')^2}{r^2} \int  \left[1 +\sum_{\ell=1}^\infty\left(\frac{\theta }{r } -\frac{\theta^2}{4r^2\sin^2\pare{\alpha/2} }\right)^\ell\right]\pare{1+\cos^2\alpha}\frac{\alpha \theta }{2\sin^2\pare{\alpha/2}} \dd \alpha.
\end{align*}
We find that
$$
\frac{1}{4}\frac{1}{2\pi}\frac{(r')^2}{r^2} \int  \left[1 +\sum_{\ell=1}^\infty\left(\frac{\theta }{r } -\frac{\theta^2}{4r^2\sin^2\pare{\alpha/2} }\right)^\ell\right]\pare{1+\cos^2\alpha}\frac{\alpha \theta }{2\sin^2\pare{\alpha/2}} \dd \alpha\leq \cP\pare{\av{h}_{W^{1, \infty}}}  \av{h}_{W^{1, \infty}} \av{h'}_{L^\infty}.
$$
Similarly, we can expand and obtain that
$$
\sum_{\ell=1}^\infty\left(\frac{\theta }{r } -\frac{\theta^2}{4r^2\sin^2\pare{\alpha/2} }\right)^\ell=\sum_{\ell=1}^\infty \sum_{k=0}^\ell c_{\ell,k}\left(\frac{\theta }{r }\right)^{\ell-k}\left(-\frac{\theta^2}{4r^2\sin^2\pare{\alpha/2} }\right)^k.
$$
Using this identity we have that
\begin{multline*}
- \frac{1}{4}\frac{1}{2\pi}\frac{(r')^2}{r} \int  \left[\sum_{\ell=1}^\infty\left(\frac{\theta }{r } -\frac{\theta^2}{4r^2\sin^2\pare{\alpha/2} }\right)^\ell\right]\pare{1+\cos^2\alpha}\frac{\alpha }{2\sin^2\pare{\alpha/2}} \dd \alpha\\
\leq - \frac{1}{4}\frac{1}{2\pi}\frac{(r')^2}{r} \int  \left[\sum_{\ell=1}^\infty\left(-\frac{\theta^2}{4r^2\sin^2\pare{\alpha/2} }\right)^\ell\right]\pare{1+\cos^2\alpha}\frac{\alpha }{2\sin^2\pare{\alpha/2}} \dd \alpha\\
+\cP\pare{\av{h}_{W^{1, \infty}}}  \av{h}_{W^{1, \infty}} \av{h'}_{L^\infty}.
\end{multline*}
Furthermore, we compute
\begin{multline*}
- \frac{1}{4}\frac{1}{2\pi}\frac{(r')^2}{r} \int  \left[\sum_{\ell=1}^\infty\left(-\frac{\theta^2}{4r^2\sin^2\pare{\alpha/2} }\right)^\ell\right]\pare{1+\cos^2\alpha}\frac{\alpha }{2\sin^2\pare{\alpha/2}}\left(1\pm\cos(\alpha/2)\right) \dd \alpha\\
\leq \av{h}_{W^{1, \infty}}    \cP\pare{\av{h}_{W^{1, \infty}}} \int   \frac{h'\pare{\overline{s^1_t}}\alpha - \bar{\theta}}{2\sin^3\pare{\alpha/2}}\cos\pare{\alpha/2} \dd \alpha+\cP\pare{\av{h}_{W^{1, \infty}}}  \av{h}_{W^{1, \infty}} \av{h'}_{L^\infty}.
\end{multline*}
As a consequence, we find that
\begin{equation}\label{eq:bL2}
\begin{aligned}
\bL_2\pare{\overline{s^1_t}} \leq & \av{h}_{W^{1, \infty}}    \cP\pare{\av{h}_{W^{1, \infty}}}  \int   \frac{h'\pare{\overline{s^1_t}}\alpha - \bar{\theta} }{2\sin^3\pare{\alpha/2}}\cos\pare{\alpha/2} \dd \alpha + \cP\pare{\av{h}_{W^{1, \infty}}}  \av{h}_{W^{1, \infty}} \av{h'}_{L^\infty} .
\end{aligned}
\end{equation}

The inequalities \eqref{eq:bL1} and \eqref{eq:bL2} allow us to deduce the estimate
\begin{equation}\label{eq:cJ32I}
\cJ_{3, 2,\RN{1}} \pare{\overline{s^1_t}} \leq  \av{h}_{W^{1, \infty}}    \cP\pare{\av{h}_{W^{1, \infty}}} \int   \frac{h'\pare{\overline{s^1_t}}\alpha - \bar{\theta} }{2\sin^3\pare{\alpha/2}}\cos\pare{\alpha/2} \dd \alpha + \cP\pare{\av{h}_{W^{1, \infty}}}  \av{h}_{W^{1, \infty}} \av{h'}_{L^\infty} . 
\end{equation}

The term $ \cJ_{3, 2, \RN{2}} $ provides a linear contribution, which is
\begin{equation}\label{eq:cJ32II}
\cJ_{3, 2, \RN{2}} = -\frac{1}{4\pi} \int \theta \sin \alpha \ \dd \alpha = \frac{1}{4\pi} \int \theta'\cos(\alpha) \ \dd\alpha . 
\end{equation}

Let us now study the term $ \cJ_{3, 2, \RN{3}} $. We use the identity
\begin{equation*}
\frac{ 2\sin^2\pare{\alpha/2}}{2r\pare{r-\theta} + \frac{\theta^2 }{2 \sin^2\pare{\alpha/2}}} rr'\pare{r-\theta} = r' \sin^2\pare{\alpha / 2} + \cP\pare{\av{h'}_{L^\infty}^2}, 
\end{equation*}
which lead to the bound
\begin{equation}\label{eq:cJ32III}
\cJ_{3, 2, \RN{3}}\pare{\overline{s^1_t}} = \frac{1}{4} h'\pare{\overline{s^1_t}} + \cP\pare{\av{h}_{W^{1, \infty}}} \av{h}_{W^{1, \infty}} \av{h'}_{L^\infty} . 
\end{equation}

We combine \eqref{eq:cJ32I}, \eqref{eq:cJ32II} and \eqref{eq:cJ32III} in order to obtain a bound on $ \cJ_{3, 2} $ which is 
\begin{multline}\label{eq:cJ32}
\cJ_{3, 2} \pare{\overline{s^1_t}} \leq  \frac{h'\pare{\overline{s^1_t}}}{4}  +\frac{1}{2}\frac{1}{2\pi} \int \theta'\cos(\alpha) \ \dd\alpha   \\
+  \av{h}_{W^{1, \infty}}    \cP\pare{\av{h}_{W^{1, \infty}}} \int   \frac{h'\pare{\overline{s^1_t}}\alpha - \bar{\theta} }{2\sin^3\pare{\alpha/2}}\cos\pare{\alpha/2} \dd \alpha + \cP\pare{\av{h}_{W^{1, \infty}}} \av{h}_{W^{1, \infty}} \av{h'}_{L^\infty} . 
\end{multline}
The estimates \eqref{eq:cJ31} and \eqref{eq:cJ32} close, finally the estimation of $ \cJ_3 $, which is
\begin{multline}\label{eq:cJ3}
\cJ_{3 } \pare{\overline{s^1_t}} \leq  \frac{h'\pare{\overline{s^1_t}}}{4}  +  \frac{1}{2\pi} \int \theta'\cos(\alpha) \ \dd\alpha  + \av{h}_{W^{1, \infty}} \cP\pare{\av{h}_{W^{1, \infty}} }   \int \frac{\btp}{2\sin^2\pare{\alpha/2}} \dd\alpha  \\
+ \av{h}_{W^{1, \infty}}    \cP\pare{\av{h}_{W^{1, \infty}}} \int   \frac{h'\pare{\overline{s^1_t}}\alpha - \bar{\theta} }{2\sin^3\pare{\alpha/2}}\cos\pare{\alpha/2} \dd \alpha +  \cP\pare{\av{h}_{W^{1, \infty}} }  \av{h}_{W^{1, \infty}}\av{h'}_{L^\infty}.
\end{multline}

\subsubsection*{Bound of $ \cJ_4 $}

As it was done for $ \cJ_3 $, we decompose $ \cJ_4 = \cJ_{4, 1} + \cJ_{4, 2} $ where
\begin{align*}
\cJ_{4, 1} = & \   \frac{1}{2\pi}   \int \partial_\alpha\set{ \frac{ \bra{r- \pare{r-\theta} \cos \alpha} \bra{ r \cos \alpha - \pare{r-\theta}}   }{ \pare{4r\pare{r-\theta}\sin^2\pare{\alpha / 2} + \theta^2}^{2}} 
 \bra{ r r' -\pare{r-\theta}r'\cos \alpha + r \pare{r-\theta} \sin(\alpha) } } \ \theta' \ \dd \alpha , \\
 \cJ_{4, 2} = & \   \frac{1}{2\pi}   \int \set{ \frac{ \bra{r- \pare{r-\theta} \cos \alpha} \bra{ r \cos \alpha - \pare{r-\theta}}   }{ \pare{4r\pare{r-\theta}\sin^2\pare{\alpha / 2} + \theta^2}^{2}} 
 \bra{ r r' -\pare{r-\theta}r'\cos \alpha + r \pare{r-\theta} \sin(\alpha) } }\pare{r-\theta}   \dd \alpha .
\end{align*}
And we expand $ \cJ_{4, 1} $ thus obtaining 
\begin{equation*}
\cJ_{4, 1} = \cJ_{4, 1 , \RN{1}} + \cJ_{4, 1 , \RN{2}} + \cJ_{4, 1 , \RN{3}} + \cJ_{4, 1 , \RN{4}}, 
\end{equation*}
where
\begin{align*}
 \cJ_{4, 1 , \RN{1}}  = & \   \frac{1}{2\pi}   \int  \frac{ -2 \pare{-4 r \partial_\alpha\theta\sin^2\pare{\alpha / 2} + 2r\pare{r-\theta}\sin(\alpha) + 2\theta\partial_\alpha\theta} }{ \pare{4r\pare{r-\theta}\sin^2\pare{\alpha / 2} + \theta^2}^{3}} \\
 & \qquad \times \bra{r- \pare{r-\theta} \cos \alpha} \bra{ r \cos \alpha - \pare{r-\theta}}   
 \bra{ r r' -\pare{r-\theta}r'\cos \alpha + r \pare{r-\theta} \sin(\alpha) }  \ \theta' \ \dd \alpha , \\
 \cJ_{4, 1 , \RN{2}}  = & \    \frac{1}{2\pi}  \int  \frac{ \bra{\partial_\alpha\theta \cos \alpha + \pare{r-\theta}\sin(\alpha)} \bra{ r \cos \alpha - \pare{r-\theta}}   }{ \pare{4r\pare{r-\theta}\sin^2\pare{\alpha / 2} + \theta^2}^{2}} 
 \bra{ r r' -\pare{r-\theta}r'\cos \alpha + r \pare{r-\theta} \sin(\alpha) }  \ \theta' \ \dd \alpha, \\
 \cJ_{4, 1, \RN{3}} = & \   \frac{1}{2\pi}   \int  \frac{ \bra{r- \pare{r-\theta} \cos \alpha} \bra{-r\sin(\alpha) + \partial_\alpha\theta}   }{ \pare{4r\pare{r-\theta}\sin^2\pare{\alpha / 2} + \theta^2}^{2}} 
 \bra{ r r' -\pare{r-\theta}r'\cos \alpha + r \pare{r-\theta} \sin(\alpha) }  \ \theta' \ \dd \alpha , \\
 \cJ_{4, 1, \RN{4}} = & \   \frac{1}{2\pi}   \int  \frac{ \bra{r- \pare{r-\theta} \cos \alpha} \bra{ r \cos \alpha - \pare{r-\theta}}   }{ \pare{4r\pare{r-\theta}\sin^2\pare{\alpha / 2} + \theta^2}^{2}} 
 \bra{ \pare{\pare{r-\theta} r' - r\partial_\alpha\theta}\sin(\alpha) + \pare{\partial_\alpha\theta r' + r\pare{r-\theta}}\cos(\alpha) }  \ \theta' \ \dd \alpha , \\
\end{align*}

Let us start analyzing the term $ \cJ_{4, 1, \RN{1}} $, and we reformulate it as
\begin{multline*}
 \cJ_{4, 1 , \RN{1}}  =   \frac{1}{2\pi}  \pv \int \frac{\sin^6 \pare{\alpha /2} }{ \pare{4r\pare{r-\theta}\sin^2\pare{\alpha / 2} + \theta^2}^{3}} \frac{ -2 \pare{-4 r \partial_\alpha\theta\sin^2\pare{\alpha / 2} + 2r\pare{r-\theta}\sin(\alpha) + 2\theta\partial_\alpha\theta} }{ \sin\pare{\alpha /2}} \\
  \times \bra{\frac{r- \pare{r-\theta} \cos \alpha}{ \sin\pare{\alpha /2}}} \bra{ \frac{r \cos \alpha - \pare{r-\theta}}{\sin\pare{\alpha /2}}}   
 \bra{ \frac{r r' -\pare{r-\theta}r'\cos \alpha + r \pare{r-\theta} \sin(\alpha) }{\sin\pare{\alpha /2}}}  \ \frac{ \theta'}{\sin^2\pare{\alpha/2}} \ \dd \alpha .
\end{multline*}
We use the following identities
\begin{align*}
\frac{\sin^6 \pare{\alpha /2} }{ \pare{4r\pare{r-\theta}\sin^2\pare{\alpha / 2} + \theta^2}^{3}} & = \frac{1}{\pare{2r}^6}  +\cP\pare{\av{h'}_{L^\infty}} , \\
\frac{ -2 \pare{-4 r \partial_\alpha\theta\sin^2\pare{\alpha / 2} + 2r\pare{r-\theta}\sin(\alpha) + 2\theta\partial_\alpha\theta} }{ \sin\pare{\alpha /2}}  & = - 8 r^2 \cos\pare{\alpha / 2}  +\cP\pare{\av{h'}_{L^\infty}} , \\
\frac{r- \pare{r-\theta} \cos \alpha}{ \sin\pare{\alpha /2}} & = 2r\sin\pare{\alpha/2} +\cP\pare{\av{h'}_{L^\infty}} , \\
 \frac{r \cos \alpha - \pare{r-\theta}}{\sin\pare{\alpha /2}} & = - 2r\sin\pare{\alpha/2}  +\cP\pare{\av{h'}_{L^\infty}} , \\
 \frac{r r' -\pare{r-\theta}r'\cos \alpha + r \pare{r-\theta} \sin(\alpha) }{\sin\pare{\alpha /2}} & = 2r^2\cos\pare{\alpha / 2}   +\cP\pare{\av{h'}_{L^\infty}} ,
\end{align*}
so that we proved that
\begin{equation*}
\cJ_{4, 1, \RN{1}} \pare{\overline{s^1_t}} =  \frac{1}{2}\frac{1}{2\pi} \int \btp \pare{1+\cos(\alpha)}\dd \alpha   + \av{h}_{W^{1, \infty}}\cP\pare{\av{h}_{W^{1, \infty}}} \int \frac{\btp}{2\sin^2\pare{\alpha / 2}} \dd\alpha. 
\end{equation*}
In a similar fashion we can deduce that
\begin{align*}
\cJ_{4, 1, \RN{2}} \pare{\overline{s^1_t}} & =  -\frac{1}{4}\frac{1}{2\pi} \int \btp \pare{1+\cos(\alpha)}\dd \alpha + \av{h}_{W^{1, \infty}}\cP\pare{\av{h}_{W^{1, \infty}}} \int \frac{\btp}{2\sin^2\pare{\alpha / 2}} \dd\alpha, \\
\cJ_{4, 1, \RN{3}} \pare{\overline{s^1_t}} & =  -\frac{1}{4}\frac{1}{2\pi} \int \btp \pare{1+\cos(\alpha)}\dd \alpha + \av{h}_{W^{1, \infty}}\cP\pare{\av{h}_{W^{1, \infty}}} \int \frac{\btp}{2\sin^2\pare{\alpha / 2}} \dd\alpha, \\
\cJ_{4, 1, \RN{4}} \pare{\overline{s^1_t}} & =  -\frac{1}{4}\frac{1}{2\pi} \int \btp  \cos(\alpha) \ \dd \alpha + \av{h}_{W^{1, \infty}}\cP\pare{\av{h}_{W^{1, \infty}}} \int \frac{\btp}{2\sin^2\pare{\alpha / 2}} \dd\alpha.
\end{align*}
As a consequence we have that
\begin{equation}\label{eq:J41}
\cJ_{4, 1 }\pare{\overline{s^1_t}} \leq  -\frac{1}{4}\frac{1}{2\pi} \int \btp   \cos(\alpha) \ \dd \alpha   + \av{h}_{W^{1, \infty}}\cP\pare{\av{h}_{W^{1, \infty}}} \int \frac{\btp}{2\sin^2\pare{\alpha / 2}} \dd\alpha . 
\end{equation}

We now use the identity
\begin{equation*}
 \bra{r- \pare{r-\theta} \cos \alpha} \bra{ r \cos \alpha - \pare{r-\theta}}  = -4r \pare{r-\theta} \sin^4\pare{\alpha / 2} + \cO\pare{\av{h'}_{L^\infty}}, 
\end{equation*}
in order to obtain the required bound for $ \cJ_{4, 2} $
\begin{equation*}
\cJ_{4, 2}\pare{\overline{s^1_t}} \leq  - \frac{h'\pare{\overline{s^1_t}}}{4} -\frac{1}{4}\frac{1}{2\pi} \int \btp \ \cos(\alpha) \ \dd\alpha  + \cP\pare{\av{h}_{W^{1, \infty}}}\av{h}_{W^{1, \infty}} \av{h'}_{L^\infty}.
\end{equation*}
Collecting both estimates we conclude that
\begin{multline} \label{eq:cJ4}
\cJ_4\pare{\overline{s^1_t}} \leq  - \frac{h'\pare{\overline{s^1_t}}}{4} -\frac{1}{2}\frac{1}{2\pi} \int \btp  \cos(\alpha)  \ \dd \alpha  + \av{h}_{W^{1, \infty}}\cP\pare{\av{h}_{W^{1, \infty}}} \int \frac{\btp}{2\sin^2\pare{\alpha / 2}} \dd\alpha \\
  +  \cP\pare{\av{h}_{W^{1, \infty}}}\av{h}_{W^{1, \infty}} \av{h'}_{L^\infty} . 
\end{multline}

\subsubsection*{Bound of $ \cJ_5 $}

Let us rewrite $ \cJ_5 $ as 
\begin{equation*}
\begin{aligned}
 \cJ_5 = & \    \frac{1}{2 \pi}   \int    \frac{ \pare{  \eta    +  r  }r'   \sin \alpha   }{4r\pare{r-\theta}\sin^2\pare{\alpha / 2} + \theta^2 }  \pare{r-\theta}'   \dd \alpha , \\
 = & \    \frac{1}{2 \pi}   \int    \frac{ \eta r'   \sin \alpha   }{4r\pare{r-\theta}\sin^2\pare{\alpha / 2} + \theta^2 }  \pare{r-\theta}'   \dd \alpha  +  \frac{1}{2 \pi}  \int    \frac{ r r'   \sin \alpha   }{4r\pare{r-\theta}\sin^2\pare{\alpha / 2} + \theta^2 }  \pare{r-\theta}'   \dd \alpha  ,  \\
 = & \ \cJ_{5, 1} +  \cJ_{5, 2} .
\end{aligned}
\end{equation*}
Thus we compute
\begin{equation}\label{eq:J5_1}
\cJ_{5, 1}\pare{\overline{s^1_t}} \leq\cP\pare{\av{h}_{W^{1, \infty}}} \av{h}_{W^{1, \infty}} \av{h'}_{L^\infty} . 
\end{equation}

We study now the term $ \cJ_{5, 2} $. We decompose ir as
\begin{equation*}
 \cJ_{5, 2} =    \frac{1}{2 \pi}   \int    \frac{ r \pare{ r' }^2   \sin \alpha   }{4r\pare{r-\theta}\sin^2\pare{\alpha / 2} + \theta^2 }     \dd \alpha  - \frac{1}{2 \pi}   \int    \frac{ r  r'   \sin \alpha   }{4r\pare{r-\theta}\sin^2\pare{\alpha / 2} + \theta^2 }  \ \theta' \    \dd \alpha = \cJ_{5, 2, \RN{1}} + \cJ_{5, 2, \RN{2}}. 
\end{equation*}
We start studying the term $ \cJ_{5, 2, \RN{2}} $, and we rewrite it as
\begin{equation*}
\cJ_{5, 2, \RN{2}}  = - \frac{1}{2\pi}   \int \frac{rr'}{2r\pare{r-\theta} + \frac{\theta^2}{2\sin^2\pare{\alpha/2}}} \cot\pare{\alpha/ 2} \theta' \ \dd \alpha. 
\end{equation*}
Using Lemma \ref{lem:monotonicity_lemma2} we conclude that
\begin{equation}\label{eq:J5_2}
\cJ_{5, 2, \RN{2}}\pare{\overline{s^1_t}} \leq  \av{h}_{W^{1, \infty}}\cP\pare{\av{h}_{W^{1, \infty}}} \int \frac{\btp}{2\sin^2\pare{\alpha/2}} \dd\alpha. 
\end{equation}
We can now consider the term $ \cJ_{5, 2, \RN{1}} $. To estimate this term we proceed similarly as for $\mathbb{L}_2$ and find that
 \begin{multline}\label{eq:cJ5}
 \cJ_5 \pare{\overline{s^1_t}} \leq \av{h}_{W^{1, \infty}} \cP\pare{\av{h}_{W^{1, \infty}}}  \int   \frac{h'\pare{\overline{s^1_t}}\alpha - \bar{\theta} }{2\sin^3\pare{\alpha/2}}\cos\pare{\alpha/2} \dd \alpha \\
  +  \av{h}_{W^{1, \infty}}\cP\pare{\av{h}_{W^{1, \infty}}} \int \frac{\btp}{2\sin^2\pare{\alpha/2}} \dd\alpha + \cP\pare{\av{h}_{W^{1, \infty}}} \av{h}_{W^{1, \infty}}\av{h'}_{L^{\infty}}. 
 \end{multline}

\subsubsection*{Bound of $ \cJ_6 $}

Let us rewrite $ \cJ_6 $ as 
\begin{equation*}
 \cJ_6 =    \frac{1}{2 \pi}  \int    \frac{r\cos \alpha \ \pare{2r \sin^2\pare{\alpha / 2}+\theta \cos(\alpha)} + r\pare{r-\theta}\sin^2\alpha }{ 4 r \pare{r-\theta}\sin^2\pare{\alpha/2} + \theta^2 }  \pare{r-\theta}'  \dd \alpha , 
\end{equation*}
and let us notice that
\begin{equation*}
\cJ_6 =  \frac{1}{2\pi} \int \bra{1+\frac{\theta}{r} + \cP\pare{\av{\theta}_{L^{ \infty}}^2}} \bra{ \frac{\cos(\alpha)}{2} + \frac{\theta \cos^2\alpha }{4r\sin^2\pare{\alpha / 2}} + \pare{1-\frac{\theta}{r}}\cos^2\pare{\alpha / 2} } \pare{r-\theta}'\dd\alpha.
\end{equation*}
Hence computations similar to the ones performed for the term $ \cJ_5 $ lead us to the estimate
 \begin{multline}\label{eq:cJ6}
 \cJ_6 \pare{\overline{s^1_t}} \leq  \frac{1}{2\pi} \int \pare{\frac{\cos(\alpha) }{2}+ \cos^2\pare{\alpha / 2}} \pare{\brp-\btp }\dd\alpha  \\
+  \av{h}_{W^{1, \infty}} \cP\pare{\av{h}_{W^{1, \infty}}}  \int   \frac{h'\pare{\overline{s^1_t}}\alpha - \bar{\theta}}{2\sin^3\pare{\alpha/2}}\cos\pare{\alpha/2} \dd \alpha \\
  +  \av{h}_{W^{1, \infty}}\cP\pare{\av{h}_{W^{1, \infty}}} \int \frac{\btp}{2\sin^2\pare{\alpha/2}} \dd\alpha + \cP\pare{\av{h}_{W^{1, \infty}}} \av{h}_{W^{1, \infty}} \av{h'}_{L^\infty}  .
 \end{multline}

\subsubsection*{Bound of $ \cJ_7 $}
Let us rewrite $ \cJ_7 $ as 

\begin{equation*}
\begin{aligned}
\cJ_7 = & \ -  \frac{1}{2 \pi} \int    \frac{  r  \pare{2r\sin^2\pare{\alpha / 2} +\theta \cos \alpha}  \sin \alpha    }{\pare{4r\pare{r-\theta}\sin^2\pare{\alpha / 2} + \theta^2  }^2 }  \pare{2 r r' + 2 \pare{r-\theta }\pare{ r \sin(\alpha) -r' \cos(\alpha) }} \pare{r-\theta}'   \dd \alpha .
\end{aligned}
\end{equation*}
We observe that
\begin{equation*}
\begin{aligned}
\frac{2 r r' + 2 \pare{r-\theta }\pare{ r \sin(\alpha) -r' \cos(\alpha) }}{2r \sin\pare{\alpha / 2}} & = 2r \cos\pare{\alpha/2} + \cP\pare{\av{h'}_{L^\infty}}, \\
\pare{\frac{4r^2\sin^2\pare{\alpha / 2}}{4r\pare{r-\theta}\sin^2\pare{\alpha / 2} + \theta^2}}^2 & = 1+\frac{2\theta}{r} + \cP\pare{\av{h'}_{L^\infty}^2}, \\
\frac{r\ \sin(\alpha)}{\pare{ 2r\sin\pare{\alpha/2} }^2 } & = \frac{\cot\pare{\alpha / 2}}{4r}, \\
\frac{ 2r\sin^2\pare{\alpha / 2} + \theta\cos(\alpha) }{2 r \sin \pare{\alpha/2}} & = \sin\pare{\alpha / 2} + \frac{\theta\cos(\alpha)}{2r\sin\pare{\alpha / 2}}, 
\end{aligned}
\end{equation*}
and this in turn allow us to deduce that
\begin{equation}\label{eq:cJ7}
\cJ_7\pare{\overline{s^1_t}} \leq - \frac{1}{2\pi}\int\pare{\brp - \btp}\cos^2\pare{ \alpha / 2 }\dd\alpha +  \cP\pare{\av{h}_{W^{1, \infty}}} \av{h}_{W^{1, \infty}} \av{h'}_{L^\infty}  . 
\end{equation}

We now sum \eqref{eq:cJ6} and \eqref{eq:cJ7} and obtain the estimate
 \begin{multline}\label{eq:cJ6cJ7}
 \cJ_6 \pare{\overline{s^1_t}} + \cJ_7 \pare{\overline{s^1_t}}  \leq  - \frac{1}{2}\frac{1}{2\pi} \int \btp \cos(\alpha) \ \dd\alpha  
+  \av{h}_{W^{1, \infty}} \cP\pare{\av{h}_{W^{1, \infty}}}  \int   \frac{h'\pare{\overline{s^1_t}}\alpha - \bar{\theta} }{2\sin^3\pare{\alpha/2}}\cos\pare{\alpha/2} \dd \alpha \\
  +  \av{h}_{W^{1, \infty}}\cP\pare{\av{h}_{W^{1, \infty}}}  \int \frac{\btp}{2\sin^2\pare{\alpha/2}} \dd\alpha + \cP\pare{\av{h}_{W^{1, \infty}}} \av{h}_{W^{1, \infty}} \av{h'}_{L^\infty} .
 \end{multline}

\subsubsection*{The equation for the evolution of $|h'|_{L^\infty}$}
We combine the estimates \eqref{eq:cJ1},  \eqref{eq:cJ2},  \eqref{eq:cJ3},  \eqref{eq:cJ4},  \eqref{eq:cJ5} and  \eqref{eq:cJ6cJ7} and use 

$$
h'\pare{\overline{s^1_t}} =\max_{s}\{h'(s,t)\}
$$

in order to obtain the estimate
\begin{multline*}
\ddt \max_{s}\{h'(s,t)\}+ \gamma'\pare{\overline{s^1_t}} \cdot \dot{M}\pare{t} \\
\leq - \frac{1}{4} \frac{1}{2\pi} \bra{1- \av{h}_{W^{1, \infty}} \cP\pare{\av{h}_{W^{1, \infty}}}} \pv \int \frac{\btp }{2\sin^2\pare{\alpha/2}}  \dd \alpha - \frac{1}{4} \frac{1}{2 \pi} \int \btp \pare{ 1 + \cos(\alpha)} \dd\alpha \\
 + \frac{h'\pare{ \overline{s^1_t} }}{4} 
+  \av{h}_{W^{1, \infty}}    \cP\pare{\av{h}_{W^{1, \infty}}} \pv \int   \frac{h'\pare{\overline{s^1_t}}\alpha - \theta \pare{\overline{s^1_t}, \overline{s^1_t} -\alpha } }{2\sin^3\pare{\alpha/2}}\cos\pare{\alpha/2} \dd \alpha 
 + \cP\pare{\av{h}_{W^{1, \infty}}}\av{h}_{W^{1, \infty}} \av{h'}_{L^\infty} . 
\end{multline*}
Let us remark that
\begin{equation*}
- \frac{1}{4} \frac{1}{2\pi} \int \btp \ \dd\alpha + \frac{h'\pare{ \overline{s^1_t} }}{4} = 0 , 
\end{equation*}
\begin{align*}
 \gamma'\pare{s } \cdot \dot{M}\pare{t} = \frac{1}{4} \frac{1}{2\pi} \int h' \pare{s -\alpha} \ \cos(\alpha) \ \dd\alpha , && \forall \ s\in \bS^1, 
\end{align*}
which we use combined with \eqref{eq:Lambda_alternative} in order to derive that
\begin{equation*}
\ddt \max_{s}\{h'(s,t)\}
\leq - \frac{1}{4} \frac{1}{2\pi} \bra{1- \av{h}_{W^{1, \infty}} \cP\pare{\av{h}_{W^{1, \infty}}}} \int \frac{\btp }{2\sin^2\pare{\alpha/2}}  \dd \alpha    + \cP\pare{\av{h}_{W^{1, \infty}}}\av{h}_{W^{1, \infty}} \av{h'}_{L^\infty} . 
\end{equation*}

Similar computations allow us to control the positive quantity 
$$
-\min_{s}\{h'(s,t)\}= -h'\pare{\underline{s^1_t}} 
$$ 
as
\begin{equation*}
 - \ddt \min_{s}\{h'(s,t)\} 
\leq - \frac{1}{4} \frac{1}{2\pi} \bra{1- \av{h}_{W^{1, \infty}} \cP\pare{\av{h}_{W^{1, \infty}}}}  \int \frac{- \utp }{2\sin^2\pare{\alpha/2}}  \dd \alpha    + \cP\pare{\av{h}_{W^{1, \infty}}}\av{h}_{W^{1, \infty}} \av{h'}_{L^\infty} ,  
\end{equation*}
so that we can estimate the evolution of $ \av{h'\pare{t}}_{L^\infty} = \max \set{h'\pare{\overline{s^1_t}, t}, \ - h'\pare{\underline{s^1_t}}  } $ as
\begin{equation}
\label{eq:control_h'_in_Linfty}
\begin{aligned}
\ddt \av{h'}_{L^\infty} 
&\leq - \frac{1}{4} \frac{1}{2\pi} \bra{1- \av{h}_{W^{1, \infty}} \cP\pare{\av{h}_{W^{1, \infty}}}} \max\set{\int \frac{\btp}{2\sin^2\pare{\alpha/2}}  \dd \alpha ,  \int \frac{ -\utp  }{2\sin^2\pare{\alpha/2}}  \dd \alpha} \\
&\quad   + \cP\pare{\av{h}_{W^{1, \infty}}}\av{h}_{W^{1, \infty}} \av{h'}_{L^\infty} .  
\end{aligned}
\end{equation}
We combine the estimates \eqref{eq:control_h_in_Linfty} and \eqref{eq:control_h'_in_Linfty} and we deduce 
\begin{multline}\label{eq:est1}
\ddt \av{h}_{W^{1, \infty}} 
\leq - \frac{1}{4} \frac{1}{2\pi} \bra{1- \av{h}_{W^{1, \infty}} \cP\pare{\av{h}_{W^{1, \infty}}}}  \max\set{\int \frac{\bt}{2\sin^2\pare{\alpha/2}}  \dd \alpha ,  \int \frac{ -\ut  }{2\sin^2\pare{\alpha/2}}  \dd \alpha} \\
- \frac{1}{4} \frac{1}{2\pi} \bra{1- \av{h}_{W^{1, \infty}} \cP\pare{\av{h}_{W^{1, \infty}}}}  \max\set{\int \frac{\btp}{2\sin^2\pare{\alpha/2}}  \dd \alpha ,  \int \frac{ -\utp  }{2\sin^2\pare{\alpha/2}}  \dd \alpha}\\
   + \cP\pare{\av{h}_{W^{1, \infty}}}\av{h}_{W^{1, \infty}} \av{h'}_{L^\infty} . 
\end{multline}
We apply Lemma \ref{eq:Lambda_Linfty} in order to state that
\begin{equation*}
\av{h'}_{L^\infty} \leq C \max\set{\int \frac{\btp}{2\sin^2\pare{\alpha/2}}  \dd \alpha ,  \int \frac{ -\utp  }{2\sin^2\pare{\alpha/2}}  \dd \alpha}   . 
\end{equation*}
As a consequence we can further simplify \eqref{eq:est1} and conclude that
\begin{multline}
\label{eq:est2}
\ddt \av{h}_{W^{1, \infty}} 
\leq - \frac{1}{4} \frac{1}{2\pi} \bra{1- \av{h}_{W^{1, \infty}} \cP\pare{\av{h}_{W^{1, \infty}}}}  \max\set{\int \frac{\bt}{2\sin^2\pare{\alpha/2}}  \dd \alpha ,  \int \frac{ -\ut  }{2\sin^2\pare{\alpha/2}}  \dd \alpha} \\
- \frac{1}{4} \frac{1}{2\pi} \bra{1- \av{h}_{W^{1, \infty}} \cP\pare{\av{h}_{W^{1, \infty}}}}  \max\set{\int \frac{\btp}{2\sin^2\pare{\alpha/2}}  \dd \alpha ,  \int \frac{ -\utp  }{2\sin^2\pare{\alpha/2}}  \dd \alpha}. 
\end{multline}

As a consequence, we can ensure that the right hand side of \eqref{eq:est2} is strictly negative if $ \av{h}_{W^{1, \infty}} $ is sufficiently small. We obtain that there exists a positive constant $ 0< \cC \ll 1 $ such that if $ \av{h_0}_{W^{1, \infty}}\leq \cC $ then for each $ t > 0 $
\begin{equation*}
\av{h\pare{t}}_{W^{1, \infty}}\leq \av{h_0}_{W^{1, \infty}}. 
\end{equation*}

We prove now the pointwise decay in time of $ \av{h\pare{t}}_{W^{1, \infty}} $. From \eqref{eq:control_h'_in_Linfty} we deduce, using the Poincar\'e-type inequality \eqref{eq:Lambda_Linfty}, that there exists a $ \delta > 0 $ s.t.
\begin{equation}\label{eq:exp_decay_h'}
\av{h'\pare{t}}_{L^\infty}\leq \av{h'_0}_{L^\infty} e^{-\delta t}. 
\end{equation}
This concludes the proof of Proposition \ref{prop:W1infty_enest}. \hfill $ \Box $

\section{\emph{A priori} estimates in $ H^{1} $}\label{sec:H1inftydecay}
The purpose of this section is to obtain the parabolic gain of regularity 
$$
L^2\pare{ 0,T;H^{3/2}}
$$
for the solution. Although these estimates are lower order compared to the pointwise estimates, this regularity is necessary in order we can pass to the limit in the weak formulation of the N-Peskin problem. 

\begin{prop}\label{prop:unif_bounds_L2h}
Let  $ T^\star \in (0,\infty] $ and $h = h(s,t)$ be a $\cC \pare{ \bra{ 0,T^\star}; \cC^2 } $ solution of \eqref{eq:eveqh2} such that $ \left. h\right|_{t=0} = h_0 $. Assume that
$$ 
\av{h_0}_{W^{1, \infty}}< c_0 
$$ 
with $c_0$ the constant in Proposition \ref{prop:W1infty_enest}. Then, for all $T\leq T^*$, there exists a  $ C(T) \in \pare{0, \infty} $ depending on $ T $ only such that 
\begin{equation*}
h \in L^2  \pare{\bra{0, T}; H^{3/2} \pare{\bS^1}}.
\end{equation*}
and the following bound holds true
\begin{equation*}
\norm{h}_{ L^2  \pare{\bra{0, T}; H^{3/2} \pare{\bS^1}}} \leq C(T). 
\end{equation*}
\end{prop}

\begin{proof}

All along the proof we denote with $ 0<\nu\ll 1 $ a positive constant whose explicit value may vary from line to line. \\

Let us recall that the evolution equation for $ h' $ can be written as
\begin{equation*}
\partial_t h' + \gamma'\cdot \dot{M} = \sum_{j=1}^7 \cJ_j , 
\end{equation*}
where the explicit formulations of the terms $ \cJ_j, \ j=1, \ldots, 7 $ are given in \eqref{eq:cJ's} and \eqref{eq:cJ's2}. Using computations similar to the ones performed in \eqref{eq:cJ1}, \eqref{eq:cJ2}, \eqref{eq:cJ3}, \eqref{eq:cJ4}, \eqref{eq:cJ5} and \eqref{eq:cJ6cJ7}, which isolate the linear (in $ h $) contribution of every $ \cJ_j $, we reformulate the evolution equation for $ h' $ as
\begin{multline*}
\partial_t h' +\frac{1}{4} \ \Lambda h' = \pare{ \cJ_2 + \frac{1}{4} \ \Lambda h' } +\pare{  \cJ_1 + \frac{1}{4}\frac{1}{2\pi} \int \theta' \ \cos(\alpha) \ \dd\alpha } + \pare{  \cJ_3 -\frac{h'\pare{s}}{4}  - \frac{1}{2\pi} \int \theta' \ \cos(\alpha) \ \dd\alpha } \\
 + \pare{  \cJ_4 +\frac{h'\pare{s}}{4} + \frac{1}{2}\frac{1}{2\pi} \int \theta' \ \cos(\alpha) \ \dd\alpha } + \cJ_5 +  \pare{\cJ_6 + \cJ_7 + \frac{1}{2}\frac{1}{2\pi} \int \theta' \ \cos(\alpha) \ \dd\alpha }, 
\end{multline*}
so that defining 
\begin{equation*}
\begin{aligned}
\cI_1 = & \   \cJ_1 + \frac{1}{4}\frac{1}{2\pi} \int \theta' \  \cos \alpha \ \dd\alpha , \\
\cI_2 = & \  \cJ_2 + \frac{1}{4} \ \Lambda h', \\
\cI_3 = &  \ \cJ_3 -\frac{h'\pare{s}}{4}  - \frac{1}{2\pi} \int \theta' \ \cos(\alpha) \ \dd\alpha , \\
\cI_4 = & \ \cJ_4 +\frac{h'\pare{s}}{4} + \frac{1}{2}\frac{1}{2\pi} \int \theta' \ \cos(\alpha) \ \dd\alpha , \\
\cI_5 = & \ \cJ_5
, \\
\cI_6 = & \ \cJ_6 + \cJ_7 + \frac{1}{2}\frac{1}{2\pi} \int \theta' \ \cos(\alpha) \ \dd\alpha , 
\end{aligned}
\end{equation*}
the evolution equation for $ h' $ becomes
\begin{equation}\label{eq:eveqh'Is}
\partial_t h' + \frac{1}{4} \ \Lambda h' = \sum_{j=1}^6 \cI_j .
\end{equation}
The advantage in the formulation \eqref{eq:eveqh'Is} is that the terms $ \cI_j $ on the right hand side are all nonlinear in $ h $. Furthermore, we observe that there are three families of contributions. When testing with $h'$ and integrating, the terms $\cI_j$ can be written in one of the following three ways:
\begin{align*}
\int \mathcal{N}(h,h')h'h' \dd s, && \int \mathcal{N}(h,h')\Lambda h h' \dd s
&& \text{ and } &&
\int \mathcal{N}(h,h')\Lambda h'h' \dd s,
\end{align*}
where $\mathcal{N}$ denotes a nonlinear term.

Once we are equipped with the estimates in $W^{1,\infty},$ this part is rather straightforward and as such we only sketch the proof. We start with the term $\cI_2$, we define
\begin{equation*}
\tilde{\cI}_2 = \int \cI_2\pare{s}h'\pare{s}\dd s, 
\end{equation*}
 thus, using the splitting defined in \eqref{eq:defcJ2} for the term $ \cJ_2 $, we find that
$$
\tilde{\cI}_{2}=\int \cI_2(s) h'(s)\dd s =\tilde{\cI}_{2,1}+\tilde{\cI}_{2,2}+\tilde{\cI}_{2,3},
$$
where $ \tilde{\cI}_{2,3} $ contains the lower order terms due to the cancellation of the linear part of the equation.

Using the Taylor expansion together with H\"older and Poincar\'e inequalities, we find that 
\begin{equation}
\label{eq:tildeI2,1_1}
\begin{aligned}
\tilde{\cI}_{2,1}&=-  \frac{1}{2}\frac{1}{2\pi} \iint    \frac{
 -  r' \partial_\alpha \theta \cos(\alpha) -  \pare{r-\theta} r' \sin(\alpha) +r \partial_\alpha\theta \sin \alpha
 }{4r^2 \sin^2\pare{\alpha / 2}}\frac{4r^2 \sin^2\pare{\alpha / 2}}{ 4r\pare{r-\theta}\sin^2\pare{\alpha / 2} + \theta^2}  \  \eta ' h'(s)   \dd \alpha \dd s\\
&\leq -  \frac{1}{2}\frac{1}{2\pi} \iint    \frac{
 -  r' \partial_\alpha \theta \cos(\alpha) -  \pare{r-\theta} r' \sin(\alpha) +r \partial_\alpha\theta \sin \alpha
 }{4r^2 \sin^2\pare{\alpha / 2}}\left[1 +\sum_{\ell=1}^\infty\left(\frac{\theta }{r } -\frac{\theta^2}{4r^2\sin^2\pare{\alpha/2} }\right)^\ell\right] \  \theta ' h'(s)   \dd \alpha \dd s\\
 &\quad + \cP\pare{\av{h}_{W^{1, \infty}}}   \av{h}_{W^{1, \infty}} \av{h'}_{L^2}^2.
 \end{aligned}
\end{equation}
 Let us define
 \begin{align*}
  W_{ \RN{1} } & =   \frac{1}{2}\frac{1}{2\pi} \iint    \frac{
   r' \partial_\alpha \theta 
 }{4r^2 \sin^2\pare{\alpha / 2}} \pare{1+\frac{\theta}{r}}  \theta ' h'(s)   \dd \alpha \dd s, \\
   W_{ \RN{2} } & =   \frac{1}{2}\frac{1}{2\pi} \iint    \frac{
  r r' \cot\pare{\alpha / 2}}{2r^2} \  \theta ' h'(s)   \dd \alpha \dd s , \\
   W_{ \RN{3} } & = -  \frac{1}{2}\frac{1}{2\pi} \iint    \frac{ r \partial_\alpha\theta \cot\pare{\alpha/2}
 }{2r^2 } \  \theta ' h'(s)   \dd \alpha \dd s.
 \end{align*}

With the above decomposition we find that
\begin{multline*}
-  \frac{1}{2}\frac{1}{2\pi} \iint    \frac{
 -  r' \partial_\alpha \theta \cos(\alpha) -  \pare{r-\theta} r' \sin(\alpha) +r \partial_\alpha\theta \sin \alpha
 }{4r^2 \sin^2\pare{\alpha / 2}}\left[1 +\sum_{\ell=1}^\infty\left(\frac{\theta }{r } -\frac{\theta^2}{4r^2\sin^2\pare{\alpha/2} }\right)^\ell\right] \  \theta ' h'(s)   \dd \alpha \dd s
\\
- W_{ \RN{1} } - W_{\RN{2} } - W_{ \RN{3} } = {\bf J} ,
\end{multline*}
where $ {\bf J} $ is a operator with a regularizing kernel satisfying
$$
{\bf J}\leq  \cP\pare{\av{h}_{W^{1, \infty}}}   \av{h}_{W^{1, \infty}} \av{h'}_{L^2}^2 .
 $$ 
Then, it suffices to control the singular part of the integral $ \tilde{\cI}_{2 , 1} $ composed by the simplified terms $ W_{j}, j=\RN{1}, \RN{2}, \RN{3} $.

We prove now the estimates for the terms $ W_{\RN{2}} $ and $ W_{\RN{3}} $. We perform the computations for the term  $ W_{\RN{2}} $ being the other identical. We use the boundedness of $ \av{h}_{W^{1, \infty}} $ in order to argue that
$$
 W_{ \RN{2} } =   -C \iint    \frac{
  r r' \cot\pare{\alpha / 2}}{2r^2} \  h'(s-\alpha) h'(s)   \dd \alpha \dd s= C \int   
  r r'  h'(s) \Lambda h(s) \dd s \leq \cP\pare{\av{h}_{W^{1, \infty}}} \av{h}_{W^{1, \infty}}
  $$    
  \begin{equation*}
 W_{\RN{3}} \leq \cP\pare{\av{h}_{W^{1, \infty}}} \av{h}_{W^{1, \infty}} \iint \av{\cot\pare{\alpha/2}} \av{\theta'}\dd\alpha \dd s   
\end{equation*}
thus we  use the embedding $ L^2\pare{\bS^2}\hra L^1\pare{\bS^2} $ and the fact that $ \av{\cot\pare{\alpha/2}}^2\lesssim \pare{ \sin\pare{\alpha/2} }^{-2} $ and we deduce the control
\begin{equation*}
W_{\RN{3}} \leq \cP\pare{\av{h}_{W^{1, \infty}}} \av{h}_{W^{1, \infty}} \av{\Lambda^{1/2} h'}_{L^2}. 
\end{equation*}
We conclude that
\begin{equation}\label{eq:WII+WIII}
W_{\RN{2}} + W_{\RN{3}} \leq \nu  \av{\Lambda^{1/2} h}_{L^2}^2 + \frac{\cP\pare{\av{h}_{W^{1, \infty}}} }{\nu} \av{h}_{W^{1, \infty}}^2. 
\end{equation}
We study now the term $ W_{\RN{1}} $, which is the more singular of the three. Let us reformulate it as
\begin{equation*}
W_{\RN{1}} = W_{\RN{1}, \RN{1}} + W_{\RN{1}, \RN{2}}, 
\end{equation*}
where
\begin{align*}
W_{\RN{1}, \RN{1}} &  =   \frac{1}{2}\frac{1}{2\pi} \iint    \frac{
   r' \partial_\alpha \theta 
 }{4r^2 \sin^2\pare{\alpha / 2}}  \theta ' h'(s)   \dd \alpha \dd s, \\
 W_{\RN{1}, \RN{2}}  & =   \frac{1}{2}\frac{1}{2\pi} \iint    \frac{
   r' \theta \partial_\alpha \theta 
 }{4r^3 \sin^2\pare{\alpha / 2}}  \theta ' h'(s)   \dd \alpha \dd s.
\end{align*} 
The term $ W_{\RN{1}, \RN{2}}  $ can be controlled using computations close to the ones performed to control the terms $ W_{\RN{2}} $ and $ W_{\RN{3}} $ obtaining that
\begin{equation}\label{eq:WI,II}
W_{\RN{1}, \RN{2}}  \leq \nu  \av{\Lambda^{1/2} h}_{L^2}^2 + \frac{\cP\pare{\av{h}_{W^{1, \infty}}} }{\nu} \av{h}_{W^{1, \infty}}^2. 
\end{equation}
The term $ W_{\RN{1}, \RN{1}} $ is indeed the more singular one. Substituting the explicit values of the functions $ r=1+h\pare{s} $ and $ \theta = h\pare{s} - h\pare{s-\alpha} $ and changing variables, we obtain that
\begin{align*}
W_{\RN{1}, \RN{1}} & = \frac{1}{2}\frac{1}{2\pi} \iint \frac{\pare{h'\pare{s}}^2 h'\pare{s-\alpha}}{4\pare{1+h\pare{s}}^2\sin^2\pare{\alpha/2}} \pare{h'\pare{s} - h'\pare{s-\alpha}}\dd\alpha  \dd s, \\
& = \frac{1}{2}\frac{1}{2\pi} \iint \frac{\pare{h'\pare{s}}^2 h'\pare{\sigma}}{4\pare{1+h\pare{s}}^2\sin^2\pare{\frac{s-\sigma}{2}}} \pare{h'\pare{s} - h'\pare{\sigma}}\dd\sigma  \dd s , \\
& = -\frac{1}{2}\frac{1}{2\pi} \iint \frac{\pare{h'\pare{\sigma}}^2 h'\pare{s}}{4\pare{1+h\pare{\sigma }}^2\sin^2\pare{\frac{s-\sigma}{2}}} \pare{h'\pare{s} - h'\pare{\sigma}}\dd\sigma  \dd s.
\end{align*}
Then, we find that
 \begin{equation}
 \label{eq:WI,I}
 \begin{aligned}
 W_{\RN{1}, \RN{1}} & = \frac{1}{2}\frac{1}{2\pi} \iint \frac{h'\pare{s} h'\pare{\sigma}\theta'}{2\sin^2\pare{\frac{s-\sigma}{2}}} \bra{ \frac{h'\pare{s}}{\pare{1+h\pare{s}}^2} - \frac{h'\pare{\sigma}}{\pare{1+h\pare{\sigma}}^2} }\dd\sigma \dd s, \\
 & \leq C \av{h}_{W^{1, \infty}}^2 \pare{ \iint \pare{\frac{\theta'	}{2\sin \pare{\frac{s-\sigma}{2}}}}^2 \dd s \dd \sigma }^{1/2} \pare{ \iint  \bra{ \frac{\frac{h'\pare{s}}{\pare{1+h\pare{s}}^2} - \frac{h'\pare{\sigma}}{\pare{1+h\pare{\sigma}}^2} }{2\sin\pare{\frac{s-\sigma}{2}}}     }^2 \dd\sigma \dd s }^{1/2}, \\
 & \leq \cP\pare{\av{h}_{W^{1, \infty}}}\av{h}_{W^{1, \infty}} \av{\Lambda^{1/2}h'}_{L^2}\av{\Lambda^{1/2}\pare{ \frac{h'}{\pare{1+h}^2} }}_{L^2}, \\
  & \leq \cP\pare{\av{h}_{W^{1, \infty}}}\av{h}_{W^{1, \infty}}  \av{\Lambda^{1/2} h' }_{L^2}^2 + \cP\pare{\av{h}_{W^{1, \infty}}}  \av{h}_{W^{1, \infty}}^2. 
 \end{aligned}
 \end{equation}
 
Using the smallness of $h$ in $W^{1,\infty}$, the rest of terms can be handled similarly and we find that 
$$
\tilde{\cI}_{2}\leq \nu  \av{\Lambda^{1/2} h' }_{L^2}^2 + \frac{\cP\pare{\av{h}_{W^{1, \infty}}} }{\nu} \av{h}_{W^{1, \infty}}^2. 
$$

 We define
$$
\tilde{\cI}_1=\int \cI_1(s) h'(s) \ \dd s
$$
and decompose $\mathcal{J}_1$ as in \eqref{eq:dec_cJ1_1} to find that
\begin{multline*}
\tilde{\cI}_1\leq -\frac{1}{2}\frac{1}{r^2}\frac{1}{2\pi}\iint  \frac{\theta \theta' \cos (\alpha) }{2\sin^2\pare{\alpha / 2} }\bra{ 1 +\sum_{\ell=1}^\infty\left(\frac{\theta }{r } -\frac{\theta^2}{4r^2\sin^2\pare{\alpha/2} }\right)^\ell
} \ \partial_\alpha \bra{  \sin(\alpha) \pare{r-\theta} }h'(s)  \dd \alpha \dd s\\
 +\cP\pare{\av{h}_{W^{1, \infty}}}   \av{h}_{W^{1, \infty}} \av{h'}_{L^2}^2 .
\end{multline*}
In order to deduce the above estimate we used the fact that the integral defining the term $ \cJ_{1,1} $ is not singular. 
The term $ \tilde{\cI}_{1} $ is now in a form which resembles the one deduced for the term $ \tilde{\cI}_{2,1} $ in equation \eqref{eq:tildeI2,1_1}. Very similar computations allow us to produce the bound
\begin{equation*}
\tilde{\cI}_{1}\leq \nu  \av{\Lambda^{1/2} h' }_{L^2}^2 + \frac{\cP\pare{\av{h}_{W^{1, \infty}}} }{\nu} \av{h}_{W^{1, \infty}}^2. 
\end{equation*}

We decompose $\cI_3$ following our splitting of $\cJ_3$ (see Section \ref{sec:cJ3}) and find that 
\begin{align*}
\tilde{\cI}_{3}= & \ \int \cI_3(s) h'(s)\dd s , \\
= & \ \tilde{\cI}_{3,1}+\tilde{\cI}_{3,2}+\tilde{\cI}_{3,3} , \\
= & \ \tilde{\cI}_{3,1,\RN{1}}+\tilde{\cI}_{3,1,\RN{2}}+\tilde{\cI}_{3,2,\RN{1}}+\tilde{\cI}_{3,2,\RN{2}}+\tilde{\cI}_{3,2,\RN{3}}+\tilde{\cI}_{3,3},
\end{align*}
where, as before, $\tilde{\cI}_{3,3} $ contains lower order terms due to the cancellation of the linear part of the equation. The terms $\cI_{3,1,j}, \ j=\RN{1}, \RN{2}$ can be estimated using computations similar to the ones performed in order to control $ \tilde{\cI}_{2, 1} $ and we find
$$
\tilde{\cI}_{3,1}\leq \nu  \av{\Lambda^{1/2} h}_{L^2}^2 + \frac{\cP\pare{\av{h}_{W^{1, \infty}}} }{\nu} \av{h}_{W^{1, \infty}}^2. 
$$
Now we estimate the term $\cI_{3,2,\RN{1}}$, the other terms being easier. We find that
\begin{align*}
\tilde{\cI}_{3,2,\RN{1}}&=-\frac{1}{4\pi} \iint \frac{r'\pare{r-\theta} \pare{1+\cos^2\alpha}}{2r\pare{r-\theta} + \frac{\theta^2}{2\sin^2\pare{\alpha / 2}}}\frac{\theta}{2\sin^2\pare{\alpha/2}} h'(s)\dd \alpha \dd s.
 \end{align*}
We study now the more singular contribution of $ \tilde{\cI}_{3, 2, \RN{1}} $, i.e. the integral
\begin{equation*}
-\frac{1}{4\pi} \iint  \frac{r'}{r}  \frac{\theta}{2\sin^2\pare{\alpha/2}} h'(s)\dd \alpha \dd s = -\frac{1}{4\pi} \int  \frac{ \Lambda h\pare{s} h'(s) }{1+h\pare{s}}  \dd s \leq  \nu  \av{\Lambda^{1/2} h}_{L^2}^2 + \frac{\cP\pare{\av{h}_{W^{1, \infty}}} }{\nu} \av{h}_{W^{1, \infty}}^2. 
\end{equation*} 
 The rest of terms can be handled in a similar way and we conclude that 
$$
\tilde{\cI}_{3}\leq \nu  \av{\Lambda^{1/2} h}_{L^2}^2 + \frac{\cP\pare{\av{h}_{W^{1, \infty}}} }{\nu} \av{h}_{W^{1, \infty}}^2.
$$ 
The term $\tilde{\cI}_4$ resembles the term $\cI_3$ and as a consequence it can be handled using the same ideas. Then we obtain that
$$
\tilde{\cI}_{4}\leq \nu  \av{\Lambda^{1/2} h}_{L^2}^2 + \frac{\cP\pare{\av{h}_{W^{1, \infty}}} }{\nu} \av{h}_{W^{1, \infty}}^2.
$$ 
The term $\tilde{\cI}_5$ is similar to $\cI_{3,2,\RN{1}}$ and then we find that
$$
\tilde{\cI}_{5}\leq \nu  \av{\Lambda^{1/2} h}_{L^2}^2 + \frac{\cP\pare{\av{h}_{W^{1, \infty}}} }{\nu} \av{h}_{W^{1, \infty}}^2.
$$ 
The term $\tilde{\cI}_6$ can be estimated using the previous Taylor expansion together with the same ideas used to bound $\tilde{\cI}_1$. Then, choosing $\nu$ small enough and using the maximum principle for $\av{h}_{W^{1,\infty}}$, we conclude

$$
\ddt \av{h'}_{L^2}^2+\frac{1}{4}\av{\Lambda^{1/2}h'}_{L^2}^2\leq \cP\pare{\av{h}_{W^{1, \infty}}}   \av{h}_{W^{1, \infty}}^2\leq C.
$$
Invoking now Gronwall's inequality we find that
\begin{equation*}
\int_{0}^T \av{h\pare{t}}_{H^{3/2}}^2\dd t\leq C T.
\end{equation*}
\end{proof}

\section{Estimates for $ \partial_t h $} \label{sec:path}
The result we prove in the present section is the following one
\begin{prop}\label{prop:unif_bounds_path}
Let  $ T^\star \in(0,\infty]$ and $h = h(s,t)$ be a $\cC \pare{ \bra{ 0,T^\star}; \cC^2 } $ solution of \eqref{eq:eveqh2} such that $ \left. h\right|_{t=0} = h_0 $. Assume that
$$ 
\av{h_0}_{W^{1, \infty}}< c_0 
$$ 
with $c_0$ the constant in Proposition \ref{prop:W1infty_enest}. Then, for all $T\leq T^*$, there exists a  $ C(T) \in \pare{0, \infty} $ depending on $ T $ only such that 
\begin{equation*}
\partial_t h \in L^2  \pare{\bra{0, T}; H^{-1} \pare{\bS^1}}, 
\end{equation*}
and the following bound holds true
\begin{equation*}
\norm{\partial_t h}_{ L^2  \pare{\bra{0, T}; H^{-1} \pare{\bS^1}}} \leq C(T).
\end{equation*}
\end{prop}

\begin{proof}
Thanks to the regularity results proved in the previous sections it suffices to prove a suitable  bound for the nonlinear terms in the evolution equation for $ h $. The bounds of Proposition \ref{prop:unif_bounds_path} are necessary in the application of an Aubin-Lions compactness theorem (cf. \cite{Simon87}) and are somewhat standard, for this reason we will sketch the computations only for the more singular terms and leave the rest of the computations for the interested reader. 

The term $ J_2 $ is the more singular term in \eqref{eq:J's} due to the presence of the term $ \partial_\alpha^2 \theta $. Using the notation of Section \ref{sec:J2} the term
\begin{equation*}
J_{2, 2}  = 
\ - \frac{1}{4\pi} \pv \int \partial_\alpha \pare{ \frac{\bra{2r\sin^2\pare{\alpha / 2} + \theta \cos(\alpha)} \bra{2r\sin^2\pare{\alpha/2}-\theta}}{4r\pare{r-\theta} \sin^2\pare{\alpha/2} +\theta^2} } \ \partial_\alpha \theta \  \dd \alpha , 
\end{equation*}
is the more singular  of the subterms making $ J_2 $. The term $ J_{2, 2} $, can be decomposed as in the previous section
\begin{equation*}
J_{2, 2} = J_{2, 2, \RN{1}} + J_{2, 2, \RN{2}} + J_{2, 2, \RN{3}}.
\end{equation*}
Let us denote with $ \overline{J_{2, 2, j}}, \ j=\RN{1}, \RN{2}, \RN{3} $ the more singular contributions of the terms $ J_{2, 2, j} , \ j=\RN{1}, \RN{2}, \RN{3} $, whose explicit expressions are
\begin{equation*}
\begin{aligned}
\overline{J_{2, 2, \RN{1}} }= & \  - \frac{1}{4\pi} \pv \int \frac{ 2r^2 \sin(\alpha) + 2\theta\partial_\alpha \theta}{\pare{2 r \sin\pare{\alpha / 2} }^4}   \theta^2 \partial_\alpha \theta \ \dd \alpha , \\
\overline{J_{2, 2, \RN{2}}} = & \   \frac{1}{4\pi} \pv \int  \frac{ 
 \theta \pare{ \partial_\alpha \theta  }^2 
}{\pare{2r \sin\pare{\alpha / 2}}^2 }  \ \dd \alpha , \\
\overline{J_{2, 2, \RN{3}}} = & \   \frac{1}{4\pi} \pv \int \frac{ \theta \pare{ \partial_\alpha \theta  }^2  }{\pare{2r \sin\pare{\alpha / 2}}^2}  \ \dd \alpha . 
\end{aligned}
\end{equation*}
Let us remark that the terms $ J_{2, 2, j} - \overline{J_{2, 2, j}}, \ j=\RN{1}, \RN{2}$ can be written as integral operator whose integration kernel is homogeneous of order zero. In particular the following bound holds true for any $ t\in \bra{0, T} $ and $ s\in\bS^1 $
\begin{equation*}
\av{J_{2, 2, j} \pare{s, t}  - \overline{J_{2, 2, j}}\pare{s, t}} \leq \cP \pare{\av{h\pare{t}}_{W^{1, \infty}}}. 
\end{equation*} As a consequence, they are more regular contributions and we can consider any $ \phi \in L^2 \pare{\bra{0, T}; H^1}, \  T \in\pare{ 0, T^\star }  $ with unitary norm and deduce the estimate
\begin{equation*}
\int_0^{T} \int \pare{J_{2, 2, j} \pare{s, t}  - \overline{J_{2, 2, j}}\pare{s, t}}  \phi\pare{s, t}\dd s \ \dd t \leq \cP\pare{\norm{h}_{L^\infty\pare{\bra{0, T}; W^{1, \infty}}}}. 
\end{equation*}
Let $ \phi $ be as above. We will indeed bound the remaining more singular terms by duality. Let us first focus on the term $\overline{J_{2, 2, \RN{2}}}$. We compute
\begin{multline*}
\int \pare{   \int \frac{ \theta \pare{ \partial_\alpha \theta  }^2  }{\pare{2r \sin\pare{\alpha / 2}}^2}  \ \dd \alpha } \phi\pare{s} \ds = \iint \frac{h\pare{s} - h\pare{s-\alpha}}{4 \sin^2\pare{\alpha/2}} \ \pare{h'\pare{s-\alpha}}^2 \ \frac{\phi\pare{s}}{\pare{ 1+h\pare{s} }^2 } \dd \alpha \ \ds \\
= - \iint \frac{h\pare{z} - h\pare{z-\beta}}{4 \sin^2\pare{\beta/2}} \ \pare{h'\pare{z}}^2 \ \frac{\phi\pare{z-\beta}}{\pare{ 1+h\pare{z-\beta} }^2 } \dd \beta \ \dd z, 
\end{multline*}
where in the last identity we used the change of variables $ s-\alpha = z, \  \beta = -\alpha $. We hence symmetrized the term $ \overline{J_{2, 2, \RN{2}}},$ as
\begin{multline}
\label{eq:nonlinear_cancellation}
\int \pare{  \int \frac{ \theta \pare{ \partial_\alpha \theta  }^2  }{\pare{2r \sin\pare{\alpha / 2}}^2}  \ \dd \alpha } \phi\pare{s} \ds \\ 
\begin{aligned}
= & \frac{1}{2}\iint \frac{h\pare{s} - h\pare{s-\alpha}}{4 \sin^2\pare{\alpha/2}} \bra{ \pare{h'\pare{s-\alpha}}^2 \ \frac{\phi\pare{s}}{\pare{ 1+h\pare{s} }^2} - \pare{h'\pare{s}}^2 \ \frac{\phi\pare{s-\alpha}}{\pare{ 1+h\pare{s-\alpha} }^2 }} \dd \alpha \ \ds, \\
= & \frac{1}{2} \iint \frac{h\pare{s} - h\pare{\sigma }}{4 \sin^2\pare{\frac{s-\sigma}{2}}} \bra{ \pare{h'\pare{\sigma}}^2 \ \frac{\phi\pare{s}}{\pare{ 1+h\pare{s} }^2} - \pare{h'\pare{s}}^2 \ \frac{\phi\pare{\sigma }}{\pare{ 1+h\pare{\sigma}}^2  } } \dd \sigma \ \ds.
\end{aligned}
\end{multline}

We compute
\begin{multline*}
 \pare{h'\pare{\sigma}}^2 \ \frac{\phi\pare{s}}{\pare{ 1+h\pare{s} }^2} - \pare{h'\pare{s}}^2 \ \frac{\phi\pare{\sigma }}{\pare{ 1+h\pare{\sigma}}^2  } \\
  =  \pare{h'\pare{\sigma}}^2  \pare{ \frac{\phi\pare{s}}{\pare{ 1+h\pare{s} }^2}   - \frac{\phi\pare{\sigma }}{\pare{ 1+h\pare{\sigma } }^2}  } - \frac{\phi\pare{\sigma }}{\pare{ 1+h\pare{\sigma } }^2} \pare{h'\pare{\sigma}+ h'\pare{s}}\pare{h'\pare{s}- h'\pare{\sigma} } , 
\end{multline*}
so that
\begin{multline*}
\int \pare{  \int \frac{ \theta \pare{ \partial_\alpha \theta  }^2  }{\pare{2r \sin\pare{\alpha / 2}}^2}  \ \dd \alpha } \phi\pare{s} \ds  \\
\begin{aligned}
 = & \ \frac{1}{2} \iint \frac{h\pare{s} - h\pare{\sigma }}{4 \sin^2\pare{\frac{s-\sigma}{2}}} \pare{h'\pare{\sigma}}^2  \pare{ \frac{\phi\pare{s}}{\pare{ 1+h\pare{s} }^2}   - \frac{\phi\pare{\sigma }}{\pare{ 1+h\pare{\sigma } }^2}  } \dd \sigma \ \dd s \\
& - \frac{1}{2} \iint \frac{h\pare{s} - h\pare{\sigma }}{4 \sin^2\pare{\frac{s-\sigma}{2}}} \frac{\phi\pare{\sigma }}{\pare{ 1+h\pare{\sigma } }^2} \pare{h'\pare{\sigma}+ h'\pare{s}}\pare{h'\pare{s}- h'\pare{\sigma} }  \dd \sigma \ \dd s 
& = M_1 +M_2 .
\end{aligned}
\end{multline*}
We start analyzing $ M_2 $. A H\"older inequality provides the bound
\begin{equation*}
\begin{aligned}
M_2 \leq & \ C \ \frac{\av{\phi}_{L^\infty} \av{h'}_{L^\infty}}{1-\av{h}_{L^\infty}} \pare{\iint \pare{\frac{h\pare{s} - h\pare{\sigma}}{2\sin\pare{\frac{s-\sigma}{2}}}}^2 \dd \sigma \ \dd s}^{1/2} \pare{\iint \pare{\frac{h' \pare{s} - h ' \pare{\sigma}}{2\sin\pare{\frac{s-\sigma}{2}}}}^2 \dd \sigma \ \dd s}^{1/2}, \\
= & C \ \frac{\av{\phi}_{L^\infty} \av{h'}_{L^\infty}}{1-\av{h}_{L^\infty}} \av{\Lambda^{1/2} h}_{L^2} \av{\Lambda^{1/2} h' }_{L^2} , \\
\leq & \cP\pare{\av{h}_{W^{1, \infty}}} \av{\phi}_{H^1} \av{\Lambda^{1/2} h'}_{L^2}. 
\end{aligned}
\end{equation*}
We control now the term $ M_1 $ as
\begin{equation*}
\begin{aligned}
M_1 \leq & \ C \av{h'}_{L^\infty}^2  \pare{\iint \pare{\frac{h\pare{s} - h\pare{\sigma}}{2\sin\pare{\frac{s-\sigma}{2}}}}^2 \dd \sigma \ \dd s}^{1/2}  \pare{\iint \pare{\frac{\frac{\phi\pare{s}}{\pare{1+h\pare{s}}^2} - \frac{\phi\pare{\sigma}}{\pare{1+h\pare{\sigma}}^2}}{2\sin\pare{\frac{s-\sigma}{2}}}}^2 \dd \sigma \ \dd s}^{1/2}, \\
= & \ C \av{h'}_{L^\infty}^2  \av{\Lambda^{1/2} h}_{L^2} \av{\Lambda^{1/2}\pare{\frac{\phi}{\pare{1+h}^2}}}_{L^2}, \\
\leq & \  \cP\pare{\av{h}_{W^{1, \infty}}} \av{\phi}_{H^1}. 
\end{aligned}
\end{equation*}
The bounds provided for $ M_1 $ and $ M_2 $ allow us to argue that 
\begin{equation*}
\int_0^{T} \overline{J_{2, 2,\RN{2}}}\pare{s, t} \ \phi\pare{s, t} \dd s\ \dd t \leq \cP\pare{\norm{h}_{L^\infty\pare{\bra{0, T}; W^{1, \infty}}}} \norm{\phi}_{L^2\pare{\bra{0, T}; H^1}} \pare{\norm{\Lambda^{1/2}h'}_{L^2\pare{\bra{0, T}; L^2}} + \sqrt{T}}. 
\end{equation*}
Similar bounds hold true for $ \overline{J_{2, 2,\RN{1}}} $. We hence proved that
\begin{equation*}
\norm{J_2}_{L^2\pare{\bra{0, T}; H^{-1}}} \leq C_T \cP\pare{\norm{h}_{L^\infty\pare{\bra{0, T}; W^{1, \infty}}}}  \pare{1 + \norm{\Lambda^{1/2}h'}_{L^2\pare{\bra{0, T}; L^2}}}.
\end{equation*}
Following the same ideas, we can obtain appropriate bounds for the terms $ J_1 $ and $ J_3 $. These estimates combined with the result of Proposition \ref{prop:unif_bounds_L2h} allow us to conclude the proof of Proposition \ref{prop:unif_bounds_path}. 
\end{proof}

\section{Proof of Theorem \ref{teo1}}\label{sec7}

In the present section we prove the main result of the manuscript via an approximation and compactness argument. Let us consider the regularized problem
\begin{equation}
\label{eq:Peskin_approximated}
\left\lbrace
\begin{aligned}
& \partial_t h_\varepsilon + \Lambda h_\varepsilon -\varepsilon h''_\varepsilon  = \cN\pare{h_\varepsilon}, \\
& \left. h^\varepsilon\right|_{t=0} = \eta^\varepsilon \star  h_0, 
\end{aligned}
\right. 
\end{equation}
where for $ \varepsilon > 0 $, $ s\in\bS^1 $, the function $ \eta^\varepsilon $ is the periodic heat kernel at time $\varepsilon$ and the nonlinearity $ \cN $ is defined as
\begin{equation}\label{eq:cN}
\cN\pare{h_\varepsilon} = J_1\pare{h_\varepsilon} + J_2\pare{h_\varepsilon} + J_3\pare{h_\varepsilon} + \Lambda h_\varepsilon - \frac{1}{4} h_\varepsilon \star \cos, 
\end{equation}
and the terms $ J_k = J_k\pare{h}, \ k=1, 2, 3 $ are defined in \eqref{eq:eveqh2}.

Using Picard's Theorem together with the standard mollifier approach and energy estimates (see \cite{bertozzi2001vorticity}) we can prove that, fixed $ \varepsilon > 0 $, there exists a $ T_\varepsilon \in \left( 0, \infty\right] $ and a maximal solution $ h_\varepsilon $ of \eqref{eq:Peskin_approximated} which belongs to the space
\begin{align}\label{eq:analytic_regularity_approx_solutions}
h_\varepsilon \in \cC^1\pare{\left[ 0, T_\varepsilon \right) ; H^3}.
\end{align}
At this point these approximate solutions may be defined only locally in time. Furthermore, using that our approximation scheme is merely a vanishing viscosity approach, this solution satisfies the same \emph{a priori} bounds in $L^\infty\pare{ 0,T;W^{1,\infty} } $ and $L^2 \pare{ 0,T;H^{3/2} }$ stated in Propositions \ref{prop:W1infty_enest}, \ref{prop:unif_bounds_L2h} and \ref{prop:unif_bounds_path}. Furthermore, we can prove the following $ L^2 $ estimate for \eqref{eq:Peskin_approximated}. We have indeed that for $ t\in\bra{0, T_\varepsilon } $
\begin{multline*}
\frac{1}{2} \ddt \av{h_\varepsilon\pare{t}}_{L^2}^2 +   \av{\Lambda^{1/2} h_\varepsilon\pare{ t}}_{L^2}^2 +\varepsilon \av{h_\varepsilon ' \pare{ t}}_{L^2}^2 
\leq \av{\cN\pare{h_\varepsilon}}_{H^{-1}} \av{h_\varepsilon}_{H^1} \\
  \leq \frac{C}{\varepsilon} \av{\cN\pare{ h_\varepsilon \pare{t} }}_{H^{-1}}^2 + \frac{\varepsilon }{2} \pare{ \av{h_\varepsilon\pare{t}}_{L^2}^2 +  \av{h_\varepsilon '\pare{t}}_{L^2}^2  }.
\end{multline*}
As a consequence, an integration-in-time gives that
\begin{equation*}
 \av{h_\varepsilon\pare{t}}_{L^2}^2 + \int_0 ^t \pare{\av{\Lambda^{1/2} h_\varepsilon\pare{\tau}}_{L^2}^2 +\varepsilon \av{h_\varepsilon ' \pare{\tau}}_{L^2}^2 } \dd \tau \leq \av{h_0}_{L^2}^2 e^{\frac{C t}{\varepsilon}} + \frac{C e^{\frac{C t}{\varepsilon}} }{\varepsilon} \av{ \cN\pare{h_\varepsilon}}_{L^2\pare{\bra{0, t}; H^{-1}}}^2. 
\end{equation*}
The computations performed in the proof of Proposition \ref{prop:unif_bounds_path} assure us that $ \cN \pare{h_\varepsilon} \in L^2 \pare{\bra{0, T_\varepsilon}; H^{-1}}$, so that
\begin{align*}
\av{h_\varepsilon\pare{t}}_{L^2}^2 + \int_0 ^t \pare{\av{\Lambda^{1/2} h_\varepsilon\pare{\tau}}_{L^2}^2 +\varepsilon \av{h_\varepsilon ' \pare{\tau}}_{L^2}^2 } \dd \tau \leq \av{h_0}_{L^2}^2 + \frac{C\pare{ T_\varepsilon, \varepsilon}}{\varepsilon} \pare{1 + \av{h_\varepsilon}_{L^\infty\pare{\bra{0, T_\varepsilon}; W^{1, \infty}}}^N } , && N\gg 1.
\end{align*}
We recall now the result of Proposition \ref{prop:W1infty_enest}, which ensures us that $  \norm{h_\varepsilon}_{L^\infty\pare{\bra{0, T_\varepsilon}; W^{1, \infty}}}\leq c_0 $. This allow us to bound the right hand side of the above inequality with a quantity which is independent of $ h_\varepsilon $. A continuation argument for ODEs allow us to bootstrap the result, thus proving the following bound
\begin{equation*}
\av{h_\varepsilon\pare{t}}_{L^2}^2 + \int_0 ^t \pare{\av{\Lambda^{1/2} h_\varepsilon\pare{\tau}}_{L^2}^2 +\varepsilon \av{h_\varepsilon ' \pare{\tau}}_{L^2}^2 } \dd \tau \leq c_0^2 + \frac{C\pare{T, \varepsilon}}{\varepsilon} \pare{1 + c_0^N } . 
\end{equation*}
Similarly, using standard energy estimates together with the ideas in the previous sections and Propositions \ref{prop:W1infty_enest}, \ref{prop:unif_bounds_L2h} and \ref{prop:unif_bounds_path} it is possible to prove that if
\begin{equation*}
\av{h_\varepsilon ^{\pare{n-1}} \pare{t}}_{L^2}^2 + \int_0^t \pare{  \av{\Lambda^{1/2}h_\varepsilon^{\pare{n-1}} \pare{\tau}}_{L^2}^2 + \frac{\varepsilon}{2} \av{h_\varepsilon ^{\pare{n}} \pare{\tau}}_{L^2}^2 } \dd \tau \leq C_{n-1}(\varepsilon, T , c_0) , 
\end{equation*}
then
\begin{equation*}
\av{h_\varepsilon ^{\pare{n}} \pare{t}}_{L^2}^2 + \int_0^t \pare{  \av{\Lambda^{1/2}h_\varepsilon^{\pare{n}} \pare{\tau}}_{L^2}^2 + \frac{\varepsilon}{2} \av{h_\varepsilon ^{\pare{n+1}} \pare{\tau}}_{L^2}^2 } \dd \tau \leq C_{n}(\varepsilon, T , c_0) , 
\end{equation*}
for any $ n\geq 2 $, 
where the constants $ C_j\pare{\varepsilon, T}, \ j\geq 2 $ are \textit{not} uniformly bounded in $ \varepsilon $. \\

 We can thus find that 
\begin{equation}\label{eq:global_est_H3_approx_Peskin}
\av{h_\varepsilon\pare{t}}_{H^3}^2 + \int_0^t \pare{  \av{\Lambda^{1/2}h_\varepsilon\pare{\tau}}_{H^3}^2 + \frac{\varepsilon}{2} \av{h_\varepsilon ' \pare{\tau}}_{H^3}^2 } \dd \tau \leq C(\varepsilon, T , c_0 ). 
\end{equation}
In particular, we find that the approximate solutions are smooth and global in time.  

The global bounds of \eqref{eq:global_est_H3_approx_Peskin} allow us to apply the regularity results stated in Propositions \ref{prop:W1infty_enest}, \ref{prop:unif_bounds_L2h} and \ref{prop:unif_bounds_path}. Using a standard Aubin-Lions compactness theorem (cf. \cite[Corollary 4]{Simon87}), we find that
\begin{equation}
\label{eq:weak_convergence}
\begin{aligned}
h_\varepsilon\rightarrow h & \text{ in }L^2 \pare{ 0,T;H^{\frac{3}{2}-\vartheta} }, \ \vartheta > 0, \\
h_\varepsilon \rightharpoonup h &\text{ in }L^2 \pare{ 0,T;H^{3/2} }, \\
h_\varepsilon \xrightharpoonup{\ast} h & \text{ in }L^\infty \pare{ 0,T;L^\infty } ,
 \\
h'_\varepsilon \xrightharpoonup{\ast} h' & \text{ in }L^\infty \pare{ 0,T;L^\infty }.
\end{aligned}
\end{equation}

We now  take $\varphi\in C^\infty_{c}\pare{ [0,T)\times\mathbb{S}^1 } $ and consider the weak formulation of the approximate problems
\begin{multline}\label{eq:weak_eq_}
-\int \varphi(s,0)\eta^\varepsilon \star  h_0(s) \dd s \\
+\int_0^T\int \Pare{ -\partial_t\varphi (s,t) h_\varepsilon (s,t) + \Lambda\varphi (s,t) h_\varepsilon(s,t) -\varepsilon \varphi''(s,t)h_\varepsilon(s,t)  - \cN\pare{h_\varepsilon(s,t)}\varphi(s,t) } \dd s \ \dd t=0.
\end{multline}

The previous regularity and convergence results are enough in order to pass to the limit in the nonlinear terms. Let us sketch why it is so. Let us use the notation $ \theta_\varepsilon = h_\varepsilon\pare{s}-h_\varepsilon\pare{s-\alpha} $ and $ r_\varepsilon = 1+h_\varepsilon\pare{s} $. We use an argument similar the one stated in Section \ref{sec:path} to argue that the term
\begin{equation*}
Z_\varepsilon \pare{s, t} = \int \frac{\theta_\varepsilon \pare{\partial_\alpha \theta_\varepsilon}^2}{\pare{ 2 r_\varepsilon \sin\pare{\alpha / 2} }^2 } \dd\alpha , 
\end{equation*}
is the more singular contribution of the many composing $ \cN\pare{h_\varepsilon\pare{s}} $, hence, defining
\begin{equation*}
Z \pare{s, t} = \int \frac{\theta \pare{\partial_\alpha \theta }^2}{\pare{ 2 r \sin\pare{\alpha / 2} }^2 } \dd\alpha , 
\end{equation*}
we aim to prove that
\begin{equation*}
\cZ_\varepsilon =
\int_0^T \int \pare{Z_\varepsilon\pare{s, t} - Z\pare{s, t}}\varphi\pare{s, t}\dd s \ \dd t \xrightarrow{\varepsilon\to 0}0, 
\end{equation*}
for each $\varphi\in C^\infty_{c} \pare{ [0,T)\times\mathbb{S}^1 }$. This will establish the weak convergence for the more singular nonlinear term in $ \cN $. We write $ \cZ_\varepsilon = \cZ_{\varepsilon, 1} + \cZ_{\varepsilon, 2} $ where
\begin{align*}
\cZ_{\varepsilon, 1} & = \int_0^T \int \pare{ \int \frac{\pare{ \theta_\varepsilon -\theta  }\pare{\partial_\alpha \theta_\varepsilon}^2}{\pare{ 2 r_\varepsilon \sin\pare{\alpha / 2} }^2 } \dd\alpha } \varphi\pare{s, t} \dd s \ \dd t, \\
\cZ_{\varepsilon, 2} & = \int_0^T \int \pare{ \int \frac{\theta   \pare{\partial_\alpha \theta_\varepsilon + \partial_\alpha \theta} \pare{\partial_\alpha \theta_\varepsilon - \partial_\alpha \theta}}{\pare{ 2 r_\varepsilon \sin\pare{\alpha / 2} }^2 } \dd\alpha } \varphi\pare{s, t} \dd s \ \dd t .
\end{align*}
Let us symmetrize the term  $ \cZ_{\varepsilon, 1} $. This gives that
\begin{align*}
\cZ_{\varepsilon, 1} = & \  \int_0^T \iint \frac{\pare{\theta_\varepsilon - \theta} h'\pare{\sigma} \varphi\pare{s}}{4\pare{1+h_\varepsilon\pare{s}}^2 \sin^2\pare{\frac{s-\sigma}{2}}} \dd s  \dd \sigma \  \dd t - \int_0^T \iint \frac{\pare{\theta_\varepsilon - \theta} h'\pare{s} \varphi\pare{ \sigma }}{4\pare{1+h_\varepsilon\pare{\sigma}}^2 \sin^2\pare{\frac{s-\sigma}{2}}} \dd s  \dd \sigma \  \dd t, \\
= & \ \int_0^T \iint \frac{\theta_\varepsilon - \theta}{4 \sin^2\pare{\frac{s-\sigma}{2}}}  
\set{
\frac{h'\pare{\sigma} \varphi\pare{s}}{2\pare{1+h_\varepsilon\pare{s}}^2}
- 
\frac{h'\pare{s} \varphi\pare{\sigma}}{2\pare{1+h_\varepsilon\pare{\sigma }}^2} 
} \dd s \dd \sigma \ \dd t , \\
= & \ \int_0^T \iint \frac{\theta_\varepsilon - \theta}{2 \sin\pare{\frac{s-\sigma}{2}}}\frac{1}{2 \sin\pare{\frac{s-\sigma}{2}}} \\
& \qquad \qquad \times
\set{ h'\pare{\sigma} \bra{
\frac{ \varphi\pare{s}}{2\pare{1+h_\varepsilon\pare{s}}^2}
- 
\frac{ \varphi\pare{\sigma }}{2\pare{1+h_\varepsilon\pare{\sigma }}^2}
}
- 
\frac{\varphi\pare{\sigma}}{2\pare{1+h_\varepsilon\pare{\sigma }}^2} \bra{h'\pare{\sigma} - h'\pare{s}} 
} \dd s \dd \sigma \ \dd t . 
\end{align*} 
From the above integral equality, using H\"older's inequality, we obtain that
\begin{equation*}
\av{\cZ_{\varepsilon, 1}} \leq C \int_0^T \av{\Lambda^{1/2}\pare{h_\varepsilon - h}}_{L^2} \set{\av{\Lambda^{1/2}h'}_{L^2} \av{\frac{\varphi}{\pare{1+h_\varepsilon}^2}}_{L^\infty} + \av{h'}_{L^\infty}  \av{\Lambda^{1/2}\pare{ \frac{\varphi}{\pare{1+h_\varepsilon}^2} }}_{L^2}
} \dd t.
\end{equation*}
Hence standard computations show that the convergence proved in \eqref{eq:weak_convergence} is sufficient to establish that
\begin{equation*}
\cZ_{\varepsilon, 1}\xrightarrow{\varepsilon\to 0}0.
\end{equation*}

We compute
\begin{align*}
\cZ_{\varepsilon, 2}&=\cZ_{\varepsilon, 2,\RN{1}}+\cZ_{\varepsilon, 2,\RN{2}}  \\
\end{align*}
where
\begin{align*}
\cZ_{\varepsilon, 2,\RN{1}}& = \int_0^T \int \pare{ \int \frac{\theta   \pare{\partial_\alpha \theta_\varepsilon + \partial_\alpha \theta} \pare{\partial_\alpha \theta_\varepsilon - \partial_\alpha \theta}}{\pare{ 2 r_\varepsilon \sin\pare{\alpha / 2} }^2 }-\frac{\theta   \pare{\partial_\alpha \theta_\varepsilon + \partial_\alpha \theta} \pare{\partial_\alpha \theta_\varepsilon - \partial_\alpha \theta}}{\pare{ 2 r\sin\pare{\alpha / 2} }^2 }  \dd\alpha } \varphi\pare{s, t} \dd s \ \dd t ,\\
\cZ_{\varepsilon, 2,\RN{2}}& = \int_0^T \int \pare{ \int \frac{\theta   \pare{ \pare{\partial_\alpha \theta_\varepsilon}^2 - \pare{\partial_\alpha \theta}^2 }}{\pare{ 2 r \sin\pare{\alpha / 2} }^2 } \dd\alpha } \varphi\pare{s, t} \dd s \ \dd t .
\end{align*}
The term $\cZ_{\varepsilon, 2,\RN{1}}$ can be handled as $\cZ_{\varepsilon, 1}$. For the term $\cZ_{\varepsilon, 2,\RN{1}}$ we have to use a weak-strong convergence type argument. We decompose it as
\begin{align*}
\cZ_{\varepsilon, 2,\RN{2}}& = \cA_{\varepsilon, 1}+\cA_{\varepsilon, 2},
\end{align*}
with
\begin{align*}
\cA_{\varepsilon, 1}&=\int_0^T \int \pare{ \int \frac{\theta   \pare{\partial_\alpha \theta_\varepsilon-\partial_\alpha \theta}\partial_\alpha \theta_\varepsilon}{\pare{ 2 r \sin\pare{\alpha / 2} }^2 } \dd\alpha } \varphi\pare{s, t} \dd s \ \dd t \\
\cA_{\varepsilon, 2}&=\int_0^T \int \pare{ \int \frac{\theta \partial_\alpha \theta  \pare{\partial_\alpha \theta_\varepsilon - \partial_\alpha \theta}}{\pare{ 2 r \sin\pare{\alpha / 2} }^2 } \dd\alpha } \varphi\pare{s, t} \dd s \ \dd t.
\end{align*}
Let us focus first on the second term. Using
$$
\partial_\alpha \theta=-\partial_\alpha h(s-\alpha)=h'(s-\alpha),
$$
we find that
$$
-\partial_\alpha \theta=-h'(s-\alpha)+h'(s)-h'(s)=\theta'-h'(s).
$$
Using this we can equivalently write
\begin{align*}
\cA_{\varepsilon, 2}& = \int_0^T \int \pare{ \int \frac{\theta \partial_\alpha \theta  \pare{\partial_\alpha \theta_\varepsilon - \partial_\alpha \theta}}{\pare{ 2 r \sin\pare{\alpha / 2} }^2 } \dd\alpha } \varphi\pare{s, t} \dd s \ \dd t\\
&=\int_0^T \int \pare{ \int \frac{\theta \left(\theta'-h'(s)\right)  \pare{\theta'_\varepsilon-h'_\varepsilon(s) - \theta'+h'(s)}}{\pare{ 2 r \sin\pare{\alpha / 2} }^2 } \dd\alpha } \varphi\pare{s, t} \dd s \ \dd t.
\end{align*}
We can decompose it as
\begin{align*}
\cB_{\varepsilon, 1}&=\int_0^T \int  \int \frac{\theta \theta' \pare{\theta'_\varepsilon- \theta'}}{\pare{ 2 r \sin\pare{\alpha / 2} }^2 } \dd\alpha  \varphi\pare{s, t} \dd s \ \dd t\\
\cB_{\varepsilon, 2}&=\int_0^T \int  \int \frac{\theta \theta'  \pare{h'(s)-h'_\varepsilon(s)}}{\pare{ 2 r \sin\pare{\alpha / 2} }^2 } \dd\alpha \varphi\pare{s, t} \dd s \ \dd t\\
\cB_{\varepsilon, 3}&=-\int_0^T \int  \int \frac{\theta h'(s) \pare{\theta'_\varepsilon- \theta'}}{\pare{ 2 r \sin\pare{\alpha / 2} }^2 } \dd\alpha  \varphi\pare{s, t} \dd s \ \dd t\\
\cB_{\varepsilon, 4}&=-\int_0^T \int  \int \frac{\theta h'(s)  \pare{h'(s)-h'_\varepsilon(s)}}{\pare{ 2 r \sin\pare{\alpha / 2} }^2 } \dd\alpha \varphi\pare{s, t} \dd s \ \dd t.
\end{align*}
The term $\cB_{\varepsilon, 4}$ can be handled easily. We observe that it can be rewritten as
$$
\cB_{\varepsilon, 4}=C\int_0^T \int \frac{\Lambda h(s) h'(s)}{r}  \pare{h'(s)-h'_\varepsilon(s)}\varphi\pare{s, t} \dd s \ \dd t\leq C\|h-h_\varepsilon\|_{L^2_T H^1_x}\|h\|_{L^2_T H^1_x}\|h\|_{L^\infty_T W^{1,\infty}_x}\xrightarrow{\varepsilon\to 0}0.
$$
For the term $\cB_{\varepsilon, 2}$ we proceed as follows,
\begin{equation*}
\cB_{\varepsilon, 2}=\int_0^T \int  \int \frac{\theta \theta'  \pare{ h'(s)-h'_\varepsilon(s)}}{\pare{ 2 r \sin\pare{\alpha / 2} }^2 } \dd\alpha \varphi\pare{s, t} \dd s \ \dd t \\
\leq C\norm{h'}_{L^\infty_T L^\infty_x} \|\Lambda^{1/2}h'\|_{L^2_T L^2_x}\|h'-h'_\varepsilon\|_{L^2_T L^2_x} \xrightarrow{\varepsilon\to 0}0.
\end{equation*}
Taking $0<\delta\ll1$ we also compute
\begin{align*}
\cB_{\varepsilon, 3}&\leq C\|h'\|^2_{L^\infty_T L^\infty_x}\int_0^T \iint \frac{ \av{\theta'_\varepsilon- \theta'}}{ |\sin\pare{\alpha / 2}|^{1-\delta/2+\delta/2}  } \dd\alpha \dd s \ \dd t\\
&\leq C\|h'\|_{L^\infty_T L^\infty_x}^2   \left(\int \frac{1}{|\sin\pare{\alpha / 2}|^{\delta}}\dd \alpha\right)^{1/2} \int_0^T \left(\iint \frac{ \av{\theta'_\varepsilon- \theta'}^2}{ |\sin\pare{\alpha / 2}|^{2-\delta}  } \dd\alpha \dd s\right)^{1/2}  \ \dd t \\
&\leq C \sqrt{T} \|h'\|_{L^\infty_T L^\infty_x}^2 \norm{h'_\varepsilon - h'}_{L^2_T \dot{H}^{\frac{1-\delta}{2}}_x} \xrightarrow{\varepsilon\to 0}0.
\end{align*}
Finally, the last term $\cB_{\varepsilon, 1}$ can be handled as $\cB_{\varepsilon, 3}$. As a consequence, $h$ is a weak solution of \eqref{eq:eveqh2}.

In addition, the maximum principle
$$
\av{h(t)}_{W^{1,\infty}}\leq \av{h_0}_{W^{1,\infty}}\quad\forall\,0\leq t\leq T
$$
and the exponential decay
$$
\av{h'(t)}_{L^{\infty}}\leq \av{h'_0}_{L^{\infty}}e^{-\delta t}\quad\forall\,0\leq t\leq T
$$
follow from the application of Proposition \ref{prop:W1infty_enest} to the regularized problem and the weak-$*$ lower semicontinuity of the norm. We argue as in \cite[Lemma 4.3]{CCGS12} in order to state that, since $ h $ is a uniform limit of continuous functions we obtain as well that
\begin{equation*}
h\in \cC\pare{\bra{0, T}\times\bS^1}. 
\end{equation*}

  \hfill $ \Box $

\section*{Acknowledgments}
The research of F.G. has been partially supported by the grant MTM2017-89976-P (Spain) and by the ERC through
the Starting Grant project H2020-EU.1.1.-639227.

\noindent The research of R.G.B. is supported by the project ”Mathematical Analysis of Fluids and Applications” with reference PID2019-109348GA-I00/AEI/ 10.13039/501100011033 and acronym ``MAFyA” funded by Agencia Estatal de Investigaci\'on and the Ministerio de Ciencia, Innovacion y Universidades (MICIU).

\noindent The research of S.S. is supported by the ERC through
the Starting Grant project H2020-EU.1.1.-639227.

	\begin{footnotesize}

	\end{footnotesize}

\end{document}